\documentclass[reqno, 11pt]{amsart}
\usepackage{amsfonts, amsmath, amssymb, amsthm}
\usepackage[margin=2.75cm, heightrounded]{geometry}
\usepackage{hhline, float}
\usepackage{mathrsfs}
\usepackage{nccmath}
\usepackage[dvipsnames]{xcolor}
\usepackage{diagbox, makecell, tabularx}
\usepackage[colorlinks=true]{hyperref}
\hypersetup{citecolor=red, linkcolor=blue}

\newtheorem{thm}{Theorem}[section]
\newtheorem{lemma}[thm]{Lemma}
\newtheorem{prop}[thm]{Proposition}
\newtheorem{cor}[thm]{Corollary}

\newtheorem{thmx}{Theorem}

\newtheorem{assum}[thmx]{Assumption}
\newtheorem{rmkx}[thmx]{Remark}

\theoremstyle{definition}
\newtheorem{rmk}[thm]{Remark}
\newtheorem{defn}[thm]{Definition}

\newcommand{\ga}{\gamma}
\newcommand{\ep}{\epsilon}
\newcommand{\ka}{\kappa}
\newcommand{\vep}{\varepsilon}
\newcommand{\vph}{\varphi}
\newcommand{\vrh}{\varrho}
\newcommand{\vsi}{\varsigma}
\newcommand{\vth}{\vartheta}
\newcommand{\pa}{\partial}

\newcommand{\N}{\mathbb{N}}
\newcommand{\R}{\mathbb{R}}
\renewcommand{\S}{\mathbb{S}}

\newcommand{\mca}{\mathcal{A}}
\newcommand{\mcb}{\mathcal{B}}

\newcommand{\mcl}{\mathcal{L}}
\newcommand{\mcm}{\mathcal{M}}
\newcommand{\mco}{\mathcal{O}}
\newcommand{\mcq}{\mathcal{Q}}
\newcommand{\mcr}{\mathcal{R}}
\newcommand{\mcu}{\mathcal{U}}
\newcommand{\mcv}{\mathcal{V}}
\newcommand{\mcz}{\mathcal{Z}}

\newcommand{\mfa}{\mathfrak{a}}
\newcommand{\mfb}{\mathfrak{b}}
\newcommand{\mfc}{\mathfrak{c}}
\newcommand{\mfd}{\mathfrak{d}}
\newcommand{\mfp}{\mathfrak{p}}
\newcommand{\mfq}{\mathfrak{q}}

\newcommand{\mtc}{\mathtt{c}}

\newcommand{\msr}{\mathscr{R}}

\newcommand{\bmu}{\bar{\mu}}

\newcommand{\hrh}{\hat{\rho}}

\newcommand{\tc}{\tilde{c}}
\newcommand{\tih}{\tilde{h}}
\newcommand{\tw}{\tilde{w}}
\newcommand{\tde}{\tilde{\delta}}
\newcommand{\tmu}{\tilde{\mu}}
\newcommand{\trh}{\tilde{\rho}}
\newcommand{\txi}{\tilde{\xi}}
\newcommand{\tvrh}{\tilde{\varrho}}
\newcommand{\tvph}{\tilde{\varphi}}
\newcommand{\wtc}{\widetilde{C}}

\newcommand{\loc}{\textnormal{loc}}
\newcommand{\I}{\textnormal{I}}
\newcommand{\II}{\textnormal{II}}
\newcommand{\III}{\textnormal{III}}
\newcommand{\supp}{\textnormal{supp}}

\renewcommand{\(}{\left(}
\renewcommand{\)}{\right)}
\newcommand{\la}{\left\langle}
\newcommand{\ra}{\right\rangle}

\newcommand{\bs}[1]{\boldsymbol{#1}}

\allowdisplaybreaks
\numberwithin{equation}{section}

\makeatletter
\@namedef{subjclassname@2020}{%
\textup{2020} Mathematics Subject Classification}
\makeatother

\begin{document}
\title{Sharp quantitative stability of the Yamabe problem}

\author{Haixia Chen}
\address[Haixia Chen]{Department of Mathematics and Research Institute for Natural Sciences, College of Natural Sciences, Hanyang University, 222 Wangsimni-ro Seongdong-gu, Seoul 04763, Republic of Korea}
\email{hxchen29@hanyang.ac.kr chenhaixia157@gmail.com}

\author{Seunghyeok Kim}
\address[Seunghyeok Kim]{Department of Mathematics and Research Institute for Natural Sciences, College of Natural Sciences, Hanyang University, 222 Wangsimni-ro Seongdong-gu, Seoul 04763, Republic of Korea,
School of Mathematics, Korea Institute for Advanced Study, 85 Hoegiro Dongdaemun-gu, Seoul 02455, Republic of Korea.}
\email{shkim0401@hanyang.ac.kr shkim0401@gmail.com}

\begin{abstract}
Given a smooth closed Riemannian manifold $(M,g)$ of dimension $N \ge 3$, we derive sharp quantitative stability estimates for nonnegative functions near the solution set of the Yamabe problem on $(M,g)$.
The seminal work of Struwe (1984) \cite{S} states that if $\Gamma(u) := \|\Delta_g u - \frac{N-2}{4(N-1)} R_g u + u^{\frac{N+2}{N-2}}\|_{H^{-1}(M)} \to 0$, then $\|u-(u_0+\sum_{i=1}^{\nu} \mcv_i)\|_{H^1(M)} \to 0$
where $u_0$ is a solution to the Yamabe problem on $(M,g)$, $\nu \in \N \cup \{0\}$, and $\mcv_i$ is a bubble-like function.
If $M$ is the round sphere $\S^N$, then one may take $u_0 \equiv 0$ and a bubble itself serves as a natural candidate for $\mcv_i$.
If $M$ is not conformally equivalent to $\S^N$, then either $u_0 > 0$ or $u_0 \equiv 0$, there is no canonical choice of $\mcv_i$, and so a careful selection of $\mcv_i$ must be made to attain optimal estimates.

For $3 \le N \le 5$, we construct suitable $\mcv_i$'s and then establish the inequality $\|u-(u_0+\sum_{i=1}^{\nu} \mcv_i)\|_{H^1(M)}$ $ \le C\zeta(\Gamma(u))$
where $C > 0$ and $\zeta(t) = t$, consistent with the result of Figalli and Glaudo (2020) \cite{FG} on $\S^N$.
In the case of $N \ge 6$, we investigate the single-bubbling phenomenon $(\nu = 1)$ on generic Riemannian manifolds $(M,g)$,
proving that $\zeta(t)$ is determined by $N$, $u_0$, and $g$, and can be much larger than $t$.
This exhibits a striking difference from the result of Ciraolo, Figalli, and Maggi (2018) \cite{CFM} on $\S^N$. All of the estimates presented herein are optimal.
\end{abstract}

\date{\today}
\maketitle

\section{Introduction}
\subsection{Motivations}\label{subsec:mot}
Throughout the paper, we always assume that $(M,g)$ is a smooth closed Riemannian manifold of dimension $N \ge 3$.

\medskip
The Yamabe problem is one of the classical problems in geometric analysis, which asks the existence of a metric on $M$ with a constant scalar curvature in the conformal class $[g]$ of $g$.
This problem is equivalent to searching for a positive solution $u$ on $M$ to the Yamabe equation
\begin{equation}\label{u00}
-\Delta_gu + \ka_NR_gu = cu^{2^*-1}, \quad u \ge 0 \quad \text{on } M
\end{equation}
where $\ka_N := \frac{N-2}{4(N-1)}$, $2^* := \frac{2N}{N-2}$, $R_g$ is the scalar curvature on $(M,g)$, and $c \in \R$ is a constant.
The linear operator $\mcl_g := -\Delta_g+\ka_NR_g$ is called the conformal Laplacian on $(M,g)$.

\medskip
Since the existence of a positive least energy solution to \eqref{u00} was established through a series of works of Yamabe \cite{Ya}, Trudinger \cite{Tr}, Aubin \cite{Au}, and Schoen \cite{Sc} (see also Lee and Parker \cite{LP}), researchers have attempted to comprehend the whole solution structure of \eqref{u00}.
To describe it, let us define the Yamabe quotient $Q_{(M,g)}$ and the Yamabe invariant $Y(M,[g])$ of $(M,g)$ by
\[Q_{(M,g)}(u) = \frac{\ka_N^{-1} \int_M u \mcl_g u\, dv_g}{\(\int_M u^{2^*} dv_g\)^{\frac{N-2}{N}}}
= \frac{\int_M R_h dv_h}{\(\int_M dv_h\)^{\frac{N-2}{N}}} \quad \text{for } h = u^{\frac{4}{N-2}}g \in [g],\ 0 < u \in C^{\infty}(M)\]
where $dv_g$ is the volume form on $(M,g)$ and
\begin{equation}\label{eq:Yinv}
Y(M,[g]) = \inf\left\{Q_{(M,g)}(u): 0 < u \in C^{\infty}(M)\right\}.
\end{equation}
The constant $c$ in \eqref{u00} is positive (negative, respectively) if and only if $Y(M,[g])$ is positive (negative, resp.).
If $Y(M,[g]) < 0$, it is easy to see that \eqref{u00} has a unique solution. When $Y(M,[g]) = 0$, \eqref{u00} reduces to a linear equation and solutions are unique up to a constant multiple.
Therefore, the only case of significant interest is $Y(M,[g]) > 0$, in which case a number of high-energy solutions may exist.
For example, if $M = \S^1(r)\times \S^{N-1}$ where $\S^1(r)$ is the circle of radius $r > 0$ equipped with the standard metric and $\S^{N-1}$ is the $(N-1)$-dimensional round sphere of radius $1$,
then the number of inequivalent solutions is one if $r$ is small, is non-decreasing in $r$, and tends to $\infty$ as $r \to \infty$; refer to \cite{Sc2}.
Using gluing techniques, Pollack \cite{Po} showed that for any manifold $M$ with positive scalar curvature and $n \in \N$,
there is a dense set (in the $C^0$-topology) of the positive conformal classes for which \eqref{u00} has more than $n$ inequivalent solutions.

As a means to understand the entire solution set of \eqref{u00}, Schoen \cite{Sc3} asked whether the set is compact in $C^2(M)$ provided $M$ is not conformally equivalent to $\S^N$. He also suggested a general strategy to answer this question.
Based on his idea, Khuri, Marques, and Schoen \cite{KMS} proved that the $C^2(M)$-compactness holds for $N \le 24$ under the validity of the positive mass theorem. See also Druet \cite{Dr} and Li and Zhang \cite{LZ, LZ2} for the preceding results.
Surprisingly, counterexamples exist for $N \ge 25$ as shown by Brendle \cite{Br2} and Brendle and Marques \cite{BM}, which illustrates a deep and mysterious behavior of the solution set of \eqref{u00}.

Researchers also studied whether the compactness is preserved under a perturbation of equation \eqref{u00}. Because the literature on this topic is so vast, we mention a few initial results only.
The first result in this direction was achieved by Druet \cite{Dr2}.
By applying ``the $C^0$-theory for blow-up" developed in \cite{DHR}, he deduced the $C^2(M)$-compactness of positive solutions $\{u_{\ep}\}_{\ep \in \R}$ (for $\ep \in \R$ small) to critical equations with $3 \le N \le 5$,
\begin{equation}\label{eq:pert}
-\Delta_gu_{\ep} + h_{\ep}u_{\ep} = u_{\ep}^{2^*-1} \quad \text{on } M \quad \text{where } h_{\ep} \to h_0 \text{ in } C^2(M) \text{ as } \ep \to 0
\end{equation}
under a certain pointwise condition on the function $h_0-\ka_NR_g$ on $M$.
On the other hand, by perturbing \eqref{u00} suitably (e.g. letting $h_{\ep} = \ka_NR_g+\ep$ in \eqref{eq:pert} for $\ep \in \R$ small),
one can make the perturbed equation admit one of the following types of blowing-up solutions; solutions with single or multiple blowing-up points \cite{EPV, RV3}, the bubble clusters \cite{PV, RV2}, and the bubble towers \cite{MPV, Pr}.
In view of Struwe's global compactness result \cite{S} depicted in Theorem \ref{thm:Struwe} below and Schoen's strategy in \cite{Sc3}, these solutions represent essentially all the possible blow-up scenarios.

\medskip
There is yet another approach to studying the solution structure of \eqref{u00}, which is closely related to the aforementioned ones and the main topic of this paper.
We will derive a quantitative version of Theorem \ref{thm:Struwe} for general smooth closed Riemannian manifolds $(M,g)$.

For the moment, we assume that $M = \S^N$, in which case the quantitative analysis of Theorem \ref{thm:Struwe} was completed in the recent works \cite{CFM, FG, DSW}.
The inverse stereographic projection, a conformal map from $\R^N$ to $\S^N$, allows us to work on the Euclidean space $\R^N$ instead.
Struwe \cite{S} proved that if $u$ is a nonnegative element of the homogeneous Sobolev space $\dot{H}^1(\R^N)$,
then $u$ tends to a finite sum of weakly interacting bubbles in the $\dot{H}^1(\R^N)$-sense as $\Gamma(u) := \|\Delta u+u^{2^*-1}\|_{\dot{H}^{-1}(\R^N)} \to 0$. Here, a bubble refers to a function of the form
\begin{equation}\label{eq:bubble}
U_{\delta,\sigma}(y) = \alpha_N \(\frac{\delta}{\delta^2+|y-\sigma|^2}\)^{\frac{N-2}{2}} \quad \text{for } y \in \R^N, \quad \alpha_N := (N(N-2))^{\frac{N-2}{4}}
\end{equation}
where $\delta > 0$ and $\sigma \in \R^N$. It is well-known that the set of all positive solutions to the Yamabe equation in $\R^N$
\[-\Delta u = u^{2^*-1}, \quad u \ge 0 \quad \text{in } \R^N\]
is precisely the set $\{U_{\delta,\sigma}: \delta > 0,\, \sigma \in \R^N\}$ and any nonzero constant multiple of $U_{\delta,\sigma}$ attains the sharp Sobolev constant
\begin{equation}\label{eq:S}
S := \inf\left\{\frac{\|u\|_{\dot{H}^1(\R^N)}}{\|u\|_{L^{2^*}(\R^N)}}: u \in \dot{H}^1(\R^N) \setminus \{0\}\right\} \quad \text{where } \|u\|_{\dot{H}^1(\R^N)} := \|\nabla u\|_{L^2(\R^N)}.
\end{equation}
In \cite{CFM}, Ciraolo, Figalli, and Maggi derived the first sharp quantitative estimate of Struwe's result for $N \ge 3$:
For a nonnegative function $u$ in $\dot{H}^1(\R^N)$ with $\frac{1}{2}S^N \le \|u\|_{\dot{H}^1(\R^N)}^2 \le \frac{3}{2}S^N$ and sufficiently small $\Gamma(u)$,
it holds that $\|u-U_1\|_{\dot{H}^1(\R^N)} \le C\Gamma(u)$ for some bubble $U_1$.
If $\|u\|_{\dot{H}^1(\R^N)}^2 \ge \frac{3}{2}S^N$, the number of the bubbles associated to $u$ is at least two and delicate interactions between different bubbles occur, resulting in astonishing dimensional dependent estimates:
Suppose that $2 \le \nu \in \N$ and $u$ is a nonnegative element in $\dot{H}^1(\R^N)$ with $(\nu-\frac{1}{2})S^N \le \|u\|_{\dot{H}^1(\R^N)}^2 \le (\nu+\frac{1}{2})S^N$ and sufficiently small $\Gamma(u)$.
Then there is a constant $C > 0$ depending only on $N$ and $\nu$ such that
\[\bigg\|u-\sum_{i=1}^{\nu} U_i\bigg\|_{\dot{H}^1(\R^N)} \le C\begin{cases}
\Gamma(u) &\text{if } 3 \le N \le 5 \text{ (by Figalli and Glaudo \cite{FG})},\\
\Gamma(u)|\log \Gamma(u)|^{\frac{1}{2}} &\text{if } N=6 \text{ (by Deng, Sun, and Wei \cite{DSW})},\\
\Gamma(u)^{\frac{N+2}{2(N-2)}} &\text{if } N \ge 7 \text{ (by Deng, Sun, and Wei \cite{DSW})}
\end{cases}\]
for some bubbles $U_1,\ldots,U_{\nu}$. Also, this inequality is optimal.

In this paper, we carry out the above type of analysis on smooth closed Riemannian manifolds $(M,g)$ that are not conformally equivalent to $\S^N$.
We examine when $3 \le N \le 5$ and an arbitrary number of bubbles may develop, or $N \ge 6$ and only single bubble develops. For $M = \S^N$, our study corresponds to that of \cite{CFM, FG}.
As we will discuss further in the rest of the introduction, our general setting requires a variety of new perspectives, ideas, and techniques.
One of our notable discoveries is that the sharp quantitative estimate depends on $N$ even for the single-bubbling case.
In fact, it also relies on the metric $g$ and a solution $u_0$ to \eqref{u00}, which makes the problem quite intricate.

\subsection{Global compactness result}
Let us remind the global compactness result of Struwe \cite{S} combined with the interaction estimate of Bahri and Coron \cite{BC}.
Although the original statement is formulated for a smooth bounded domain in $\R^N$, it readily extends to any smooth closed Riemannian manifold $(M,g)$; refer to \cite{Br,H,DT}.

\medskip
Let $r_0 > 0$ be a sufficiently small number, particularly much smaller than the injectivity radius of $(M,g)$, and $\chi \in C^{\infty}_c([0,\infty))$ a cut-off function such that
\begin{equation}\label{chi0}
0 \le \chi \le 1 \quad \text{on } [0,\infty), \quad \chi = 1 \quad \text{on } [0,\tfrac{r_0}{2}], \quad \text{and} \quad \chi = 0 \quad \text{on } [r_0,\infty).
\end{equation}
Given $(\delta,\xi) \in (0,\infty) \times M$, we define
\begin{equation}\label{udx}
\mcu_{\delta,\xi}(x) = \mcu_{\delta,\xi}^g(x) = U_{\delta,0}\big(d_g(x,\xi)\big) \quad \text{for } x \in M
\end{equation}
where $d_g(x,\xi)$ is the geodesic distance between $x$ and $\xi$ on $(M,g)$ and we abused the notation by writing $U_{\delta,0}(y) = U_{\delta,0}(|y|)$. Then we have the following result.
\begin{thmx}\label{thm:Struwe}
Assume that $(M,g)$ is a smooth closed Riemannian manifold of dimension $N \ge 3$ with positive Yamabe invariant so that \eqref{u00} with $c=1$ has a positive solution.
Let $\ka_N = \frac{N-2}{4(N-1)}$, $2^* = \frac{2N}{N-2}$, $\mcl_g = -\Delta_g+\ka_NR_g$ be the conformal Laplacian on $(M,g)$, $H^1(M)$ the Sobolev space endowed with the norm
\begin{equation}\label{eq:norm}
\|u\|_{H^1(M)} := \left[\int_M \(|\nabla_g u|_g^2+\ka_NR_gu^2\) dv_g\right]^{\frac{1}{2}},
\end{equation}
and $H^{-1}(M)$ its dual.

Let $\{u_n\}_{n \in \N}$ be a sequence of nonnegative functions in $H^1(M)$ such that
\[\|u_n\|_{H^1(M)} \le C_0 \quad \text{and} \quad \left\|\mcl_g u_n-u_n^{2^*-1}\right\|_{H^{-1}(M)} \to 0 \text{ as } n \to \infty\]
for some constant $C_0 > 0$. After passing to a subsequence if necessary, one can find a nonnegative function $u_0 \in C^{\infty}(M)$, a number $\nu \in \N \cup \{0\}$ satisfying $\nu \le C_0^2S^{-N}$,
and a sequence $\{(\delta_{1n},\ldots,\delta_{\nu n}, \xi_{1n},\ldots,\xi_{\nu n})\}_{n \in \N} \subset (0,\infty)^{\nu} \times M^{\nu}$ such that the followings hold:
\begin{itemize}
\item[-] $u_0$ is a smooth solution to the Yamabe equation \eqref{u00} with $c=1$. By the strong maximum principle, we have either $u_0 > 0$ or $u_0 = 0$ on $M$.
\item[-] For all $1 \le i \ne j \le \nu$, we have that $\delta_{in} \to 0$ and
\begin{equation}\label{eq:inter}
\frac{\delta_{in}}{\delta_{jn}} + \frac{\delta_{jn}}{\delta_{in}} + \frac{d_g(\xi_{in},\xi_{jn})^2}{\delta_{in}\delta_{jn}}\to \infty \quad \text{as } n \to \infty.
\end{equation}
\item[-] It holds that
\begin{equation}\label{eq:inter2}
\bigg\|u_n-\bigg(u_0+\sum_{i=1}^{\nu}\mcv_{in}\bigg)\bigg\|_{H^1(M)} \to 0 \quad \text{as } n \to \infty
\end{equation}
where each $\mcv_{in} := \mcv_{\delta_{in},\xi_{in}}$ is a bubble-like function on $M$. Throughout the paper, a bubble-like function refers to a function whose asymptotic profile gets closer to a truncated bubble $\chi(d_g(\cdot,\xi_{in})) \mcu_{\delta_{in},\xi_{in}}$ in $H^1(M)$ as $\delta_{in} \to 0$. In other words,
\begin{equation}\label{eq:bubblel}
\left\|\mcv_{in}-\chi(d_g(\cdot,\xi_{in}))\mcu_{\delta_{in},\xi_{in}}\right\|_{H^1(M)} \to 0 \quad \text{as } n \to \infty \quad \text{for } i = 1,\ldots,\nu.
\end{equation}
\end{itemize}	
\end{thmx}
\noindent The interaction estimate \eqref{eq:inter}, traced back to Bahri and Coron \cite[(5)]{BC}, implies that each bubble-like function $\mcv_{in}$ is less likely to interact with the other bubbles at the $H^1(M)$-level as $n \to \infty$.
A combination of \eqref{eq:inter}, \eqref{eq:inter2}, and \eqref{eq:S} yields
\[\|u_n\|_{H^1(M)}^2 = \|u_0\|_{H^1(M)}^2 + \sum_{i=1}^{\nu} \|U_{\delta_{in},0}\|_{\dot{H}^1(\R^N)}^2 + o(1) = \|u_0\|_{H^1(M)}^2 + \nu S^N + o(1)\]
where $o(1) \to 0$ as $n \to \infty$. This forces the bound $\nu \le C_0^2S^{-N}$.

On the other hand, if $M = \R^N \cup \{\infty\}$ (the one-point compactification of the Euclidean space), then one may take $u_0 \equiv 0$ and a bubble \eqref{eq:bubble} itself serves as a natural candidate for $\mcv_{in}$.
In contrast, if $M$ is not conformally equivalent to $\S^N$, then $u_0$ may be either positive or identically 0, and there is no a canonical choice of $\mcv_{in}$ in general.
Moreover, constructing $\mcv_{in}$ that accurately approximates $u_n$ is essential in achieving sharp quantitative estimates.
By recalling the resolution of the Yamabe problem \cite{Au, Sc, LP}, we will make use of bubbles, cut-off functions, conformal changes of a metric, and the Green's function of $\mcl_g$ to build $\mcv_{in}$'s; see \eqref{ui}, \eqref{vi}, and \eqref{w1}.

\subsection{Main results}
In this paper, we will work on the following setting.
\begin{assum}\label{assum}
Let $(M,g)$ be a smooth closed Riemannian manifold of dimension $N \ge 3$ that is not conformally equivalent to $\S^N$.
We assume that the Yamabe invariant $Y(M,[g])$ is positive so that the definition of the norm in \eqref{eq:norm} makes sense. Suppose that a nonnegative function $u$ in $H^1(M)$ satisfies
\begin{equation}\label{tui}
\bigg\|u-\bigg(u_0+\sum_{i=1}^{\nu} \chi\big(d_g(\cdot,\txi_i)\big)\mcu_{\tde_i,\txi_i}\bigg)\bigg\|_{H^1(M)} \le \vep_0
\end{equation}
for some small $\vep_0 > 0$ and $\nu \in \N$. Here, $u_0$ is a solution to \eqref{u00} with $c=1$, $\mcu_{\delta,\xi}$ is the function in \eqref{udx}, and $(\tde_i,\txi_i) \in (0,\infty) \times M$ satisfies that $\tde_i \le \vep_0$ and
\[\max\left\{\(\frac{\tde_i}{\tde_j} + \frac{\tde_j}{\tde_i} + \frac{d_g(\txi_i,\txi_j)^2}{\tde_i\tde_j}\)^{-\frac{N-2}{2}}: i, j=1,\ldots,\nu,\ i \ne j\right\} \le \vep_0.\]
Let also $\Gamma(u) = \|\mcl_g u-u^{2^*-1}\|_{H^{-1}(M)}$.
\end{assum}
\begin{rmkx}
In Theorem \ref{thm:Struwe}, the above situation happens when $u_n \nrightarrow u_0$ strongly in $H^1(M)$. In Section \ref{se6}, we also treat the simplest case $\nu = 0$.
\end{rmkx}

We now list our main results. First of all, we are concerned with $3 \le N \le 5$. Specifically, we address the case $u_0 > 0$ in Theorem \ref{th1.2} and the situation $u_0 = 0$ in Theorem \ref{th1.3}.
\begin{thm}\label{th1.2}
Suppose that $3 \le N \le 5$ and Assumption \ref{assum} holds with $u_0 > 0$ on $M$. We also assume that $u_0$ is non-degenerate, meaning that the kernel of the operator $\mcl_g-(2^*-1)u_0^{2^*-2}$ on $H^1(M)$ is trivial.
Given $(\delta,\xi) \in (0,\infty) \times M$, we set a nonnegative function $\mcv_{\delta,\xi}$ on $M$ by
\begin{equation}\label{ui}
\mcv_{\delta,\xi}(x) = \chi(d_g(x,\xi))\mcu_{\delta,\xi}(x) + (1-\chi(d_g(x,\xi))) U_{\delta,0}\(\tfrac{r_0}{2}\) \quad \text{for } x \in M.
\end{equation}
Here, $U_{\delta,0}$ is a bubble in \eqref{eq:bubble}, $\chi$ is a cut-off function satisfying \eqref{chi0}, and $r_0 > 0$ is a small number.
After reducing the size of $\vep_0 > 0$ if needed, one can find $\nu$ functions $\mcv_1 := \mcv_{\delta_1,\xi_1}, \ldots, \mcv_{\nu} := \mcv_{\delta_{\nu},\xi_{\nu}}$ such that
\begin{equation}\label{eq:sqe1}
\bigg\|u-\bigg(u_0+\sum_{i=1}^{\nu} \mcv_i\bigg)\bigg\|_{H^1(M)} \le C\Gamma(u).
\end{equation}
Here, $C > 0$ is a large constant depending only on $N$, $\nu$, $u_0$, and $(M,g)$.
\end{thm}

In Theorems \ref{th1.3} and \ref{th1.4}, we exploit the notion of conformal normal coordinates introduced by Lee and Parker \cite{LP} to devise $\mcv_i$'s:
Given any $\theta \in \N$ and $\xi \in M$, there exists a smooth positive function $\Lambda_{\xi}$ on $M$ such that $\Lambda_{\xi}(\xi)=1$, $\nabla_g\Lambda_{\xi}(\xi)=0$, and the conformal metric $g_{\xi} := \Lambda_{\xi}^{4/(N-2)}g$ satisfies
\begin{equation}\label{dvg}
\det g_{\xi}(y) = 1 + \mco\big(|y|^{\theta}\big)
\end{equation}
in $g_{\xi}$-normal coordinates $y$ around $\xi$. For our purpose, we pick $\theta$ large enough.
According to Cao \cite{C} and G\"unther \cite{G}, \eqref{dvg} can be improved to $\det g_{\xi}(y) = 1$.
\begin{thm}\label{th1.3}
Suppose that $3 \le N \le 5$ and Assumption \ref{assum} holds with $u_0 = 0$ on $M$. Given $(\delta,\xi) \in (0,\infty) \times M$, we set a nonnegative function $\mcv_{\delta,\xi}$ on $M$ by
\begin{equation}\label{vi}
\mcv_{\delta,\xi}(x) = \ga_NG_g(x,\xi) \left[\chi(d_{g_{\xi}}(x,\xi))d_{g_{\xi}}(x,\xi)^{N-2}\mcu_{\delta,\xi}^{g_{\xi}}(x) + (1-\chi(d_{g_{\xi}}(x,\xi))) \alpha_N \delta^{\frac{N-2}{2}}\right]
\end{equation}
for $x \in M$. Here, $\alpha_N > 0$ is the number in \eqref{eq:bubble}, $\ga_N := (N-2)|\S^{N-1}|$, $|\S^{N-1}|$ is the surface measure of the sphere $\S^{N-1}$, and $G_g$ is the Green's function of the conformal Laplacian $\mcl_g$.
After reducing the size of $\vep_0 > 0$ if needed, one can find $\nu$ functions $\mcv_1 := \mcv_{\delta_1,\xi_1}, \ldots, \mcv_{\nu} := \mcv_{\delta_{\nu},\xi_{\nu}}$ such that
\begin{equation}\label{eq:sqe2}
\bigg\|u-\sum_{i=1}^{\nu} \mcv_i\bigg\|_{H^1(M)} \le C\Gamma(u).
\end{equation}
Here, $C > 0$ is a large constant depending only on $N$, $\nu$, and $(M,g)$.
\end{thm}
Next, we handle the case when $N \ge 6$ and only a single-bubbling is permitted. Interestingly, it turns out that the quantitative estimate depends on $N$, $u_0$, and $g$. This is a new phenomenon.
Here and after, l.c.f. stands for locally conformally flat.
\begin{thm}\label{th1.4}
Suppose that $N \ge 6$ and Assumption \ref{assum} holds with $\nu = 1$. We also assume that
\begin{itemize}
\item[-] if $u_0 > 0$ on $M$, then $u_0$ is non-degenerate;
\item[-] in the case that $(M,g)$ is non-l.c.f., if either \textup{[}$N \ge 11$ and $u_0 > 0$\textup{]} or \textup{[}$N \ge 6$ and $u_0 = 0$\textup{]}, then the Weyl curvature tensor $\textup{Weyl}_g(\txi_1)$ at $\txi_1 \in M$ is nonzero.
\end{itemize}
Given $(\delta,\xi) \in (0,\infty) \times M$, we set a nonnegative function $\mcv_{\delta,\xi}$ on $M$ by
\begin{equation}\label{w1}
\begin{medsize}
\mcv_{\delta,\xi}(x) = \begin{cases}
\begin{aligned}
&\ga_NG_g(x,\xi) \left[\chi(d_{g_{\xi}}(x,\xi))d_{g_{\xi}}(x,\xi)^{N-2}\mcu_{\delta,\xi}^{g_{\xi}}(x) \right. \\
&\hspace{55pt} \left. + (1-\chi(d_{g_{\xi}}(x,\xi))) \alpha_N\delta^{\frac{N-2}{2}}\right]
\end{aligned}
&\text{if } \left[\begin{aligned}
& N \ge 6,\, (M,g) \text{ is l.c.f., or} \\
& u_0=0,\, 6 \le N \le 10,\, (M,g)\text{ is non-l.c.f.}
\end{aligned}\right], \\
\Lambda_{\xi}(x) \chi(d_{g_{\xi}}(x,\xi))\mcu_{\delta,\xi}^{g_{\xi}}(x)
&\text{if } \left[\begin{aligned}
& u_0>0,\, N \ge 6,\, (M,g) \text{ is non-l.c.f., or}\\
& u_0=0,\, N \ge 11,\, (M,g) \text{ is non-l.c.f.}
\end{aligned}\right]
\end{cases}
\end{medsize}
\end{equation}
for $x \in M$. After reducing the size of $\vep_0 > 0$ if needed, one can find a function $\mcv_1 := \mcv_{\delta_1,\xi_1}$ and a large constant $C > 0$ that depends only on $N$, $u_0$, and $(M,g)$ such that the following inequalities hold:

\medskip \noindent \textup{(1)} In case that $u_0 > 0$ on $M$, we have
\begin{equation}\label{eq:sqe3}
\|u-(u_0+\mcv_1)\|_{H^1(M)} \le C\zeta(\Gamma(u))
\end{equation}
where $\zeta \in C^0([0,\infty))$ satisfies
\begin{equation}\label{eq:zeta1}
\zeta(t) = \begin{cases}
t|\log t|^{\frac{1}{2}} &\text{if } N=6,\\
t^{\frac{N+2}{2(N-2)}} &\text{if } 7 \le N \le 10 \text{ or } [N \ge 11 \text{ and } (M,g) \text{ is l.c.f.}],\\
t^{\frac{N+2}{16}} &\text{if } 11 \le N \le 13 \text{ and } (M,g) \text{ is non-l.c.f.},\\
t &\text{if } N \ge 14 \text{ and } (M,g) \text{ is non-l.c.f.}
\end{cases}
\end{equation}
for $t > 0$.

\medskip \noindent \textup{(2)} In case that $u_0 = 0$ on $M$, we have
\begin{equation}\label{eq:sqe4}
\|u-\mcv_1\|_{H^1(M)} \le C\zeta(\Gamma(u))
\end{equation}
where $\zeta \in C^0([0,\infty))$ satisfies
\begin{equation}\label{eq:zeta2}
\zeta(t) = \begin{cases}
t|\log t|^{\frac12} &\text{if } N=6,\\
t^{\frac{N+2}{2(N-2)}} &\text{if } N \ge 7 \text{ and } (M,g) \text{ is l.c.f.},\\
t &\text{if } N \ge 7 \text{ and } (M,g) \text{ is non-l.c.f.}
\end{cases}
\end{equation}
for $t > 0$.
\end{thm}

\begin{rmk}
We present several remarks on Theorems \ref{th1.2}--\ref{th1.4}.

\medskip \noindent (1) The non-degeneracy assumption for $u_0$ is generic.
By \cite[Theorem 10.3]{KMS}, one can perturb the metric $g$ on $M$ slightly so that every positive solution to \eqref{u00} with the new metric is non-degenerate, provided $3 \le N \le 24$ and the positive mass theorem is valid.

In contrast, there are concrete examples for which $u_0$ is non-degenerate: Let $M = \S^1(r) \times \S^{N-1}$ be a manifold that appeared in Subsection \ref{subsec:mot}.
According to \cite[Proposition 3.4]{RV}, the constant solution $u_0 = (\frac{N-2}{2})^{(N-2)/2}$ to \eqref{u00} with $c=1$ is non-degenerate for all $r \in (0,\infty) \setminus \{l/\sqrt{N-2}: l \in \N\}$.

\medskip \noindent (2) In our proof, we crucially use the positive mass theorem when $3 \le N \le 5$ or [$N \ge 6$ and $(M,g)$ is l.c.f.]; refer to Lemmas \ref{le36} and \ref{s44}, and Proposition \ref{pr47}.
Their validity was proved by Schoen and Yau \cite{SY1, SY2}.

\medskip \noindent (3) As a matter of fact, the choice of $\mcv_{\delta,\xi}$ in \eqref{vi} is applicable to all cases in Theorems \ref{th1.2}--\ref{th1.4}.
This $\mcv_{\delta,\xi}$ is qualitatively similar to the test functions of Schoen in \cite[Section 1]{Sc}, of Brendle in \cite[(203)]{Br}, and of Esposito, Pistoia, and V\'etois in \cite[(2.7)--(2.8)]{EPV}.

However, we decided to select simpler test functions \eqref{ui} in Theorem \ref{th1.2} and \eqref{w1} in Theorem \ref{th1.4}, respectively,
to manifest which factors determine the right-hand side of the quantitative estimates \eqref{eq:sqe1}, \eqref{eq:sqe2}, \eqref{eq:sqe3}, and \eqref{eq:sqe4}; see Subsection \ref{subsec:novel}(3) for more discussion.

\medskip \noindent (4) We opted to work only with nonnegative $u$ for Theorems \ref{th1.2}, \ref{th1.3}, and \ref{th1.4}, where as the authors in \cite{FG, DSW} permitted $u$ to assume both positive and negative values.
Our choice reflects the geometric significance of positive solutions to the Yamabe equation \eqref{u00}, and forces the $H^1(M)$-weak limit $u_0$ of $u$ as $\Gamma(u) \to 0$ to be either positive or $0$ on $M$.
\end{rmk}
\begin{rmk}
We provide comments regarding the cases that are untouched in this paper.

\medskip \noindent (1) Deducing the sharp quantitative estimate for the multiple bubble case with $N \ge 6$ is more difficult, because we must take into account of the effect of a solution $u_0$ to \eqref{u00},
the metric $g$, and the mutual interaction between bubble-like functions $\mcv_1,\ldots,\mcv_{\nu}$ simultaneously.
In view of \cite{DSW}, we may also need a priori bubble-tree analysis, which is extremely complicated.
We expect that the $C^0$-theory of Druet, Hebey, and Robert \cite{DHR} and Premoselli \cite{Pr2} will be helpful.

\medskip \noindent (2) We may attempt to remove the non-degeneracy assumption on positive $u_0$ by perturbing it in a suitable sense as in Brendle \cite{Br}, Robert and Vetois \cite{RV3}, or Premoselli \cite{Pr2}.

\medskip \noindent (3) In Theorem \ref{th1.4}, a generic condition $\textup{Weyl}_g(\txi_1) \ne 0$ was imposed for some non-l.c.f. manifolds $(M,g)$ to avoid additional technical issues.
If $\textup{Weyl}_g(\txi_1) = 0$ for those manifolds, one must consider the vanishing rate of the Weyl tensor near $\txi_1$ to seek the optimal function $\zeta$.
\end{rmk}

The following theorem demonstrates the optimality of Theorems \ref{th1.2}, \ref{th1.3}, and \ref{th1.4}.
\begin{thm}\label{th1.6}
Let $\zeta$ be a continuous function on $[0,\infty)$ given by $\zeta(t) = t$ in the setting of Theorems \ref{th1.2} and \ref{th1.3}, and by \eqref{eq:zeta1} and \eqref{eq:zeta2} in the setting of Theorems \ref{th1.4}.
Estimates \eqref{eq:sqe1} in Theorem \ref{th1.2}, \eqref{eq:sqe2} in Theorem \ref{th1.3}, and \eqref{eq:sqe3} and \eqref{eq:sqe4} in Theorem \ref{th1.4} are all sharp in the following sense:
Given any $\vep_0 > 0$, there exists nonnegative $u_* \in H^1(M)$ satisfying \eqref{tui} such that
\[\inf\left\{\bigg\|u_*-\bigg(u_0+\sum_{i=1}^{\nu} \mcv_{\delta_i,\xi_i}\bigg)\bigg\|_{H^1(M)}: (\delta_i,\xi_i) \in (0,\infty) \times M \text{ for } i = 1,\ldots,\nu\right\} \ge C\zeta(\Gamma(u_*))\]
where $C > 0$ depends only on $N$, $\nu$, $u_0$, and $(M,g)$.
\end{thm}

Finally, by combining Theorem \ref{thm:Struwe} and Theorems \ref{th1.2}, \ref{th1.3}, and \ref{th1.4}, we obtain
\begin{cor}\label{cor:main}
Let $S > 0$ be the sharp Sobolev constant in \eqref{eq:S} and $\nu_0 \in \N \cup \{0\}$. We assume that every positive solution to \eqref{u00} with $c=1$ is non-degenerate.

\medskip \noindent \textup{(1)} Assume that $3 \le N \le 5$ and $\nu_0 \in \N \cup \{0\}$. If $u$ is a nonnegative function in $H^1(M)$ with $\|u\|_{H^1(M)}^2 \le (\nu_0+\frac{1}{2})S^N$, then there exists a constant $C > 0$ depending only on $N$, $\nu_0$, and $(M,g)$ such that
\begin{equation}\label{eq:cor}
\begin{medsize}
\displaystyle \inf\left\{\bigg\|u-\bigg(u_0+\sum_{i=1}^{\nu} \mcv_{\delta_i,\xi_i}\bigg)\bigg\|_{H^1(M)}: u_0 \text{ solves } \eqref{u00} \text{ with } c=1,\, \mcv_{\delta_i,\xi_i} \in \mcb,\, \nu = 0,\ldots,\nu_0\right\} \le C\zeta(\Gamma(u))
\end{medsize}
\end{equation}
where $\zeta(t) = t$ for $t \in [0,\infty)$ and
\begin{multline*}
\mcb := \{\mcv_{\delta,\xi}: \mcv_{\delta,\xi} \text{ is a bubble-like function defined by \eqref{ui} if } u_0 > 0 \\ \text{ and \eqref{vi} if } u_0 = 0,\, (\delta,\xi) \in (0,\infty) \times M\}.
\end{multline*}
We obey the convention $\sum_{i=1}^0 \mcv_{\delta_i,\xi_i} = 0$.

\medskip \noindent \textup{(2)} Assume that $N \ge 6$ and $(M,g)$ is l.c.f. If $u$ is a nonnegative function in $H^1(M)$ with $\|u\|_{H^1(M)}^2 \le \frac{3}{2}S^N$, then there exists a constant $C > 0$ depending only on $N$ and $(M,g)$ such that \eqref{eq:cor} with $\nu_0 = 1$ holds where $\zeta \in C^0([0,\infty))$ satisfies
\[\zeta(t) = \begin{cases}
t|\log t|^{\frac12} &\text{if } N=6,\\
t^{\frac{N+2}{2(N-2)}} &\text{if } N \ge 7\\
\end{cases}\]
for $t > 0$ and
\begin{equation}\label{eq:cor2}
\mcb := \{\mcv_{\delta,\xi}: \mcv_{\delta,\xi} \text{ is a bubble-like function defined by \eqref{w1}},\, (\delta,\xi) \in (0,\infty) \times M\}.
\end{equation}

\medskip \noindent \textup{(3)} Assume that $N \ge 6$ and the Weyl tensor on $(M,g)$ never vanishes.
If $u$ is a nonnegative function in $H^1(M)$ with $\|u\|_{H^1(M)}^2 \le \frac{3}{2}S^N$, then there exists a constant $C > 0$ depending only on $N$ and $(M,g)$ such that \eqref{eq:cor} with $\nu_0 = 1$ holds where $\zeta \in C^0([0,\infty))$ satisfies
\[\zeta(t) = \begin{cases}
t|\log t|^{\frac12} &\text{if } N=6,\\
t^{\frac{N+2}{2(N-2)}} &\text{if } 7 \le N \le 10,\\
t^{\frac{N+2}{16}} &\text{if } 11 \le N \le 13,\\
t &\text{if } N \ge 14
\end{cases}\]
and $\mcb$ is defined by \eqref{eq:cor2}.
\end{cor}

\subsection{Related results}
Quantitative stability for sharp functional inequalities is a classical subject that have attracted to researchers for decades.
In a seminal work \cite{BL}, Brezis and Lieb raised a question of quantitative stability for extermizers of the Sobolev embedding $\dot{H}^1(\R^N) \hookrightarrow L^{2^*}(\R^N)$. Bianchi and Egnell \cite{BE} answered it by deriving
\begin{equation}\label{eq:BE1}
\|u\|_{\dot{H}^1(\R^N)}^2 - S^2\|u\|_{L^{2^*}(\R^N)}^2 \ge C_{\text{BE}} \inf\left\{\left\|u - cU_{\delta,\sigma}\right\|_{\dot{H}^1(\R^N)}^2: \delta > 0,\ \sigma \in \R^N,\ c \in \R \right\}
\end{equation}
for any $u \in H^1(\S^N)$ and some $C_{\text{BE}} > 0$ determined by $N$; refer to a recent paper of Dolbeault, Esteban, Figalli, Frank, and Loss \cite{DEFFL} where the authors proved the explicit stability constant estimate $C_{\text{BE}} \ge \frac{\beta}{N}$ for some $\beta > 0$ independent of $N$.
Let $g_0$ be the metric on the round sphere $\S^N$ and $\mcm_{(M,g)}$ the set of minimizers of \eqref{eq:Yinv} that attain the Yamabe invariant.
Owing to the conformal equivalence between the manifolds $\R^N \cup \{\infty\}$ and $\S^N$, inequality \eqref{eq:BE1} is rephrased as
\begin{equation}\label{eq:BE2}
Q_{(\S^N,g_0)}(u) - Y(\S^N,[g_0]) \ge \wtc_{\text{BE}}\, \frac{\inf\left\{\|u-v\|_{H^1(\S^N)}^2: v \in \mcm_{(\S^N,g_0)} \right\}}{\|u\|_{H^1(\S^N)}^2}
\end{equation}
for any $0 \le u \in H^1(\S^N)$ and some $\wtc_{\text{BE}} > 0$.
By utilizing the Łojasiewicz inequality, Engelstein, Neumayer, and Spolaor \cite{ENS} recently obtained a generalization of \eqref{eq:BE1}--\eqref{eq:BE2} that holds on any smooth closed Riemannian manifold $(M,g)$.
Their main result is that if $(M,g)$ is not conformally equivalent to $(\S^N,g_0)$, then there exists $\wtc_{\text{ENS}} > 0$ and $\ga \ge 0$ depending on $(M,g)$ such that
\begin{equation}\label{eq:ENS}
Q_{(M,g)}(u) - Y(M,[g]) \ge \wtc_{\text{ENS}}\, \frac{\inf\left\{\|u-v\|_{H^1(M)}^{2+\ga}: v \in \mcm_{(M,g)} \right\}}{\|u\|_{H^1(M)}^{2+\ga}}
\end{equation}
for any $0 \le u \in H^1(M)$. In addition, one can take $\ga = 0$ generically (in the sense made in \cite{ENS}), but $\ga = 2$ is optimal if $M = \S^1(\frac{1}{\sqrt{N-2}}) \times \S^{N-1}$ as shown by Frank \cite{Fr}.
In \cite{NV}, Nobili and Violo established a similar stability result on a wide class of Riemannian manifolds, which makes a direct comparison between almost extremal functions and bubbles.

\medskip
Inequalities \eqref{eq:BE2} and \eqref{eq:ENS} concern the stability of the variational problem \eqref{eq:Yinv} near its minimizers, or equivalently, that of equation \eqref{u00} near positive least energy solutions.
On the other hand, our quantitative estimates \eqref{eq:sqe1}, \eqref{eq:sqe2}, \eqref{eq:sqe3}, and \eqref{eq:sqe4} take into account the overall solution structure of \eqref{u00}, so their accompanying analysis is much more cumbersome.
The latter type of studies have been spotlighted after the works of \cite{CFM, FG, DSW} mentioned in Subsection \ref{subsec:mot}.
Analogous results were achieved for the Caffarelli-Kohn-Nirenberg inequalities \cite{WW}, the fractional Sobolev inequalities \cite{Ar, DK, CKW},
the half-harmonic maps \cite{DSW2}, the Poincar\'e-Sobolev inequalities \cite{BGKM}, the Hardy-Littlewood-Sobolev inequalities \cite{LZZ, PYZ}, among others.

\subsection{Novelty of the proof}\label{subsec:novel}
Our argument is influenced by Deng, Sun, and Wei \cite{DSW}, which essentially provides an alternative proof of \cite[Theorem 3.3]{FG} for $3 \le N \le 5$ as a by-product.
To work on arbitrary smooth compact Riemannian manifolds $(M,g)$, one has to develop several new technical novelties. We briefly explain the unique features of our proof.

\medskip \noindent (1) Unlike the case $M = \S^N$, one has to consider when $u_0$ is positive. The presence of such $u_0$ increases the complexity of the analysis as can be seen
in the derivation of a coercivity inequality (see e.g. Proposition \ref{41}) and evaluation of the interaction strength between $u_0$ and bubble-like functions $\mcv_i$ (see e.g. Lemmas \ref{le2.3}, \ref{le24}, and \ref{le2.4}).
Particularly, compared to \cite{DSW}, we also have to control an additional term $\max_{i=1,\ldots,\nu} \delta_i^{(N-2)/2}$ in Propositions \ref{pr2.2} and \ref{lou0}. This term is non-comparable to $\mcq$ in \eqref{rq} directly.

\medskip \noindent (2) In contrast to \cite{CFM}, the function $\zeta$ in the quantitative estimates \eqref{eq:sqe3} and \eqref{eq:sqe4} may be significantly larger than $t$ even for the single-bubbling case with $N \ge 6$.
This phenomenon happens due to the interplay of the bubbles in $\R^N$, a solution $u_0$ to \eqref{u00}, and the metric $g$
(more precisely, the Weyl tensor $\textup{Weyl}_g$ and the mass $A_g$ on $(M,g)$ evaluated at the centers of the bubbles; refer to \eqref{gx} and the following sentence for the definition of the mass).

The dimensions $N = 11$ and $14$ appear in \eqref{eq:zeta1} since they are the smallest dimensions such that $\frac{N-2}{2} > 4$ (see \eqref{eq:pr442}, \eqref{eq:s43}, and \eqref{n13m}) and $\frac{N+2}{4} \ge 4$ (see \eqref{4i32}, \eqref{123}, and \eqref{eq:rhoest}), respectively.

\medskip \noindent (3) The choice of the bubble-like functions $\mcv_i$ depends on the dimension $N$, geometric assumptions on $(M,g)$, and whether $u_0$ is positive or identically $0$ on $M$.

If $3 \le N \le 5$ and $u_0 > 0$, then $u_0$ is the most dominant factor for the quantitative estimate, enabling us to take a truncated bubble $\chi(d_g(\cdot,\xi_i))\mcu_{\delta_i,\xi_i}$ for $\mcv_i$ near the concentration point $\xi_i \in M$.
The term $(1-\chi(d_g(\cdot,\xi)))U_{\delta,0}(\tfrac{r_0}{2})$ is required to capture the interactions among different $\mcv_i$'s.

If $3 \le N \le 5$ and $u_0 = 0$, then we need more precise information of $\mcv_i$ than before, which we achieve by using conformal changes of the metric $g$ and the Green's function $G_g$ of the conformal Laplacian $\mcl_g$.

If $N \ge 6$ and $\nu = 1$, then the combined effects of the bubbles in $\R^N$, $u_0$, and $g$ determine which choice of $\mcv_1$ is the simplest.
To illustrate, suppose that $u_0 = 0$ and $(M,g)$ is non-l.c.f. In order to achieve \eqref{eq:sqe4}, one has to deduce an optimal estimate for the $L^{2N/(N+2)}(M)$-norm of the term $\III_3$ in \eqref{III23}.
If $N \ge 11$, one can obtain it by simply taking $\mcv_1 = \Lambda_{\xi_1} \chi(d_{g_{\xi_1}}(\cdot,\xi_1))\mcu_{\delta_1,\xi_1}^{g_{\xi_1}}$; see \eqref{123}.
If $6 \le N \le 10$, one needs improve it as in \eqref{w1}; see \eqref{124}.

\medskip \noindent (4) Unlike \cite{CFM}, we need pointwise estimates of $\rho := u-(u_0+\mcv_1)$ to deduce the optimal estimates for $N \ge 6$ and $\nu = 1$.
If $N=6$, the $L^{2N/(N+2)}(M)$-estimates of the error terms in the proof of Propositions \ref{pr41} and \ref{pr47} yield merely a rough estimate of powers of the $|\log \delta_1|$ terms in \eqref{l4u1} and \eqref{l4q}.
To obtain the optimal result (see Corollaries \ref{cr43} and \ref{cr46}), we appeal to pointwise estimates of $\rho$ (see Lemmas \ref{le4.2} and \ref{le4.8}).
We need pointwise estimates of $\rho$ even for $N \ge 7$; refer to \eqref{poi} and Subsections \ref{se5.2} and \ref{se5.3}.

\medskip \noindent (5) In the derivation of the coercivity inequalities in Propositions \ref{41} and \ref{co}, we do not use bump functions as in \cite{FG}, providing a relatively simpler proof. This method is also used in \cite{CKW}.

\medskip \noindent (6) In Section \ref{se5}, we give a proof for the optimality of Theorem \ref{th1.2} and \ref{th1.3}.
When $M = \S^N$, this result was taken for granted in \cite{FG} and \cite{DSW}, and can be shown by modifying \cite[Remark 1.2]{CFM} suitably.
However, we decided to include the proof here to point out the necessity of delicate estimates arising from the interaction between different bubbles.

\subsection{Structure of the paper}
Our paper is organized as follows:

In Section \ref{se2}, we handle the case when $3 \le N \le 5$ and $u_0 > 0$, proving Theorem \ref{th1.2}.

In Section \ref{se3}, we treat the case when $3 \le N \le 5$ and $u_0 = 0$, deducing Theorem \ref{th1.3}.

In Section \ref{se4}, we deal with the situation when $N \ge 6$ and $\nu = 1$, establishing Theorem \ref{th1.4}.

The proofs of Theorems \ref{th1.2}--\ref{th1.4} follow a parallel structure, though the difficult parts in each theorem vary, as depicted in the following table. 
\begin{table}[!ht]
\centering
\renewcommand\arraystretch{2}{
\begin{tabular}{|c|c|c|c|c|}
\hline
\diagbox{Results}{Cases} & {\makecell{$3 \le N \le 5$,\\ $u_0>0$}} & {\makecell{$3 \le N \le 5$,\\ $u_0=0$}} & {\makecell{$N \ge 6,\, \nu=1$,\\$u_0>0$}} & {\makecell{$N \ge 6,\, \nu=1$,\\ $u_0=0$}}\\
\hline
{\makecell{coercivity\\estimates for\\multi-bubbles}} & Proposition \ref{41} & Proposition \ref{co} & (not applicable) & (not applicable) \\
\hline
{\makecell{$L^{2N/(N+2)}(M)$-\\estimates for the\\error terms}} & Lemma \ref{s2l2} & Lemma \ref{lem3.3} & {\makecell{in the proof of\\Proposition \ref{pr41} }} & {\makecell{in the proof of\\Proposition \ref{pr4.7}}}\\
\hline
{\makecell{{\footnotesize{$\|\rho\|_{H^1(M)} \lesssim \|f\|_{H^{-1}(M)}$}}\\{\footnotesize{+(auxiliary terms)}}}} & {\makecell{Proposition \ref{pr2.2}\\(followed by\\Lemmas \ref{le2.3}\\and \ref{le24})}}
& Proposition \ref{pr3.2} & {\makecell{Proposition \ref{pr41}\\(for $N \ge 7$),\\Corollary \ref{cr43}\\(for $N=6$)}} & {\makecell{Proposition \ref{pr4.7}\\(for $N \ge 7$),\\Corollary \ref{cr46}\\(for $N=6$)}}\\
\hline
{\makecell{(auxiliary terms)\\$\lesssim \|f\|_{H^{-1}(M)}$}} & Proposition \ref{lou0} & Proposition \ref{3.3} & Proposition \ref{pr44} & Proposition \ref{pr47}\\
\hline
{\makecell{projections of\\the error terms in\\ the $\delta_j\frac{\partial{\mcv_j}}{\partial \delta_j}$-direction}} & {\makecell{Lemmas \ref{le2.4},\\\ref{le29}, and \ref{le2.6}}} & {\makecell{Lemmas \ref{le35}\\and \ref{le36}}} & {\makecell{Lemmas \ref{s43}\\and \ref{s44}}} & {\makecell{in the proof of\\Proposition \ref{pr47}}}\\
\hline
\end{tabular}}
\end{table}

In Section \ref{se5}, we show that the quantitative stability estimate stated in Theorems \ref{th1.2}, \ref{th1.3}, and \ref{th1.4} are all optimal, deriving Theorem \ref{th1.6}.

In Section \ref{se6}, we prove Corollary \ref{cor:main}.

In Appendices \ref{a} and \ref{a2}, we present some useful estimates and technical computations that are necessary in the proof of the main theorems.

\subsection{Conventions}
Here, we list some notations that will be used throughout the paper.
\begin{itemize}
\item[-] Given any $\delta > 0$ and $\sigma = (\sigma^1,\ldots,\sigma^N) \in \R^N$, the solution space of the linear problem
    \[-\Delta v=(2^*-1) U_{\delta,\sigma}^{2^*-2} v \quad \text { in } \R^N, \quad v \in \dot{H}^1(\R^N)\]
    is spanned by the functions
    \[Z_{\delta,\sigma}^0 := \delta \frac{\pa U_{\delta,\sigma}}{\pa \delta} \quad \text{and} \quad Z_{\delta,\sigma}^k := \delta \frac{\pa U_{\delta,\sigma}}{\pa \sigma^k} \quad \text{for } k=1,\ldots,N.\]
    Let $U$ be the standard bubble $U_{1,0}$ and $Z^k = Z_{1,0}^k$ for $k=0,\ldots,N$.
\vspace{1mm}

\item[-] The notations $\nabla_g$, $\Delta_g$, $\la\cdot,\cdot\ra_g$, $|\cdot|_g$, $dv_g$ and $\exp^g$ stand for the gradient, the Laplace-Beltrami operator,
    the inner product, the norm, the volume form and the exponential map with respect to the metric $g$, respectively.
    If the metric $g$ is Euclidean, we drop the subscript $g$.
    Also, the subscript $x$ in the integral $\int_M \cdots (dv_g)_x$ represents the variable of integration.
\vspace{1mm}

\item[-] We occasionally use the Einstein summation convention for repeated indices.
\vspace{1mm}

\item[-] For $\xi \in M$, a metric $g$ on $M$, and $r \in (0,\infty)$, we write $B^g_r(\xi) = \{x \in M: d_g(x,\xi) \le r\}$ and $(B^g_r(\xi))^c = \{x \in M: d_g(x,\xi) > r\}$.
    Let also $B_r(0) = \{y \in \R^N: |y| \le r\}$ and $B^c_r(0) = \{y \in \R^N: |y| > r\}$.
\vspace{1mm}

\item[-] $h_1=\mco(h_2)$ means that $|h_1| \le C|h_2|$ for a universal constant $C>0$ independent of $\vep_0 > 0$ in \eqref{tui}
    and the parameters $(\delta_1,\ldots,\delta_{\nu},\xi_1,\ldots,\xi_{\nu}) \in (0,\infty)^{\nu} \times M^{\nu}$ of bubble-like functions $\mcv_{\delta_1,\xi_1},\ldots,\mcv_{\delta_{\nu},\xi_{\nu}}$.
    Also, we write $h_1 = o(h_2)$ if $h_1/|h_2| \to 0$ as $\vep_0 \to 0$.
\vspace{1mm}

\item[-] $h_1 \lesssim h_2$ or $h_1 \gtrsim h_2$ denote that $h_1 \le Ch_2$ or $h_1 \ge Ch_2$ for a universal constant $C>0$, respectively.
    We write $h_1 \simeq h_2$ if $h_1 \lesssim h_2$ and $h_1 \gtrsim h_2$.
\vspace{1mm}

\item[-] Given a condition (C), we let $\bs{1}_{\textup{(C)}} = 1$ if (C) is true and 0 otherwise.
\end{itemize}

\section{The case $3 \le N \le 5$ and $u_0 > 0$}\label{se2}
This section is devoted to the proof of Theorem \ref{th1.2}. Throughout this section, we always assume that $3 \le N \le 5$ and $u_0 > 0$ on $M$.

\subsection{Setting of the problem}\label{subsec:set}
Given $\xi \in M$, we choose an orthonormal basis $\{\frac{\pa}{\pa \xi^1},\ldots,\frac{\pa}{\pa \xi^N}\}$ on the tangent space $T_{\xi}M$, isomorphic to $\R^N$, and define
\[\frac{\pa\mcv_{\delta,\xi}}{\pa \xi^k}(x) = \left. \frac{d}{dt} \mcv_{\delta,\exp^g_{\xi}\big(t \frac{\pa}{\pa \xi^k}\big)}(x) \right|_{t=0} \quad \text{for } x \in M.\]
Then we set $\mcv_i = \mcv_{\delta_i,\xi_i}$ as in \eqref{ui},
\begin{equation}\label{eq:wtmcz}
\mcz^0_i = \delta_i \frac{\pa\mcv_i}{\pa\delta_i}, \quad \text{and} \quad \mcz^k_i = \delta_i \frac{\pa\mcv_i}{\pa\xi_i^k} \quad \text{for } i=1,\ldots,\nu \text{ and } k=1,\ldots,N.
\end{equation}
By Assumption \ref{assum}, there exist $(\delta_1,\ldots,\delta_{\nu}, \xi_1,\ldots,\xi_{\nu}) \subset (0,\infty)^{\nu} \times M^{\nu}$ and $\vep_1 > 0$ small such that $\vep_1 \to 0$ as $\vep_0 \to 0$,
\[\begin{medsize}
\displaystyle \bigg\|u-\bigg(u_0+\sum_{i=1}^{\nu}\mcv_i\bigg)\bigg\|_{H^1(M)} = \inf\left\{\bigg\|u-\bigg(u_0+\sum_{i=1}^{\nu} \mcv_{\tde_i,\txi_i}\bigg)\bigg\|_{H^1(M)}: \(\tde_i,\txi_i\) \in (0,\infty) \times M,\ i=1,\ldots,\nu\right\} \le \vep_1,
\end{medsize}\]
$\delta_i \le \vep_1$ for $i=1,\ldots,\nu$, and
\begin{equation}\label{eq:bubblei}
\max\left\{\(\frac{\delta_i}{\delta_j} + \frac{\delta_j}{\delta_i} + \frac{d_g(\xi_i,\xi_j)^2}{\delta_i\delta_j}\)^{-\frac{N-2}{2}}: i,j = 1,\ldots,\nu\right\} \le \vep_1;
\end{equation}
refer to \cite[Appendix A]{BC}. Setting $\rho = u-(u_0+\sum_{i=1}^{\nu}\mcv_i)$ and $f = \mcl_gu-u^{2^*-1}$, we have
\begin{equation}\label{eqrho}
\begin{cases}
\displaystyle \mcl_g\rho - (2^*-1)\bigg(u_0+\sum_{i=1}^{\nu}\mcv_i\bigg)^{2^*-2}\rho = f+\I_1[\rho]+\I_2+\I_3+\I_4 \quad \text{on } M,\\
\displaystyle \big\langle \rho,\mcz^k_i \big\rangle_{H^1(M)} = 0 \quad \text{for } i=1,\ldots,\nu \text{ and } k=0,\ldots,N
\end{cases}
\end{equation}
where $\la \cdot,\cdot \ra_{H^1(M)}$ be the inner product on $H^1(M)$ associated with the norm $\|\cdot\|_{H^1(M)}$ in \eqref{eq:norm},
\[\I_1[\rho] := \bigg(u_0+\sum_{i=1}^{\nu}\mcv_i+\rho\bigg)^{2^*-1} - \bigg(u_0 + \sum_{i=1}^{\nu}\mcv_i\bigg)^{2^*-1} - (2^*-1) \bigg(u_0 + \sum_{i=1}^{\nu} \mcv_i\bigg)^{2^*-2}\rho,\]
\begin{equation}\label{eq:I2}
\I_2 := \bigg(u_0 + \sum_{i=1}^{\nu}\mcv_i\bigg)^{2^*-1}-u_0^{2^*-1}-\bigg(\sum_{i=1}^{\nu} \mcv_i\bigg)^{2^*-1},
\end{equation}
\begin{equation}\label{eq:I3I4}
\I_3 := \bigg(\sum_{i=1}^{\nu} \mcv_i\bigg)^{2^*-1} - \sum_{i=1}^{\nu} \mcv_i^{2^*-1}, \quad \text{and} \quad \I_4 := \sum_{i=1}^{\nu} \(-\mcl_g\mcv_i+\mcv_i^{2^*-1}\).
\end{equation}
To prove Theorem \ref{th1.2}, it is enough to verify that
\begin{equation}\label{rf}
\|\rho\|_{H^1(M)} \lesssim \|f\|_{H^{-1}(M)}
\end{equation}
provided $\vep_0 > 0$ small.

On the other hand, since the Yamabe invariant $Y(M,[g])$ is assumed to be positive, so is the conformal Laplacian $\mcl_g$.
Then the spectral theorem guarantees the existence of sequences of functions $\{\psi_m\}_{m \in \N} \in H^1(M)$ and positive numbers $\{\bmu_m\}_{m\in\N}$ satisfying the following properties:
\begin{itemize}
\item[-] $\psi_m$ solves an eigenvalue problem
    \[\mcl_g\psi_m = \bmu_m u_0^{2^*-2}\psi_m \quad \text{on } M.\]
\item[-] The set $\{\psi_m\}_{m\in\N}$ is an orthonormal basis of the space $L^2(M,u_0^{2^*-2}dv_g)$. Thus
    \begin{equation}\label{psij}
    \int_M u_0^{2^*-2}\psi_l\psi_m dv_g = \delta^{lm} := \begin{cases}
    1 &\text{for } l=m,\\
    0 &\text{for } l \ne m.
    \end{cases}
    \end{equation}
\item[-] $0<\bmu_1<\bmu_2 \le \bmu_3 \le \cdots \to \infty$.
\end{itemize}
Elliptic regularity ensures that $\psi_m \in C^{\infty}(M)$ for all $m \in \N$. Since $u_0 > 0$ on $M$, we also know that $\bmu_1=1$ and $\psi_1=\|u_0\|_{L^{2^*}(M)}^{-N/(N-2)}u_0$.
For later use, let $L$ be the greatest number such that $0<\bmu_l<2^*-1 = \frac{N+2}{N-2}$ for all $l\le L$.

Finally, it is noteworthy that
\begin{equation}\label{maeu}
-\Delta_g u=-\Delta u-\(g^{ij}-\delta^{ij}\) \pa_{ij}^2 u + g^{ij}\Gamma_{ij}^k \pa_ku
\end{equation}
in $g$-normal coordinates around any fixed point $\xi \in M$, where $\Gamma_{ij}^k$ is the Christoffel symbol, and
\begin{equation}\label{exp}
g^{ij}(x) = \delta^{ij}(x)+\mco\(d_g(x,\xi)^2\) \quad \text{and} \quad \(g^{ij} \Gamma_{ij}^k\)(x) = \mco(d_g(x,\xi))
\end{equation}
for $x \in M$ near $\xi$.

\subsection{Preliminary computations}
Let us define
\begin{equation}\label{rq}
\begin{cases}
\displaystyle q_{ij} = \left[{\frac{\delta_i}{\delta_j}} + {\frac{\delta_j}{\delta_i}} + \frac{d_g(\xi_i,\xi_j)^2}{{\delta_i\delta_j}}\right]^{-\frac{N-2}{2}}, \quad \mcq = \max\{q_{ij}: i,j = 1,\ldots,\nu\} \le \vep_1, \\
\displaystyle \msr_{ij} = \max\left\{\sqrt{\frac{\delta_i}{\delta_j}}, \sqrt{\frac{\delta_j}{\delta_i}}, \frac{d_g(\xi_i,\xi_j)}{\sqrt{\delta_i\delta_j}}\right\} \simeq q_{ij}^{-\frac{1}{N-2}}.
\end{cases}
\end{equation}
The following lemma serves estimates for the inner products of $\mcv_i$, $\mcz^k_i$ and $\psi_m$, which will be frequently invoked later.
\begin{lemma}\label{le2p}
Assume that $i,j \in \{1,\ldots,\nu\}$, $k,l \in \{0,1,\ldots,N\}$, and $m \in \N$. We have
\[\la \mcv_i,\mcv_i \ra_{H^1(M)} = \int_{\R^N} U^{2^*} + o\Big(\delta_i^{\frac{N-2}{2}}\Big), \quad \big\langle \mcz^k_i,\mcv_i \big\rangle_{H^1(M)} = o\Big(\delta_i^{\frac{N-2}{2}}\Big),\]
\[\big\langle \mcz^k_i,\mcz^l_i \big\rangle_{H^1(M)} = \big\|Z^k\big\|_{\dot{H}^1(\R^N)}^2 \delta^{kl} + o\Big(\delta_i^{\frac{N-2}{2}}\Big),\]
and
\[\left|\la \mcv_i,\mcv_j \ra_{H^1(M)}\right| + \left|\big\langle \mcz^k_i,\mcv_j \big\rangle_{H^1(M)}\right| + \left|\big\langle \mcz^k_i,\mcz^l_j \big\rangle_{H^1(M)}\right| = \mco(q_{ij}) + o\Big(\max_{\ell = 1,\ldots,\nu}\delta_{\ell}^{\frac{N-2}{2}}\Big)\]
provided $i \ne j$. Additionally,
\[\left|\la \psi_m,\mcv_i \ra_{H^1(M)}\right| + \left|\big\langle \psi_m,\mcz^k_i \big\rangle_{H^1(M)}\right| = \mco\Big(\delta_i^{\frac{N-2}{2}}\Big).\]
\end{lemma}
\begin{proof}
Using \eqref{eq:dist1} and \eqref{eq:dist2} with $\xi = \xi_i$ and $y_2 = 0$, we obtain
\begin{multline*}
\mcz^k_i(x) \\
= \begin{cases}
\displaystyle (N-2)\alpha_N \frac{\delta_i^{\frac{N}{2}}}{(\delta_i^2+|y|^2)^{\frac{N}{2}}} \left[y^k + \mco\(|y|^3\)\right] &\text{if } x = \exp_{\xi_i}^g(y) \in B^g_{\frac{r_0}{2}}(\xi_i),\, y = (y^1,\cdots,y^N), \\
\displaystyle \mco\Big(\delta_i^{\frac{N-2}{2}}\Big) &\text{if } x \in M \setminus B^g_{\frac{r_0}{2}}(\xi_i)
\end{cases}
\end{multline*}
for $i=1,\ldots,\nu$ and $k=1,\ldots,N$. Once we have this, the proof becomes standard.
\end{proof}		
Next, we present a coercivity estimate tailored to our setting, which serves as an important tool in the proof of Proposition \ref{pr2.2}.
Its Euclidean version can be found in \cite[Proposition 3.10]{FG} where bump functions is a key ingredient. Here, we present a different proof based on a blow-up argument.
Because the proof is a bit lengthy, we defer it to Appendix \ref{subsec:coer}.
\begin{prop}\label{41}
Suppose that $u_0$ is non-degenerate. Let
\begin{multline}\label{eq:Eperp}
E^{\perp} = \left\{\vrh \in H^1(M): \la \vrh,\mcv_i \ra_{H^1(M)} = \big\langle \vrh,\mcz^k_i \big\rangle_{H^1(M)} = \la \vrh,\psi_m \ra_{H^1(M)} = 0,\right. \\
\left. \text{for } i=1,\ldots,\nu,\ k=0,1,\ldots,N,\ m=1,\ldots,L\right\}.
\end{multline}
Then there exists a constant $c_0 \in(0,1)$ such that
\begin{equation}\label{eq:coer}
(2^*-1) \int_M \bigg(u_0+\sum_{i=1}^{\nu}\mcv_i\bigg)^{2^*-2}\vrh^2 dv_g \le c_0\|\vrh\|_{H^1(M)}^2 \quad \text{for any } \vrh \in E^{\perp}.
\end{equation}
\end{prop}

We will also need estimates for the $L^{2N/(N+2)}(M)$-norm of $\I_2$, $\I_3$, and $\I_4$.
\begin{lemma}\label{s2l2}
We have
\begin{equation}\label{u0rho}
\|\I_2\|_{L^{\frac{2N}{N+2}}(M)}+\|\I_3\|_{L^{\frac{2N}{N+2}}(M)}+\|\I_4\|_{L^{\frac{2N}{N+2}}(M)} \lesssim \mcq + \max_{\ell}\delta_{\ell}^{\frac{N-2}{2}}.
\end{equation}
\end{lemma}
\begin{proof}
By Lemma \ref{a4} below, it holds that
\begin{equation}\label{uln1}
\|\mcv_i\|_{L^{\frac{2N}{N+2}}(M)} \lesssim \delta_i^{\frac{N-2}{2}} \quad \text{and} \quad \left\|\mcv_i^{2^*-2}\right\|_{L^{\frac{2N}{N+2}}(M)} \lesssim \delta_i^{\frac{N-2}{2}}.
\end{equation}
Using \eqref{iqu}, \eqref{uln1}, and $u_0 \in L^{\infty}(M)$, we readily compute
\begin{equation}\label{i2r}
\|\I_2\|_{L^{\frac{2N}{N+2}}(M)} \lesssim \sum_{i=1}^{\nu} \left\|u_0\mcv_i^{2^*-2}\right\|_{L^{\frac{2N}{N+2}}(M)} + \sum_{i=1}^{\nu} \left\|u_0^{2^*-2}\mcv_i\right\|_{L^{\frac{2N}{N+2}}(M)} \lesssim \max_{\ell}\delta_{\ell}^{\frac{N-2}{2}}.
\end{equation}
Besides, \eqref{iqu} and Lemma \ref{a22} tell us that
\begin{equation}\label{i3r}
\|\I_3\|_{L^{\frac{2N}{N+2}}(M)} \lesssim \sum_{i \ne j} \left\|\mcv^{2^*-2}_i \mcv_j\right\|_{L^{\frac{2N}{N+2}}(M)} \lesssim \mcq.
\end{equation}
Let $x = \exp_{\xi_i}^g(y) \in M$ for $y \in B_{r_0}(0)$ and $\chi_i(x) = \chi(d_g(x,\xi_i))$. We have
\begin{equation}\label{i4}
\begin{aligned}
\I_4(x) &= \sum_{i=1}^{\nu} \left[\left\{\(\chi_i(x)U_{\delta_i,0}(y)+(1-\chi_i(x))U_{\delta_i,0}(\tfrac{r_0}{2})\)^{2^*-1} - \chi_i(x)U_{\delta_i,0}^{2^*-1}(y)\right\} \right. \\
&\hspace{35pt} + \left. (\Delta_g\chi_i)(x) \(U_{\delta_i,0}(y)-U_{\delta_i,0}(\tfrac{r_0}{2})\) + \la \nabla_g\chi_i(x), \nabla_gU_{\delta_i,0}(y)\ra_g\right] \\
&\ - \sum_{i=1}^{\nu} \left[\chi_i(x) \left\{\mcl_g(U_{\delta_i,0}(y)) + (\Delta U_{\delta_i,0})(y)\right\} + \ka_NR_g(x)(1-\chi_i(x))U_{\delta_i,0}(\tfrac{r_0}{2})\right].
\end{aligned}
\end{equation}
By applying \eqref{chi0}, \eqref{maeu}, and \eqref{exp}, we easily check that
\begin{equation}\label{cutt}
\begin{cases}
\displaystyle \left\|\(\chi_iU_{\delta_i,0}(y)+(1-\chi_i)U_{\delta_i,0}(\tfrac{r_0}{2})\)^{2^*-1} - \chi_i U_{\delta_i,0}^{2^*-1}(y)\right\|_{L^{\frac{2N}{N+2}}(M)} \lesssim \delta_i^{N-2}, \\
\displaystyle \left\|(\Delta_g\chi_i) \(U_{\delta_i,0}(y)-U_{\delta_i,0}(\tfrac{r_0}{2})\) + \la \nabla_g\chi_i, \nabla_gU_{\delta_i,0}(y)\ra_g\right\|_{L^{\frac{2N}{N+2}}(M)} \lesssim \delta_i^{\frac{N-2}{2}}, \\
\displaystyle |\chi_i(x) \left\{\mcl_g(U_{\delta_i,0}(y)) + (\Delta U_{\delta_i,0})(y)\right\} + \ka_NR_g(x)(1-\chi_i(x))U_{\delta_i,0}(\tfrac{r_0}{2})| \lesssim \chi(|y|)U_{\delta_i,0}(y)+\delta_i^{\frac{N-2}{2}}.
\end{cases}
\end{equation}

From \eqref{uln1} again, we observe
\begin{equation}\label{i4r}
\|\I_4\|_{L^{\frac{2N}{N+2}}(M)} \lesssim \max_{\ell}\delta_{\ell}^{\frac{N-2}{2}}.
\end{equation}
Putting \eqref{i2r}, \eqref{i3r}, and \eqref{i4r} together, we obtain \eqref{u0rho}.
\end{proof}

\subsection{Proof of Theorem \ref{th1.2}}\label{su2.3}
One can decompose the function $\rho = u-(u_0+\sum_{i=1}^{\nu}\mcv_i)$ in Subsection \ref{subsec:set} as
\begin{equation}\label{dec}
\rho = \rho_1 + \sum_{i=1}^{\nu} \beta_i\mcv_i + \sum_{i=1}^{\nu}\sum_{k=0}^N \beta_i^k\mcz^k_i + \sum_{m=1}^L \vth_m \psi_m \quad \text{for some } \beta_i, \beta_i^k, \vth_m \in \R,\, \rho_1 \in E^{\perp}
\end{equation}
where $E^{\perp}$ is the space defined in \eqref{eq:Eperp}. We will accomplish the proof of Theorem \ref{th1.2}, that is, the verification of \eqref{rf} in two stages; Propositions \ref{pr2.2} and \ref{lou0}.
\begin{prop}\label{pr2.2}
Let $\mcq$ be the quantity in \eqref{rq}. It holds that
\[\|\rho\|_{H^1(M)} \lesssim \|f\|_{H^{-1}(M)} + \mcq + \max_{\ell=1,\ldots,\nu}\delta_{\ell}^{\frac{N-2}{2}}.\]
\end{prop}
\noindent The quantities $\mcq$ and $\max_{\ell=1,\ldots,\nu}\delta_{\ell}^{(N-2)/2}$ are non-comparable.
We establish Proposition \ref{pr2.2} by deriving Lemmas \ref{le2.3} and \ref{le24} and then combining them.

\begin{lemma}\label{le2.3}
Define a number
\[\mca = \sum_{i=1}^{\nu}|\beta_i| + \sum_{i=1}^{\nu}\sum_{k=0}^N |\beta_i^k| + \sum_{m=1}^L|\vth_m|,\]
which is small by virtue of Lemma \ref{le2p}. It holds that
\begin{equation}\label{rr1}
\|\rho_1\|_{H^1(M)} \lesssim \|f\|_{H^{-1}(M)} + \mca + \mcq + \max_{\ell=1,\ldots,\nu}\delta_{\ell}^{\frac{N-2}{2}}.
\end{equation}
\end{lemma}
\begin{proof}
Using \eqref{ab6}, we obtain
\begin{equation}\label{i1r}
\|\I_1[\rho]\|_{L^{\frac{2N}{N+2}}(M)} \lesssim \|\rho\|_{H^1(M)}^2 + \|\rho\|_{H^1(M)}^{2^*-1} \simeq \|\rho\|_{H^1(M)}^2.
\end{equation}

By testing \eqref{eqrho} with $\rho_1$ and then invoking \eqref{u0rho}, \eqref{dec}, and \eqref{i1r}, we arrive at
\begin{equation}\label{pp1}
\begin{aligned}
\|\rho_1\|_{H^1(M)}^2 &= (2^*-1)\int_M \bigg(u_0+\sum_{i=1}^{\nu}\mcv_i\bigg)^{2^*-2}\rho\rho_1 dv_g + \mco\(\|f\|_{H^{-1}(M)}\|\rho_1\|_{H^1(M)}\) \\
&\ + \mco\(\|\rho\|_{H^1(M)}^2\|\rho_1\|_{H^1(M)}\) + \mco\(\Big(\mcq + \max_{\ell}\delta_{\ell}^{\frac{N-2}{2}}\Big) \|\rho_1\|_{H^1(M)}\).
\end{aligned}
\end{equation}
In addition, since $\rho_1 \in E^{\perp}$, Proposition \ref{41} gives
\begin{equation}\label{p41}
(2^*-1) \int_M \bigg(u_0+\sum_{i=1}^{\nu}\mcv_i\bigg)^{2^*-2} \rho_1^2 dv_g \le c_0\|\rho_1\|_{H^1(M)}^2 \quad \text{for some } c_0 \in (0,1).
\end{equation}
We see from H\"older's inequality that
\begin{equation}\label{ap1}
\begin{medsize}
\displaystyle (2^*-1) \int_M \bigg(u_0+\sum_{i=1}^{\nu}\mcv_i\bigg)^{2^*-2} \rho_1 \(\sum_{i=1}^{\nu} \beta_i\mcv_i + \sum_{i=1}^{\nu}\sum_{k=0}^N \beta_i^k\mcz^k_i + \sum_{m=1}^L \vth_m\psi_m\) dv_g = \mco\(\mca\|\rho_1\|_{H^1(M)}\)
\end{medsize}
\end{equation}
and from \eqref{dec} that
\begin{equation}\label{rb1}
\|\rho\|_{H^1(M)} \lesssim \|\rho_1\|_{H^1(M)} + \mca.
\end{equation}
Plugging \eqref{p41}--\eqref{rb1} into \eqref{pp1} produces \eqref{rr1} as desired.
\end{proof}

\begin{lemma}\label{le24}
It holds that
\begin{equation}\label{bb}
\mca \lesssim \|f\|_{H^{-1}(M)} + \mcq + \max_{\ell=1,\ldots,\nu}\delta_{\ell}^{\frac{N-2}{2}}.
\end{equation}
\end{lemma}
\begin{proof}
Firstly, given any $j \in \{1,\ldots,\nu\}$ and $q \in \{0,1,\ldots,N\}$, it holds that $\langle \rho,\mcz^q_j \rangle_{H^1(M)} = 0$, so by \eqref{dec},
\[\bigg\langle \sum_{i=1}^{\nu} \beta_i\mcv_i + \sum_{i=1}^{\nu}\sum_{k=0}^N \beta_i^k\mcz^k_i + \sum_{m=1}^L \vth_m\psi_m, \mcz^q_j \bigg\rangle_{H^1(M)} = 0.\]
By virtue of Lemma \ref{le2p}, it reads
\begin{multline}\label{z11}
|\beta_j^q| \left[\int_{\R^N} |\nabla Z^q|^2 + o\Big(\delta_j^{\frac{N-2}{2}}\Big)\right] + \sum_{(i,k) \ne (j,q)} |\beta_i^k| \left[\mco(\mcq) + o\Big(\max_{\ell}\delta_{\ell}^{\frac{N-2}{2}}\Big)\right] \\
+|\beta_j| o\Big(\delta_j^{\frac{N-2}{2}}\Big) + \sum_{i \ne j} |\beta_i| \left[\mco(\mcq) + o\Big(\max_{\ell}\delta_{\ell}^{\frac{N-2}{2}}\Big)\right] + \sum_{m=1}^L |\vth_m| \mco\Big(\delta_j^{\frac{N-2}{2}}\Big) = 0.
\end{multline}
	
Secondly, after testing \eqref{eqrho} with $\mcv_j$ for $j \in \{1,\ldots,\nu\}$, we apply \eqref{u0rho}, which yields
\begin{equation}\label{z12}
\left|\int_M \left[\mcl_g\rho - (2^*-1)\bigg(u_0+\sum_{i=1}^{\nu}\mcv_i\bigg)^{2^*-2} \rho\right]\mcv_j dv_g\right| \lesssim \|f\|_{H^{-1}(M)} + \mcq + \max_{\ell}\delta_{\ell}^{\frac{N-2}{2}} + o(\mca).
\end{equation}
Let us examine the left-hand side of \eqref{z12}. Employing \eqref{iqu}, H\"older's inequality, $u_0 \in L^{\infty}(M)$, \eqref{uln1}, Lemma \ref{a22}, and
\begin{align*}
\int_M \mcv_j^{2^*-1}\rho_1 dv_g &= \int_M \(-\mcl_g\mcv_j+\mcv_j^{2^*-1}\)\rho_1 dv_g \\
&= \mco\(\left\|\mcl_g\mcv_j-\mcv_j^{2^*-1}\right\|_{L^{\frac{2N}{N+2}}(M)}\|\rho_1\|_{H^1(M)}\) = \mco\Big(\delta_j^{\frac{N-2}{2}}\Big)\|\rho_1\|_{H^1(M)} \quad (\text{by } \eqref{i4r}),
\end{align*}
we calculate
\begin{equation}\label{teu}
\begin{aligned}
&\ \int_M \bigg(u_0+\sum_{i=1}^{\nu}\mcv_i\bigg)^{2^*-2} \rho_1\mcv_j dv_g \\
&= \int_M \mcv_j^{2^*-1}\rho_1 dv_g + \mco\bigg(\left\|u_0^{2^*-2} \mcv_j\right\|_{L^{\frac{2N}{N+2}}(M)} + \sum_{i \ne j} \left\|\mcv_i^{2^*-2}\mcv_j\right\|_{L^{\frac{2N}{N+2}}(M)} \\
&\hspace{100pt} + \left\|u_0\mcv_j^{2^*-2}\right\|_{L^{\frac{2N}{N+2}}(M)} + \sum_{i \ne j} \left\|\mcv_i\mcv_j^{2^*-2}\right\|_{L^{\frac{2N}{N+2}}(M)}\bigg) \|\rho_1\|_{H^1(M)} \\
&= \mco\Big(\mcq + \max_{\ell}\delta_{\ell}^{\frac{N-2}{2}}\Big) \|\rho_1\|_{H^1(M)}.
\end{aligned}
\end{equation}
Also, direct computations show that
\begin{equation}\label{teu2}
\begin{aligned}
&\ \left|\bigg\langle \sum_{i=1}^{\nu} \beta_i\mcv_i + \sum_{i=1}^{\nu}\sum_{k=0}^N \beta_i^k\mcz^k_i + \sum_{m=1}^L \vth_m\psi_m,\mcv_j \bigg\rangle_{H^1(M)} \right.\\
&\quad \left. - (2^*-1)\int_M\bigg(u_0+\sum_{i=1}^{\nu}\mcv_i\bigg)^{2^*-2} \left[\sum_{i=1}^{\nu} \beta_i\mcv_i + \sum_{i=1}^{\nu}\sum_{k=0}^N \beta_i^k\mcz^k_i + \sum_{m=1}^L \vth_m\psi_m\right] \mcv_j dv_g\right| \\
&= (2^*-2)|\beta_j| \int_{\R^N} U^{2^*} + o(\mca).
\end{aligned}
\end{equation}
Having $\rho_1 \in E^{\perp}$ in hand, we conclude from \eqref{z12}--\eqref{teu2} and \eqref{rr1} that
\begin{equation}\label{bu}
(2^*-2)|\beta_j|\int_{\R^N}U^{2^*} \lesssim \|f\|_{H^{-1}(M)} + \mcq + \max_{\ell}\delta_{\ell}^{\frac{N-2}{2}} + o(\mca).
\end{equation}

Lastly, given any $s \in \{1,\ldots,L\}$, by appealing to \eqref{iqu}, H\"older's inequality, $\psi_s \in L^{\infty}(M)$, and \eqref{uln1}, we find
\begin{equation}\label{ccp}
\begin{aligned}
\left|\int_M \left[\bigg(u_0+\sum_{i=1}^{\nu}\mcv_i\bigg)^{2^*-2}-u_0^{2^*-2}\right] \rho\psi_s dv_g\right| &\lesssim \sum_{i=1}^{\nu} \int_M \(\mcv_i^{2^*-2} + u_0^{2^*-3}\mcv_i\)|\rho||\psi_s| dv_g \\
&\lesssim \max_i \(\left\|\mcv_i^{2^*-2}\right\|_{L^{\frac{2N}{N+2}}(M)} + \left\|\mcv_i\right\|_{L^{\frac{2N}{N+2}}(M)}\) \\
&= \mco\Big(\max_{\ell}\delta_{\ell}^{\frac{N-2}{2}}\Big).
\end{aligned}
\end{equation}
Testing \eqref{eqrho} with $\psi_s$ and using \eqref{ccp}, we derive
\begin{equation}\label{bp}
(2^*-1-\bmu_s)|\vth_s| \int_M u_0^{2^*-2}\psi_s^2 dv_g \lesssim \|f\|_{H^{-1}(M)} + \mcq + \max_{\ell}\delta_{\ell}^{\frac{N-2}{2}} + o(\mca).
\end{equation}

\medskip
Inequality \eqref{bb} is a consequence of \eqref{z11}, \eqref{bu}, and \eqref{bp}.
\end{proof}

\begin{prop}\label{lou0}
It holds that
\begin{equation}\label{prlo}
\mcq + \max_{\ell=1,\ldots,\nu}\delta_{\ell}^{\frac{N-2}{2}} \lesssim \|f\|_{H^{-1}(M)}.
\end{equation}
\end{prop}
\noindent Its derivation is the most involved part of the proof of \eqref{rf}. Fix $j \in \{1,\ldots,\nu\}$. By testing \eqref{eqrho} with $\mcz^0_j$, we obtain
\begin{align}
\int_M \I_2\mcz^0_j dv_g + \int_M \I_3\mcz^0_j dv_g + \int_M \I_4\mcz^0_j dv_g &= -\int_M f\mcz^0_j dv_g -\int_M \I_1[\rho]\mcz^0_j dv_g \label{tezk1} \\
&\ + \int_M \bigg[\mcl_g\rho - (2^*-1)\bigg(u_0+\sum_{i=1}^{\nu}\mcv_i\bigg)^{2^*-2}\rho\bigg] \mcz^0_j dv_g. \nonumber
\end{align}
In Lemmas \ref{le2.4}--\ref{le2.6}, we analyze the left-hand side of \eqref{tezk1} term by term, which is essential in proving \eqref{prlo}.
\begin{lemma}\label{le2.4}
Let $\mfa_N = \frac{\alpha_N^{2^*-1}}{2^*} |\S^{N-1}| > 0$. For any $j \in \{1,\ldots,\nu\}$, we have
\begin{equation}\label{eq:I2Z}
\int_M \I_2\mcz^0_j dv_g = \mfa_N u_0(\xi_j) \delta_j^{\frac{N-2}{2}} + o\Big(\mcq + \max_{\ell=1,\ldots,\nu}\delta_{\ell}^{\frac{N-2}{2}}\Big).
\end{equation}
\end{lemma}
\begin{proof}
Employing \eqref{ab6}, we see that there exists a constant $\eta>0$ such that
\begin{equation}\label{i2ex}
\begin{aligned}
\begin{medsize}
\displaystyle \I_2
\end{medsize}
&\begin{medsize}
\displaystyle = \left[(2^*-1)u_0\bigg(\sum_{i=1}^{\nu} \mcv_i\bigg)^{2^*-2} + \mco\(u_0^2 \bigg(\sum_{i=1}^{\nu} \mcv_i\bigg)^{2^*-3}\) + \mco\(u_0^{2^*-1}\)\right] 1_{\cup_{i=1}^{\nu}B^g_{\eta\sqrt{\delta_i}}(\xi_i)}
\end{medsize} \\
&\begin{medsize}
\displaystyle \ +\left[(2^*-1)u_0^{2^*-2}\sum_{i=1}^{\nu} \mcv_i + \mco\(u_0^{2^*-3} \bigg(\sum_{i=1}^{\nu}\mcv_i\bigg)^2\)
+ \mco\bigg(\bigg(\sum_{i=1}^{\nu} \mcv_i\bigg)^{2^*-1}\bigg)\right] 1_{\cap_{i=1}^{\nu}\(B^g_{\eta\sqrt{\delta_i}}(\xi_i)\)^c}.
\end{medsize}
\end{aligned}
\end{equation}
We compute the integral $\int_M \I_2\mcz^0_j dv_g$ by splitting it into three steps.

\medskip \noindent \textbf{Step 1.} Since $|\mcz^0_j| \lesssim \mcv_j$, it holds that
\begin{equation}\label{21}
\begin{aligned}
&\ (2^*-1) \int_{\cup_{i=1}^{\nu} B^g_{\eta\sqrt{\delta_i}}(\xi_i)} u_0\mcv_j^{2^*-2}\mcz^0_j dv_g \\
&= (2^*-1)\left[\int_M u_0\mcv_j^{2^*-2}\mcz^0_j dv_g - \int_{\big(\cup_{i=1}^{\nu} B^g_{\eta\sqrt{\delta_i}}(\xi_i)\big)^c} u_0\mcv_j^{2^*-2}\mcz^0_j dv_g\right] \\
&= (2^*-1) \delta_j^{\frac{N-2}{2}} \int_{\{|y| \le \frac{r_0}{2\delta_j}\}} u_0\big(\exp_{\xi_j}^g(\delta_j y)\big) \(U^{2^*-2}Z^0\)(y) \(1+\mco\(\delta_j^2|y|^2\)\) dy + \mco\Big(\delta_j^{\frac{N}{2}}\Big) \\
&= \frac{N-2}{2} \delta_j^{\frac{N-2}{2}} \left[u_0(\xi_j) \int_{\R^N} U^{2^*-1} + \mco\bigg(\int_{\{|y| \le \frac{r_0}{2\delta_j}\}} |\delta_j y| U^{2^*-1}(y)dy\bigg)\right] + \mco\Big(\delta_j^{\frac{N}{2}}\Big) \\
&= \mfa_N \delta_j^{\frac{N-2}{2}} u_0(\xi_j) + o\Big(\delta_j^{\frac{N-2}{2}}\Big).
\end{aligned}
\end{equation}
In addition, by \eqref{iqu}, we have
\begin{equation}\label{22}
\begin{aligned}
\left|\int_M u_0 \left[\bigg(\sum_{i=1}^{\nu}\mcv_i\bigg)^{2^*-2} - \mcv_j^{2^*-2}\right]\mcz^0_j dv_g\right|
&\lesssim \sum_{i \ne j} \int_M \(\mcv_i^{2^*-2} + \mcv_i\mcv_j^{2^*-3}\) \left|\mcz^0_j\right| dv_g \\
&\lesssim \sum_{i \ne j} \int_M \(\mcv_i^{2^*-2}\mcv_j + \mcv_i\mcv_j^{2^*-2}\) dv_g = o(\mcq)
\end{aligned}
\end{equation}
where the last equality is proved in Appendix \ref{subsec:tech}.
	
\medskip \noindent \textbf{Step 2.} By Young's inequality and Lemma \ref{a4},
\begin{align}
\int_M u_0^2 \bigg(\sum_{i=1}^{\nu} \mcv_i\bigg)^{2^*-3} \left|\mcz^0_j\right| dv_g &\lesssim \sum_{i=1}^{\nu} \int_M \mcv_i^{2^*-3}\mcv_j dv_g \label{23} \\
&\lesssim \sum_{i=1}^{\nu} \int_M \mcv_i^{2^*-2} dv_g \lesssim
\left\{\!\begin{aligned}
&\delta_i &\text{if } N=3\\[1ex]
&\delta_i^2|\log\delta_i| &\text{if } N=4\\[1ex]
&\delta_i^2 &\text{if } N=5
\end{aligned}\right\}
= o\Big(\max_{\ell}\delta_{\ell}^{\frac{N-2}{2}}\Big). \nonumber
\end{align}
Furthermore,
\begin{equation}\label{25}
\int_{\cup_{i=1}^{\nu} B^g_{\eta\sqrt{\delta_i}}(\xi_i)} u_0^{2^*-1} \left|\mcz^0_j\right| dv_g
\lesssim \sum_{i=1}^{\nu} \bigg(\int_{B^g_{\eta\sqrt{\delta_i}}(\xi_i)} u_0^{2^*} dv_g\bigg)^{\frac{2^*-1}{2^*}} \left\|\mcz^0_j\right\|_{H^1(M)} \lesssim \max_{\ell}\delta_{\ell}^{\frac{N+2}{4}}.
\end{equation}

\medskip \noindent \textbf{Step 3.} Let us write $\Omega = \cap_{i=1}^{\nu}(B^g_{\eta\sqrt{\delta_i}}(\xi_i))^c$. From Young's inequality again, we observe
\begin{equation}\label{26}
\begin{aligned}
\sum_{i=1}^{\nu} \int_{\Omega} u_0^{2^*-2} \mcv_i \left|\mcz^0_j\right| dv_g
&\lesssim \sum_{i=1}^{\nu}\int_{\big(B^g_{\eta\sqrt{\delta_i}}(\xi_i)\big)^c} \mcv_i^2dv_g \\
&\lesssim \left\{\!\begin{aligned}
&\delta_i &\text{if } N=3\\[1ex]
&\delta_i^2|\log\delta_i| &\text{if } N=4\\[1ex]
&\delta_i^{\frac{5}{2}} &\text{if } N=5
\end{aligned}\right\}
= o\Big(\max_{\ell}\delta_{\ell}^{\frac{N-2}{2}}\Big).
\end{aligned}
\end{equation}
Since $\mcz^0_j$ is uniformly bounded on the set $\{x \in M: d_g(x,\xi_j) \ge \eta\sqrt{\delta_j}\}$, it follows that
\begin{equation}\label{28}
\sum_{i=1}^{\nu} \int_{\Omega} u_0^{2^*-3}\mcv_i^2\left|\mcz^0_j\right| dv_g \lesssim \sum_{i=1}^{\nu} \int_{\Omega} \mcv_i^2 dv_g = o\Big(\max_{\ell}\delta_{\ell}^{\frac{N-2}{2}}\Big)
\end{equation}
and
\begin{equation}\label{29}
\int_{\Omega} \bigg(\sum_{i=1}^{\nu}\mcv_i\bigg)^{2^*-1} \left|\mcz^0_j\right| dv_g \lesssim \sum_{i=1}^{\nu} \int_{\Omega} \mcv_i^{2^*-1} dv_g \lesssim \max_{\ell}\delta_{\ell}^{\frac{N}{2}}.
\end{equation}
Putting \eqref{21}--\eqref{22}, \eqref{23}--\eqref{25}, and \eqref{26}--\eqref{29} together, we finish the proof of \eqref{eq:I2Z}.
\end{proof}

\begin{lemma}\label{le29}
For any $j \in \{1,\ldots,\nu\}$, it holds that
\begin{equation}\label{eq:I3Z}
\int_M \I_3\mcz^0_j dv_g = (2^*-1) \sum_{i \ne j} \int_{B^g_{r_0/2}(\xi_j)}\mcu_j^{2^*-2}\delta_j\frac{\pa \mcu_j}{\pa \delta_j} \mcv_idv_g + o(\mcq)
\end{equation}
where $\mcu_j := \mcu_{\delta_j,\xi_j}$ is defined by \eqref{udx}. Moreover, if $\delta_i \ge \delta_j$ for some indices $1 \le i \ne j \le \nu$, then
\begin{equation}\label{neqij}
\int_{B^g_{r_0/2}(\xi_j)}\mcu_j^{2^*-2}\delta_j\frac{\pa \mcu_j}{\pa \delta_j} \mcv_idv_g \gtrsim q_{ij}
\end{equation}
provided $q_{ij}$ in \eqref{rq} small.
\end{lemma}
\begin{proof}
Arguing as in \cite[Lemma 2.1]{DSW} and using the estimate
\[\sum_{i \ne j} \int_{\big(B^g_{r_0/2}(\xi_j)\big)^c} \mcv_j^{2^*-2}\delta_j\frac{\pa \mcv_j}{\pa \delta_j}\mcv_i dv_g = \mco\bigg(\sum_{i \ne j} \delta_i^{\frac{N-2}{2}}\delta_j^{\frac{N+2}{2}}\bigg) = o\Big(\sum_{i \ne j}q_{ij}\Big) = o(\mcq),\]
we get
\begin{equation}\label{eq:le29}
\begin{aligned}
\int_M \I_3\mcz^0_j dv_g &= (2^*-1) \sum_{i \ne j} \int_M \mcv_j^{2^*-2}\delta_j\frac{\pa \mcv_j}{\pa \delta_j}\mcv_i dv_g + o(\mcq) \\
&= (2^*-1)\sum_{i \ne j} \int_{B^g_{r_0/2}(\xi_j)} \mcu_j^{2^*-2}\delta_j\frac{\pa \mcu_j}{\pa \delta_j} \mcv_i dv_g + o(\mcq),
\end{aligned}
\end{equation}
so \eqref{eq:I3Z} is true.

\medskip
In the rest of the proof, we establish \eqref{neqij} for all indices $1 \le i \ne j \le \nu$ satisfying $\delta_i \ge \delta_j$.
To achieve this, we analyze the integral in the rightmost side of \eqref{eq:le29} by considering three cases.
Let $r_0 > 0$ be a small number appearing in \eqref{ui}.

\medskip \noindent \textbf{Case 1. ($d_g(\xi_i,\xi_j) \ge \frac{3r_0}{2}$):} If $d_g(x,\xi_j)\le \frac{r_0}{2}$, then $d_g(x,\xi_i)\ge r_0$ and so $\mcv_i = U_{\delta_i,0}(\frac{r_0}{2})$. Hence
\[\int_{B^g_{r_0/2}(\xi_j)} \mcu_j^{2^*-2}\delta_j\frac{\pa \mcu_j}{\pa \delta_j} \mcv_i dv_g \simeq U_{\delta_i,0}\(\tfrac{r_0}{2}\) \delta_j^{\frac{N-2}{2}}\int_{\R^N}U^{2^*-2}Z^0 \gtrsim (\delta_i\delta_j)^{\frac{N-2}{2}} \simeq q_{ij}.\]

\medskip \noindent \textbf{Case 2. ($r_0 \le d_g(\xi_i,\xi_j) \le \frac{3r_0}{2}$):} If $d_g(x,\xi_j)\le \frac{r_0}{2}$, then $\frac{r_0}{2} \le d_g(x,\xi_i)\le 2r_0$ and so
\begin{equation}\label{eq:mcvi}
\begin{cases}
U_{\delta_i,0}(r_0)\le \mcv_i\le U_{\delta_i,0}\(\tfrac{r_0}{2}\) &\text{if } \frac{r_0}{2} \le d_g(x,\xi_i) \le r_0,\\
\mcv_i = U_{\delta_i,0}\(\tfrac{r_0}{2}\) &\text{if } r_0 \le d_g(x,\xi_i) \le 2r_0.
\end{cases}
\end{equation}
We write
\begin{multline*}
\int_{B^g_{r_0/2}(\xi_j)} \mcu_j^{2^*-2}\delta_j\frac{\pa \mcu_j}{\pa \delta_j} \mcv_i dv_g = \int_{B^g_{r_0/2}(\xi_j) \cap \big(B^g_{r_0}(\xi_i) \setminus B^g_{r_0/2}(\xi_i)\big)} \mcu_j^{2^*-2}\delta_j\frac{\pa \mcu_j}{\pa \delta_j} \mcv_i dv_g\\
+ U_{\delta_i,0}\(\tfrac{r_0}{2}\) \int_{B^g_{r_0/2}(\xi_j) \cap \big(B^g_{2r_0}(\xi_i) \setminus B^g_{r_0}(\xi_i)\big)} \mcu_j^{2^*-2}\delta_j\frac{\pa \mcu_j}{\pa \delta_j} dv_g =: J_1+J_2.
\end{multline*}
The definition of the function $\frac{\pa \mcu_j}{\pa \delta_j}$ and \eqref{eq:mcvi} yield
\begin{align*}
J_1 &\ge U_{\delta_i,0}(r_0) \int_{\big(B^g_{r_0/2}(\xi_j) \setminus B^g_{\delta_j}(\xi_j)\big) \cap \big(B^g_{r_0}(\xi_i) \setminus B^g_{r_0/2}(\xi_i)\big)} \mcu_j^{2^*-2}\delta_j\frac{\pa \mcu_j}{\pa \delta_j} dv_g \\
&\ + U_{\delta_i,0}\(\tfrac{r_0}{2}\) \int_{B^g_{\delta_j}(\xi_j) \cap \big(B^g_{r_0}(\xi_i) \setminus B^g_{r_0/2}(\xi_i)\big)} \mcu_j^{2^*-2}\delta_j\frac{\pa \mcu_j}{\pa \delta_j} dv_g.
\end{align*}
By direct computations,
\begin{align}
J_1+J_2 &\ge U_{\delta_i,0}(r_0) \int_{B^g_{r_0/2}(\xi_j) \setminus B^g_{\delta_j}(\xi_j)} \mcu_j^{2^*-2}\delta_j\frac{\pa \mcu_j}{\pa \delta_j} dv_g + U_{\delta_i,0}\(\tfrac{r_0}{2}\) \int_{B^g_{\delta_j}(\xi_j)} \mcu_j^{2^*-2}\delta_j\frac{\pa \mcu_j}{\pa \delta_j} dv_g \nonumber \\
&\gtrsim \delta_j^{\frac{N-2}{2}} \left[U_{\delta_i,0}(r_0) \int_{B_1^c(0)} U^{2^*-2}Z^0 + U_{\delta_i,0}\(\tfrac{r_0}{2}\) \int_{B_1(0)}U^{2^*-2}Z^0\right] \label{eq:J1J2} \\
&\simeq (\delta_i\delta_j)^{\frac{N-2}{2}} \alpha_N^{2^*} \times \left\{\!\begin{aligned}
&\frac{1}{r_0} \cdot \frac{1}{30}(2+\sqrt{2})-\frac{2}{r_0} \cdot \frac{\sqrt{2}}{30} &\text{if } N=3\\[1ex]
&\frac{1}{r_0^2} \cdot \frac{5}{48} - \frac{4}{r_0^2} \cdot \frac{1}{48} &\text{if } N=4\\[1ex]
&\frac{1}{r_0^3} \cdot \frac{1}{140}(12+\sqrt{2}) - \frac{8}{r_0^3} \cdot \frac{\sqrt{2}}{140} &\text{if } N=5
\end{aligned}\right\} \simeq (\delta_i\delta_j)^{\frac{N-2}{2}} \simeq q_{ij}. \nonumber
\end{align}

\medskip \noindent \textbf{Case 3. ($d_g(\xi_i,\xi_j) \le r_0$):} We write
\begin{align*}
&\ \int_{B^g_{r_0/2}(\xi_j)} \mcu_j^{2^*-2}\delta_j\frac{\pa \mcu_j}{\pa \delta_j} \mcv_i dv_g \\
&= \int_{B^g_{r_0/2}(\xi_j)} \mcu_j^{2^*-2} \delta_j\frac{\pa \mcu_j}{\pa \delta_j} \mcu_i dv_g +
\int_{B^g_{r_0/2}(\xi_j)} \mcu_j^{2^*-2} \delta_j\frac{\pa \mcu_j}{\pa \delta_j} (1-\chi(d_g(x,\xi_i))) \left[U_{\delta_i,0}\(\tfrac{r_0}{2}\)-\mcu_i\right] (dv_g)_x \\
&=: J_3 + J_4.
\end{align*}
By arguing as in the derivation of \cite[(F16)]{B}, we observe
\begin{equation}\label{eq:J3}
\begin{aligned}
J_3 &= \alpha_N \int_{\{|y| \le \frac{r_0}{2\delta_j}\}} \(U^{2^*-2}Z^0\)(y) \left[\frac{\delta_i}{\delta_j} + \frac{d_g\big(\exp_{\xi_j}^g(\delta_jy), \xi_i\big)^2}{\delta_i\delta_j}\right]^{-\frac{N-2}{2}} dy + o(q_{ij}) \\
&= \left[\frac{\delta_i}{\delta_j} + \frac{d_g(\xi_i,\xi_j)^2}{\delta_i\delta_j}\right]^{-\frac{N-2}{2}} \alpha_N \int_{\R^N} U^{2^*-2}Z^0 + o(q_{ij}) \simeq q_{ij}.
\end{aligned}
\end{equation}
If $d_g(\xi_i,\xi_j) \le \frac{r_0}{8}$, then $d_g(x,\xi_i) \le \frac{3r_0}{8}$ so that $1-\chi(d_g(x,\xi_i)) = 0$ for all $x \in B^g_{r_0/4}(\xi_j)$. It follows that
\[|J_4| \lesssim \delta_i^{\frac{N-2}{2}} \int_{B^g_{r_0/2}(\xi_j) \setminus B^g_{r_0/4}(\xi_j)} \mcu_j^{2^*-1} dv_g \simeq \delta_i^{\frac{N-2}{2}}\delta_j^{\frac{N+2}{2}} = o(q_{ij}),\]
and so $J_3+J_4 \gtrsim q_{ij}$. Thus let us assume that $\frac{r_0}{8} \le d_g(\xi_i,\xi_j) \le r_0$. By the second line of \eqref{eq:J3}, we have
\[J_3 \ge U_{\delta_i,0}(r_0) \delta_j^{\frac{N-2}{2}} \int_{\R^N} U^{2^*-2}Z^0 + o(q_{ij}).\]
Noticing that $d_g(x,\xi_i) \le \frac{3r_0}{2}$ for all $x \in B^g_{r_0/2}(\xi_j)$, we decompose $J_4$ as
\[J_4 = \int_{B^g_{r_0/2}(\xi_j) \cap \big(B^g_{r_0}(\xi_i) \setminus B^g_{r_0/2}(\xi_i)\big)} \cdots + \int_{B^g_{r_0/2}(\xi_j) \cap \big(B^g_{3r_0/2}(\xi_i) \setminus B^g_{r_0}(\xi_i)\big)} \cdots=: J_{41}+J_{42}.\]
Then, it holds that
\[J_{41} \ge \bigg[U_{\delta_i,0}\(\tfrac{r_0}{2}\) - U_{\delta_i,0}(r_0)\bigg] \int_{B^g_{\delta_j}(\xi_j) \cap \big(B^g_{r_0}(\xi_i) \setminus B^g_{r_0/2}(\xi_i)\big)} \mcu_j^{2^*-2} \delta_j\frac{\pa \mcu_j}{\pa \delta_j} dv_g\]
and
\[J_{42} \ge \bigg[U_{\delta_i,0}\(\tfrac{r_0}{2}\) - U_{\delta_i,0}\(\tfrac{3r_0}{2}\)\bigg] \int_{B^g_{\delta_j}(\xi_j) \cap \big(B^g_{3r_0/2}(\xi_i) \setminus B^g_{r_0}(\xi_i)\big)} \mcu_j^{2^*-2} \delta_j\frac{\pa \mcu_j}{\pa \delta_j} dv_g.\]
Therefore,
\begin{align*}
\begin{medsize}
\displaystyle J_3+J_4 \,
\end{medsize}
&\begin{medsize}
\displaystyle \ge \delta_j^{\frac{N-2}{2}} \left[U_{\delta_i,0}(r_0)\int_{B^c_1(0)} U^{2^*-2}Z^0 + \left\{U_{\delta_i,0}(r_0) + U_{\delta_i,0}\(\tfrac{r_0}{2}\) -U_{\delta_i,0}\(\tfrac{3r_0}{2}\)\right\} \int_{B_1(0)} U^{2^*-2}Z^0\right] + o(q_{ij})
\end{medsize} \\
&\begin{medsize}
\displaystyle \gtrsim q_{ij}
\end{medsize}
\end{align*}
where the last inequality is justified as in \eqref{eq:J1J2}. This completes the proof of \eqref{neqij}.
\end{proof}

\begin{lemma}\label{le2.6}
For any $j \in \{1,\ldots,\nu\}$, we have
\begin{equation}\label{eq:I4Z}
\int_M \I_4\mcz^0_j dv_g = o\Big(\max_{\ell=1,\ldots,\nu} \delta_{\ell}^{\frac{N-2}{2}}\Big).
\end{equation}
\end{lemma}
\begin{proof}
By means of \eqref{i4}, \eqref{cutt}, $|\mcz^0_j| \lesssim \mcv_j$, Young's inequality, and Lemma \ref{a4}, we see
\begin{align*}
\left|\int_M \I_4\mcz^0_j dv_g\right| &\lesssim \sum_{i=1}^{\nu}\int_M \mcv_i^2 dv_g+ \sum_{i=1}^{\nu} \int_{M\setminus B^g_{r_0/2}(\xi_i)} \(\mcv_i^{2^*-1} + \mcv_i + |\nabla_g \mcv_i|_g\) \mcv_j dv_g \\
&= o\Big(\max_{\ell}\delta_{\ell}^{\frac{N-2}{2}}\Big). \qedhere
\end{align*}
\end{proof}

We are now in position to conclude the proof of Proposition \ref{lou0}.
\begin{proof}[Completion of the proof of Proposition \ref{lou0}]
Choose any $j \in \{1,\ldots,\nu\}$. Since $-\Delta Z^0 = (2^*-1)U^{2^*-2}Z^0$ in $\R^N$, it holds that
\begin{equation}\label{lina}
\left\|\mcl_g\mcz^0_j-(2^*-1)\mcv_j^{2^*-2}\mcz^0_j\right\|_{L^{\frac{2N}{N+2}}(M)} = \mco\Big(\delta_j^{\frac{N-2}{2}}\Big).
\end{equation}
By \eqref{lina} and \eqref{teu},
\begin{equation}\label{lin}
\begin{aligned}
&\ \left|\int_M \bigg[\mcl_g\rho - (2^*-1)\bigg(u_0+\sum_{i=1}^{\nu}\mcv_i\bigg)^{2^*-2}\rho\bigg] \mcz^0_j dv_g\right| \\
&\lesssim \left[\left\|\mcl_g\mcz^0_j-(2^*-1)\mcv_j^{2^*-2}\mcz^0_j\right\|_{L^{\frac{2N}{N+2}}(M)} + \Big(\mcq + \max_{\ell}\delta_{\ell}^{\frac12}\Big)\right] \|\rho\|_{H^1(M)} \\
&\lesssim \Big(\mcq + \max_{\ell}\delta_{\ell}^{\frac12}\Big) \|\rho\|_{H^1(M)}.
\end{aligned}
\end{equation}
It follows from \eqref{i1r} that
\begin{equation}\label{eq:fI1}
\left|\int_M \(f+\I_1[\rho]\)\mcz^0_j dv_g\right| \lesssim \|f\|_{H^{-1}(M)}+\|\rho\|_{H^1(M)}^2.
\end{equation}
Inserting \eqref{eq:I2Z}, \eqref{eq:I3Z}, \eqref{eq:I4Z}, \eqref{lin}, and \eqref{eq:fI1} into \eqref{tezk1} yields
\begin{multline*}
\bigg|\mfa_Nu_0(\xi_j)\delta_j^{\frac{N-2}{2}} + (2^*-1) \sum_{i \ne j} \int_{B^g_{r_0/2}(\xi_j)} \mcu_j^{2^*-2}\delta_j\frac{\pa \mcu_j}{\pa \delta_j} \mcv_i dv_g\bigg| \\
\lesssim \|f\|_{H^{-1}(M)} + \|\rho\|_{H^1(M)}^2 + o(1)\|\rho\|_{H^1(M)} + o\Big(\mcq + \max_{\ell}\delta_{\ell}^{\frac{N-2}{2}}\Big),
\end{multline*}
and so by Proposition \ref{pr2.2},
\begin{equation}\label{nece}
\bigg|\mfa_Nu_0(\xi_j)\delta_j^{\frac{N-2}{2}} + (2^*-1) \sum_{i \ne j} \int_{B^g_{r_0/2}(\xi_j)} \mcu_j^{2^*-2}\delta_j\frac{\pa \mcu_j}{\pa \delta_j} \mcv_i dv_g\bigg|
\lesssim \|f\|_{H^{-1}(M)} + o\Big(\mcq + \max_{\ell}\delta_{\ell}^{\frac{N-2}{2}}\Big)
\end{equation}
for all $j \in \{1,\ldots,\nu\}$. Additionally, we observe from Lemma \ref{a22} that
\begin{equation}\label{utaj}
\bigg|\int_{B^g_{r_0/2}(\xi_j)} \mcu_j^{2^*-2}\delta_j\frac{\pa \mcu_j}{\pa \delta_j} \mcv_i dv_g\bigg| \lesssim q_{ij} \quad \text{if } i \ne j.
\end{equation}

Without loss of generality, we may assume that $\delta_1 \ge \delta_2 \ge \cdots \ge \delta_{\nu}$. Setting an induction hypothesis
\[(\textup{\textbf{P}}_l): \quad \sum_{j=l}^{\nu}\sum_{i=1}^{j-1} q_{ij} \lesssim \|f\|_{H^{-1}(M)} + o\Big(\mcq + \max_{\ell}\delta_{\ell}^{\frac{N-2}{2}}\Big),\]
we use mathematical induction on $l = \nu, \nu-1,\ldots,2$: By employing \eqref{neqij}, \eqref{nece}, \eqref{utaj}, and the positivity of $\mfa_N$ and $u_0$, we can argue as in the proof of \cite[Lemma 2.3]{DSW} to achieve
\begin{equation}\label{fi1}
\mcq \lesssim \|f\|_{H^{-1}(M)} + o\Big(\max_{\ell}\delta_{\ell}^{\frac{N-2}{2}}\Big).
\end{equation}
Furthermore, \eqref{nece}, \eqref{utaj}, and \eqref{fi1} imply
\begin{equation}\label{fi2}
\begin{aligned}
\max_{\ell}\delta_{\ell}^{\frac{N-2}{2}} &= \delta_1^{\frac{N-2}{2}} \\
&\lesssim \bigg|\sum_{i=2}^{\nu} \int_{B^g_{r_0/2}(\xi_1)}\mcu_1^{2^*-2}\delta_1\frac{\pa \mcu_1}{\pa \delta_1} \mcv_idv_g\bigg| + \|f\|_{H^{-1}(M)} + o\Big(\mcq + \max_{\ell}\delta_{\ell}^{\frac{N-2}{2}}\Big) \\
&\lesssim \mcq + \|f\|_{H^{-1}(M)} + o\Big(\mcq + \max_{\ell}\delta_{\ell}^{\frac{N-2}{2}}\Big) \lesssim \|f\|_{H^{-1}(M)} + o\Big(\max_{\ell}\delta_{\ell}^{\frac{N-2}{2}}\Big).
\end{aligned}
\end{equation}

Proposition \ref{lou0} is a consequence of \eqref{fi1} and \eqref{fi2}. This concludes the proof.
\end{proof}

\section{The case $3 \le N \le 5$ and $u_0 = 0$}\label{se3}
This section is devoted to the proof of Theorem \ref{th1.3}. Throughout this section, we always assume that $3 \le N \le 5$ and $u_0 = 0$ on $M$.

\subsection{Setting of the problem}
Given $\xi \in M$, we recall the smooth positive function $\Lambda_{\xi}$ on $M$ that raises conformal normal coordinates around $\xi$; see \eqref{dvg}.
Using $\Lambda_{\xi}(\xi)=1$, $g_{\xi} = \Lambda_{\xi}^{4/(N-2)}g$, and the conformal covariance property of the conformal Laplacian $\mcl_g$ on $(M,g)$,
\begin{equation}\label{cova}
\mcl_{g_{\xi}}(\phi)=\Lambda_{\xi}^{-(2^*-1)}\mcl_g(\Lambda_{\xi}\phi) \quad \text{for all } \phi \in C^2(M),
\end{equation}
we find that
\begin{equation}\label{gg}
G_g(\cdot,\xi)=\Lambda_{\xi}(\cdot)G_{g_{\xi}}(\cdot,\xi)
\end{equation}
where $G_g$ is the Green's function of $\mcl_g$. In \cite[Lemma 6.4]{LP}, it was shown that
\begin{equation}\label{gx}
G_{g_{\xi}}\big(\exp_{\xi}^{g_{\xi}} y,\xi\big) = \ga_N^{-1}|y|^{2-N} + A_g(\xi) + \mco(|y|) \quad \text{$C^1$-uniformly in $y$ and $\xi$}
\end{equation}
in conformal normal coordinates $y$ around $\xi$. The quantity $A_g(\xi) \in \R$ (called the mass at $\xi$) is determined by $(M,g)$ and $\xi$, and positive by the positive mass theorem in \cite{SY1}. Besides, the map $\xi \mapsto A_g(\xi)$ is smooth.

For $i \in \{1,\ldots,\nu\}$, let $\mcv_i = \mcv_{\delta_i,\xi_i}$ be the function in \eqref{vi}.
By Assumption \ref{assum} and \eqref{eq:bubblel}, there exist $(\delta_1,\ldots,\delta_{\nu}, \xi_1,\ldots,\xi_{\nu}) \subset (0,\infty)^{\nu} \times M^{\nu}$ and $\vep_1 > 0$ small such that $\vep_1 \to 0$ as $\vep_0 \to 0$,
\[\bigg\|u-\sum_{i=1}^{\nu}\mcv_i\bigg\|_{H^1(M)} = \inf\left\{\bigg\|u-\sum_{i=1}^{\nu} \mcv_{\tde_i,\txi_i}\bigg\|_{H^1(M)}: \(\tde_i,\txi_i\) \in (0,\infty) \times M,\ i=1,\ldots,\nu\right\} \le \vep_1,\]
and \eqref{eq:bubblei} holds. Setting $\rho = u-\sum_{i=1}^{\nu}\mcv_i$ and $f = \mcl_gu-u^{2^*-1}$, we have
\begin{equation}\label{eqrho2}
\begin{cases}
\displaystyle \mcl_g\rho - (2^*-1)\bigg(\sum_{i=1}^{\nu}\mcv_i\bigg)^{2^*-2}\rho = f+\II_1[\rho]+\II_2+\II_3 \quad \text{on } M,\\
\displaystyle \big\langle \rho,\mcz^k_i \big\rangle_{H^1(M)} = 0 \quad \text{for } i=1,\ldots,\nu \text{ and } k=0,\ldots,N
\end{cases}
\end{equation}
where $\mcz^0_i = \delta_i \frac{\pa\mcv_i}{\pa\delta_i}$, $\mcz^k_i = \delta_i \frac{\pa\mcv_i}{\pa\xi_i^k}$ for $k=1,\ldots,N$,
\[\II_1[\rho] := \bigg(\sum_{i=1}^{\nu}\mcv_i+\rho\bigg)^{2^*-1} - \bigg(\sum_{i=1}^{\nu}\mcv_i\bigg)^{2^*-1} - (2^*-1)\bigg(\sum_{i=1}^{\nu}\mcv_i\bigg)^{2^*-2}\rho,\]
\[\II_2 := \bigg(\sum_{i=1}^{\nu} \mcv_i\bigg)^{2^*-1} - \sum_{i=1}^{\nu}\mcv_i^{2^*-1}, \quad \text{and} \quad \II_3 := \sum_{i=1}^{\nu} \(-\mcl_g\mcv_i+\mcv_i^{2^*-1}\).\]
To prove Theorem \ref{th1.3}, it is sufficient to verify that
\begin{equation}\label{end}
\|\rho\|_{H^1(M)} \lesssim \|f\|_{H^{-1}(M)}
\end{equation}
provided $\vep_0 > 0$ small. Since the proof of \eqref{end} is rather parallel to that of \eqref{rf}, we will minimize the overlaps and focus on the distinct parts.

\subsection{Preliminary computations}
The following two results are analogies of Lemma \ref{le2p} and Proposition \ref{41}, respectively.
\begin{lemma}\label{le3p}
Assume that $i,j \in \{1,\ldots,\nu\}$ and $k,l \in \{0,1,\ldots,N\}$. We have
\[\la \mcv_i,\mcv_i \ra_{H^1(M)} = \int_{\R^N} U^{2^*} + o\Big(\delta_i^{\frac{N-2}{2}}\Big), \quad \big\langle \mcz^k_i,\mcv_i \big\rangle_{H^1(M)} = o\Big(\delta_i^{\frac{N-2}{2}}\Big),\]
\[\big\langle \mcz^k_i,\mcz^l_i \big\rangle_{H^1(M)} = \big\|Z^k\big\|_{\dot{H}^1(\R^N)}^2 \delta^{kl} + o\Big(\delta_i^{\frac{N-2}{2}}\Big),\]
and
\[\left|\la \mcv_i,\mcv_j \ra_{H^1(M)}\right| + \left|\big\langle \mcz^k_i,\mcv_j \big\rangle_{H^1(M)}\right| + \left|\big\langle \mcz^k_i,\mcz^l_j \big\rangle_{H^1(M)}\right| = \mco(q_{ij}) + o\Big(\max_{\ell = 1,\ldots,\nu}\delta_{\ell}^{\frac{N-2}{2}}\Big)\]
provided $i \ne j$.
\end{lemma}
\begin{proof}
Following the proof of \cite[Lemma 3]{EPV}, we can show that
\[\mcz^k_i(x) = \begin{cases}
\displaystyle (N-2)\alpha_N \frac{\delta_i^{\frac{N}{2}}y^k}{(\delta_i^2+|y|^2)^{\frac{N}{2}}} + \mco\(\frac{\delta_i^{\frac{N}{2}}|y|}{(\delta_i^2+|y|^2)^{\frac{N-2}{2}}}\)
&\text{if } x=\exp^{g_{\xi_i}}_{\xi_i}(y) \in B^{g_{\xi_i}}_{\frac{r_0}{2}}(\xi_i), \\
\displaystyle \mco\Big(\delta_i^{\frac{N-2}{2}}\Big), &\text{if } x \in M\setminus B^{g_{\xi_i}}_{\frac{r_0}{2}}(\xi_i)
\end{cases}\]
for $i=1,\ldots,\nu$ and $k=1,\ldots,N$. Once we have this, the rest of the proof is standard.
\end{proof}	
\begin{prop}\label{co}
Let
\begin{equation}\label{eq:wteperp}
E^{\perp} = \left\{\vrh \in H^1(M): \la \vrh,\mcv_i \ra_{H^1(M)} = \big\langle \vrh,\mcz^k_i \big\rangle_{H^1(M)} = 0 \text{ for } i=1,\ldots,\nu,\ k=0,1,\ldots,N\right\}.
\end{equation}
Then there exists a constant $c_0 \in(0,1)$ such that
\[(2^*-1) \int_M \bigg(\sum_{i=1}^{\nu}\mcv_i\bigg)^{2^*-2}\vrh^2 dv_g \le c_0\|\vrh\|_{H^1(M)}^2 \quad \text{for any } \vrh \in E^{\perp}.\]
\end{prop}
\begin{proof}
The proof is similar to that of Proposition \ref{41}, so we omit it.
\end{proof}
We will also need estimates for the $L^{2N/(N+2)}(M)$-norm of $\II_2$ and $\II_3$.
\begin{lemma}\label{lem3.3}
We have
\[\|\II_2\|_{L^{\frac{2N}{N+2}}(M)} + \|\II_3\|_{L^{\frac{2N}{N+2}}(M)} \lesssim \mcq + \max_{\ell=1,\ldots,\nu}\delta_{\ell}^{N-2}.\]
\end{lemma}
\begin{proof}
It is straightforward to check
\[\|\II_2\|_{L^{\frac{2N}{N+2}}(M)} \lesssim \sum_{i \ne j} \left\|\mcv_i^{2^*-2}\mcv_{j}\right\|_{L^{\frac{2N}{N+2}}(M)} \lesssim \mcq.\]
From now on, we are devoted to estimating the $L^{2N/(N+2)}(M)$-norm of $\II_3$. Fixing $i \in \{1,\ldots,\nu\}$, we examine three different cases determined by the distance between a point $x \in M$ and $\xi_i$.

\medskip \noindent \textbf{Case 1. ($d_{g_{\xi_i}}(x,\xi_i) \ge r_0$):} We know that $\mcv_i = \alpha_N\ga_N \delta_i^{(N-2)/2} G_g(x,\xi_i)$ and so
\begin{equation}\label{ca13}
-\mcl_g\mcv_i +\mcv_i^{2^*-1} =-\alpha_N\ga_N \delta_i^{\frac{N-2}{2}} \Lambda_{\xi_i}^{2^*-1} \mcl_{g_{\xi_i}} G_{g_{\xi_i}}(x,\xi_i) +\mcv_i^{2^*-1} = \mcv_i^{2^*-1} = \mco\Big(\delta_i^{\frac{N+2}{2}}\Big)
\end{equation}
where we used $\mcl_{g_{\xi_i}}G_{g_{\xi_i}}(x,\xi_i) = 0$ for the second equality.

\medskip \noindent \textbf{Case 2. ($d_{g_{\xi_i}}(x,\xi_i) \le \frac{r_0}{2}$):} We define
\begin{equation}\label{eq:Fi}
F_i(x) = d_{g_{\xi_i}}(x,\xi_i)^{N-2} \mcu_i^{g_{\xi_i}}(x) \quad \text{for } x = \exp_{\xi_i}^{g_{\xi_i}}y \in B_{r_0/2}^{g_{\xi_i}}(\xi_i)
\end{equation}
where $\mcu_i^{g_{\xi_i}} := \mcu^{g_{\xi_i}}_{\delta_i,\xi_i}$ is defined by \eqref{udx} so that $\mcv_i(x) = \ga_NG_g(x,\xi_i)F_i(x)$. Since $F_i(\xi_i)=0$ and $\mcl_{g_{\xi_i}}G_{g_{\xi_i}}(\cdot,\xi_i) = \delta_{\xi_i}$, it holds that
\begin{multline}\label{eq:mclv}
-\mcl_g\mcv_i+\mcv_i^{2^*-1} = \Lambda_{\xi_i}^{2^*-1} \left[\(\ga_NG_{g_{\xi_i}}(\cdot,\xi_i)F\)^{2^*-1}+\ga_NG_{g_{\xi_i}}(\cdot,\xi_i) \Delta_{g_{\xi_i}} F \right. \\
\left. + 2\ga_N \la\nabla_{g_{\xi_i}} G_{g_{\xi_i}}(\cdot,\xi_i), \nabla_{g_{\xi_i}} F\ra_{g_{\xi_i}}\right].
\end{multline}
Notice that if $x = \exp_{\xi}^{g_{\xi}}y$ for a given $\xi \in M$, $v(x)=u(y)$, $u$ is radially symmetric, and $r = |y|$ is the radial variable of polar coordinates, then \eqref{dvg} gives
\begin{equation}\label{gxvu}
\Delta_{g_{\xi}}v = \Delta u + \frac{\pa_r\sqrt{|g_{\xi}|}}{\sqrt{|g_{\xi}|}}\pa_ru = \Delta u + \mco\(r^{\theta-1}|\pa_ru|\) \quad \text{around } 0;
\end{equation}
refer to \cite[(2.18)]{GM} for more explanations. Direct computations using \eqref{gxvu} show that
\begin{equation}\label{com}
\Delta_{g_{\xi_i}} F_i(x) = -\delta_i^{-\frac{N+2}{2}}|y|^{N-2} U\(\frac{y}{\delta_i}\)^{2^*-1} + 2(N-2)^2\alpha_N \frac{\delta_i^{\frac{N+2}{2}}|y|^{N-4}}{(\delta_i^2+|y|^2)^{\frac{N}{2}}} + \mco\(\frac{\delta_i^{\frac{N+2}{2}}|y|^{\theta+N-4}}{(\delta_i^2+|y|^2)^{\frac{N}{2}}}\)
\end{equation}
and
\begin{equation}\label{com1}
\begin{aligned}
\begin{medsize}
\displaystyle \la\nabla_{g_{\xi_i}} G_{g_{\xi_i}}(x,\xi_i),\nabla_{g_{\xi_i}} F_i(x)\ra_{g_{\xi_i}}
\end{medsize}
&\begin{medsize}
\displaystyle = \pa_r G_{g_{\xi_i}}(x,\xi_i) \pa_r F_i(x)
\end{medsize} \\
&\begin{medsize}
\displaystyle = -(N-2)^2\alpha_N\ga_N^{-1}|y|^{1-N} \frac{\delta_i^{\frac{N+2}{2}}|y|^{N-3}}{(\delta_i^2+|y|^2)^{\frac{N}{2}}} + \mco\(\frac{\delta_i^{\frac{N+2}{2}}|y|^{N-3}}{(\delta_i^2+|y|^2)^{\frac{N}{2}}}\).
\end{medsize}
\end{aligned}
\end{equation}
Plugging \eqref{com}, \eqref{com1}, and \eqref{gx} into \eqref{eq:mclv}, we obtain
\begin{multline}\label{ca23}
\(-\mcl_g\mcv_i+\mcv_i^{2^*-1}\)(x) = \Lambda_{\xi_i}^{2^*-1}(x) \alpha_N\ga_NA_g(\xi_i) \delta_i^{\frac{N+2}{2}} \\
\times \left[4N \frac{|y|^{N-2}}{(\delta_i^2+|y|^2)^{\frac{N+2}{2}}} + 2(N-2)^2 \frac{|y|^{N-4}}{(\delta_i^2+|y|^2)^{\frac{N}{2}}} + \mco\(\frac{|y|^{N-3}}{(\delta_i^2+|y|^2)^{\frac{N}{2}}}\)\right].
\end{multline}
Thus
\begin{equation}\label{ca232}
\int_{B^{g_{\xi_i}}_{r_0/2}(\xi_i)} \left|\mcl_g\mcv_i-\mcv_i^{2^*-1}\right|^{\frac{2N}{N+2}} dv_g
\lesssim \int_{\{|y| \le \frac{r_0}{2}\}} \left[\frac{\delta_i^{\frac{N+2}{2}}|y|^{N-4}}{(\delta_i^2+|y|^2)^{\frac{N}{2}}}\right]^{\frac{2N}{N+2}} dy \lesssim \delta_i^{\frac{2N(N-2)}{N+2}}
\end{equation}
where we used $3 \le N \le 5$ and $dv_g = \Lambda_{\xi_i}^{-2^*} dv_{g_{\xi_i}}$.

\medskip \noindent \textbf{Case 3. ($\frac{r_0}{2} \le d_{g_{\xi_i}}(x,\xi_i) \le r_0$):} We rewrite
\[\mcv_i(x) = \chi(d_{g_{\xi_i}}(x,\xi_i)) \ga_NG_g(x,\xi_i) \left[F_i(x)-\alpha_N\delta_i^{\frac{N-2}{2}}\right] + \alpha_N\ga_N \delta_i^{\frac{N-2}{2}} G_g(x,\xi_i).\]
Then the definition of $G_g$ leads to
\begin{equation}\label{331}
-\alpha_N\ga_N \delta_i^{\frac{N-2}{2}}\mcl_gG_g(x,\xi_i) +\mcv_i^{2^*-1} = \mco\Big(\delta_i^{\frac{N+2}{2}}\Big).
\end{equation}
Moreover, it follows from \eqref{com} and \eqref{com1} that
\begin{equation}\label{332}
-\mcl_g\(\chi(d_{g_{\xi_i}}(x,\xi_i))G_g(x,\xi_i) \left[F_i(x)-\alpha_N\delta_i^{\frac{N-2}{2}}\right]\) = \mco\Big(\delta_i^{\frac{N+2}{2}}\Big).
\end{equation}
Therefore,
\begin{equation}\label{ca3}
-\mcl_g\mcv_i+\mcv_i^{2^*-1} =\mco\Big(\delta_i^{\frac{N+2}{2}}\Big).
\end{equation}

\medskip
From \eqref{ca13}, \eqref{ca232}, and \eqref{ca3}, we conclude that
\[\|\II_3\|_{L^{\frac{2N}{N+2}}(M)} \lesssim \max_{\ell}\delta_{\ell}^{N-2}.\]
This completes the proof.
\end{proof}

\subsection{Proof of Theorem \ref{th1.3}}
We adopt the approach used in Subsection \ref{su2.3}. Specifically, we decompose $\rho=u-\sum_{i=1}^{\nu}\mcv_i$ as follows:
\[\rho = \rho_1 + \sum_{i=1}^{\nu} \beta_i\mcv_i + \sum_{i=1}^{\nu} \sum_{k=0}^N \beta_i^k\mcz^k_i \quad \text{for some } \beta_i,\, \beta_i^k \in \R,\, \rho_1 \in E^{\perp}\]
where $E^{\perp}$ is the space defined in \eqref{eq:wteperp}.

By employing Lemma \ref{le3p}, Proposition \ref{co}, and Lemma \ref{lem3.3}, one can prove the following proposition as we did for Proposition \ref{pr2.2}. We omit its proof.
\begin{prop}\label{pr3.2}
Let $\mcq$ be the quantity in \eqref{rq}. It holds that
\begin{equation}\label{eq:3.2}
\|\rho\|_{H^1(M)} \lesssim \|f\|_{H^{-1}(M)} + \mcq + \max_{\ell=1,\ldots,\nu}\delta_{\ell}^{N-2}.
\end{equation}
\end{prop}
\noindent The term $\max_{\ell=1,\ldots,\nu} \delta_{\ell}^{N-2}$ in \eqref{eq:3.2} stems from Lemma \ref{lem3.3}. The quantities $\mcq$ and $\max_{\ell=1,\ldots,\nu} \delta_{\ell}^{N-2}$ are non-comparable.

\medskip
Owing to Proposition \ref{pr3.2}, deducing the subsequent proposition will lead us to establish \eqref{end}.
\begin{prop}\label{3.3}
It holds that
\[\mcq + \max_{\ell=1,\ldots,\nu}\delta_{\ell}^{N-2} \lesssim \|f\|_{H^{-1}(M)}.\]
\end{prop}
\noindent Fix $j \in \{1,\ldots,\nu\}$. By testing \eqref{eqrho2} with $\mcz^0_j$, we obtain
\begin{equation}\label{tezk2}
\begin{aligned}
\int_M \II_2\mcz^0_j dv_g + \int_M \II_3\mcz^0_j dv_g &= -\int_M f\mcz^0_j dv_g - \int_M \II_1[\rho]\mcz^0_j dv_g \\
&\ + \int_M \left[\mcl_g\mcz^0_j - (2^*-1)\bigg(\sum_{i=1}^{\nu}\mcv_i\bigg)^{2^*-2}\mcz^0_j\right]\rho dv_g.
\end{aligned}
\end{equation}
In Lemmas \ref{le35} and \ref{le36}, we evaluate two integrals on the left-hand side of \eqref{tezk2}, respectively.

\begin{lemma}\label{le35}
For any $j \in \{1,\ldots,\nu\}$, we have
\begin{equation}\label{eq:le35}
\int_M \II_2\mcz^0_j dv_g = (2^*-1) \sum_{i \ne j} \int_{B^{g_{\xi_j}}_{r_0/2}(\xi_j)} \Lambda^{2^*-1}_{\xi_j}\(\mcu_j^{g_{\xi_j}}\)^{2^*-2} \delta_j\frac{\pa \mcu_j^{g_{\xi_j}}}{\pa\delta_j} \mcv_i dv_g + o(\mcq).
\end{equation}
Moreover, if $\delta_i \ge \delta_j$ for some indices $1 \le i \ne j \le \nu$, then
\begin{equation}\label{neqij2}
\sum_{i \ne j} \int_{B^{g_{\xi_j}}_{r_0/2}(\xi_j)} \Lambda^{2^*-1}_{\xi_j}\(\mcu_j^{g_{\xi_j}}\)^{2^*-2} \delta_j\frac{\pa \mcu_j^{g_{\xi_j}}}{\pa\delta_j}\mcv_i dv_g \gtrsim q_{ij}
\end{equation}
provided $q_{ij}$ in \eqref{rq} small.
\end{lemma}
\begin{proof}
By \eqref{gg}, \eqref{gx}, and \eqref{iqu},
\begin{multline*}
\left|\sum_{i \ne j} \int_{B^{g_{\xi_j}}_{r_0/2}(\xi_j)} \left[\left\{\ga_NG_g(\cdot,\xi_j)d_{g_{\xi_j}}(\cdot,\xi_j)^{N-2}\right\}^{2^*-1}-\Lambda^{2^*-1}_{\xi_j}\right] \(\mcu_j^{g_{\xi_j}}\)^{2^*-2}\delta_j\frac{\pa \mcu_j^{g_{\xi_j}}}{\pa\delta_j}\mcv_i dv_g\right| \\
\lesssim \sum_{i \ne j} \int_{B^{g_{\xi_j}}_{r_0/2}(\xi_j)} d_{g_{\xi_j}}(\cdot,\xi_j)^{N-2}\(\mcu_j^{g_{\xi_j}}\)^{2^*-1}\mcv_i dv_g \lesssim \sum_{i \ne j} \delta_j^{\frac{N-2}{2}} \int_{B^{g_{\xi_j}}_{r_0/2}(\xi_j)} \mcv_j^{2^*-2}\mcv_i dv_g = o(\mcq)
\end{multline*}
where the proof of \eqref{22} in Appendix \ref{subsec:tech} validates the equality on the last line. Thus we can argue as in Lemma \ref{le29} to deduce
\begin{align*}
&\ \frac{1}{2^*-1} \int_M \II_2\mcz_j^0 dv_g \\
&= \sum_{i \ne j} \int_{B^{g_{\xi_j}}_{r_0/2}(\xi_j)} \left[\ga_NG_g(x,\xi_j)d_{g_{\xi_j}}(x,\xi_j)^{N-2}\right]^{2^*-1} \(\mcu_j^{g_{\xi_j}}\)^{2^*-2}\delta_j\frac{\pa \mcu_j^{g_{\xi_j}}}{\pa\delta_j}\mcv_i dv_g + o(\mcq) \\
&= \sum_{i \ne j} \int_{B^{g_{\xi_j}}_{r_0/2}(\xi_j)} \Lambda^{2^*-1}_{\xi_j} \(\mcu_j^{g_{\xi_j}}\)^{2^*-2}\delta_j\frac{\pa \mcu_j^{g_{\xi_j}}}{\pa\delta_j}\mcv_i dv_g + o(\mcq),
\end{align*}
which is \eqref{eq:le35}.

\medskip
We next derive \eqref{neqij2} provided $d_g(\xi_i,\xi_j) \le r_0$. It is easier to handle the case $d_g(\xi_i,\xi_j) \ge r_0$. From \eqref{udx} and \eqref{gg}, we know that
\begin{align*}
&\ \int_{B^{g_{\xi_j}}_{r_0/2}(\xi_j)} \Lambda^{2^*-1}_{\xi_j}\(\mcu_j^{g_{\xi_j}}\)^{2^*-2} \delta_j\frac{\pa \mcu_j^{g_{\xi_j}}}{\pa\delta_j}\mcv_i dv_g \\
&= \int_{B^{g_{\xi_j}}_{r_0/2}(\xi_j)} \Lambda^{2^*-1}_{\xi_j}\(\mcu_j^{g_{\xi_j}}\)^{2^*-2} \delta_j\frac{\pa \mcu_j^{g_{\xi_j}}}{\pa\delta_j} \times \left[\ga_NG_g(x,\xi_i)d_{g_{\xi_i}}(x,\xi_i)^{N-2}\mcu_i^{g_{\xi_i}} \right. \\
&\hspace{55pt} \left. + (1-\chi(d_{g_{\xi_i}}(x,\xi_i))) \alpha_N\ga_N \delta_i^{\frac{N-2}{2}} G_g(x,\xi_i) \left\{1-\alpha_N^{-1} \delta_i^{-\frac{N-2}{2}} d_{g_{\xi_i}}(x,\xi_i)^{N-2}\mcu_i^{g_{\xi_i}}\right\}\right] (dv_g)_x \\
&= \alpha_N \int_{\{|y| \le \frac{r_0}{2\delta_j}\}} \(\frac{\Lambda_{\xi_i}}{\Lambda_{\xi_j}}\)\big(\exp_{\xi_j}^{g_{\xi_j}}(\delta_jy)\big) \(U^{2^*-2}Z^0\)(y) \cdot \ga_N \(G_{g_{\xi_i}}d_{g_{\xi_i}}^{N-2}\)\big(\exp_{\xi_j}^{g_{\xi_j}}(\delta_jy),\xi_i\big) \\
&\hspace{220pt} \times \left[\frac{\delta_i}{\delta_j} + \frac{d_{g_{\xi_i}}\big(\exp_{\xi_j}^{g_{\xi_j}}(\delta_jy),\xi_i\big)^2}{\delta_i\delta_j}\right]^{-\frac{N-2}{2}} dy + o(q_{ij}) \\
&=: J_5 + o(q_{ij}).
\end{align*}
Also, by using \eqref{gx}, \eqref{eq:dist1}, the equivalence between the metrics $d_{g_{\xi_i}}$, $d_{g_{\xi_j}}$, and $d_g$ on $M$, and the expansion
\[\(\frac{\Lambda_{\xi_i}}{\Lambda_{\xi_j}}\)\big(\exp_{\xi_j}^{g_{\xi_j}}(\delta_jy)\big) = \Lambda_{\xi_i}(\xi_j) + \mco\(|\delta_jy|\) \quad \text{for } y \in B_{r_0/(2\delta_j)}(0) \text{ and } 1 \le i \ne j \le \nu,\]
and reducing the size of $r_0 > 0$ if necessary, we can adopt the argument in the proof of \cite[(F16)]{B} to prove that
\begin{align*}
J_5 &= \alpha_N \Lambda_{\xi_i}(\xi_j) \int_{\{|y| \le \frac{r_0}{2\delta_j}\}} \(U^{2^*-2}Z^0\)(y)
\left[1+\ga_NA_g(\xi_i) d_{g_{\xi_i}}^{N-2}\big(\exp_{\xi_j}^{g_{\xi_j}}(\delta_jy),\xi_i\big) \right. \\
&\hspace{75pt} \left. + \mco\(d_{g_{\xi_i}}^{N-1}\big(\exp_{\xi_j}^{g_{\xi_j}}(\delta_jy),\xi_i\big)\) \right]
\left[\frac{\delta_i}{\delta_j} + \frac{d_{g_{\xi_i}}\big(\exp_{\xi_j}^{g_{\xi_j}}(\delta_jy),\xi_i\big)^2}{\delta_i\delta_j}\right]^{-\frac{N-2}{2}} dy + o(q_{ij}) \\
&\gtrsim q_{ij} + A_g(\xi_i)(\delta_i\delta_j)^{\frac{N-2}{2}} \ge q_{ij};
\end{align*}
cf. \eqref{eq:J3}. This leads to \eqref{neqij2}.
\end{proof}
\noindent In the above proof, we exploited a more refined bubble-like function $\mcv_i$ than the one used for Lemma \ref{le29}, which allowed us to avoid calculating integrals as in \eqref{eq:J1J2}.

\begin{lemma}\label{le36}
Let $\mfb_N = \frac{(N-2)^2}{2} \alpha_N^2\ga_N|\S^{N-1}| > 0$. For any $j \in \{1,\ldots,\nu\}$, we have
\begin{equation}\label{eq:le36}
\int_M \II_3\mcz^0_j dv_g = \mfb_NA_g(\xi_j)\delta_j^{N-2} + o\Big(\mcq + \max_{\ell=1,\ldots,\nu}\delta_{\ell}^{N-2}\Big).
\end{equation}
Here, $A_g(\xi_j) > 0$ thanks to the positive mass theorem in \cite{SY1}.
\end{lemma}
\begin{proof}
We provide the proof in two steps.

\medskip \noindent \textbf{Step 1.} We claim that
\begin{equation}\label{eq:le361}
\int_M \(-\mcl_g\mcv_j+\mcv_j^{2^*-1}\)\mcz^0_j dv_g =\mfb_NA_g(\xi_j)\delta_j^{N-2} + o\(\delta_j^{N-2}\).
\end{equation}

\medskip
By \eqref{ca13} and \eqref{ca3}, we have
\begin{equation}\label{f61}
\int_M \(-\mcl_g\mcv_j+\mcv_j^{2^*-1}\)\mcz^0_j dv_g = \int_{B^{g_{\xi_j}}_{r_0/2}(\xi_j)} \(-\mcl_g\mcv_j+\mcv_j^{2^*-1}\)\mcz^0_j \Lambda_{\xi_j}^{-2^*} dv_{g_{\xi_j}} + \mco\(\delta_j^N\).
\end{equation}
Estimate \eqref{ca23} and identity \eqref{gg} show
\begin{align}
&\ \int_{B^{g_{\xi_j}}_{r_0/2}(\xi_j)} \(-\mcl_g\mcv_j+\mcv_j^{2^*-1}\)\mcz^0_j \Lambda_{\xi_j}^{-2^*} dv_{g_{\xi_j}} \nonumber \\
&= 2\alpha_N^{2^*} \ga_NA_g(\xi_j) \int_{\{|y| \le \frac{r_0}{2}\}} \frac{\delta_j^{\frac{N+2}{2}}|y|^{N-2}}{(\delta_j^2+|y|^2)^{\frac{N+2}{2}}} \frac{\delta_j^{\frac{N-2}{2}}(|y|^2-\delta_j^2)}{(\delta_j^2+|y|^2)^{\frac{N}{2}}} dy \nonumber \\
&\ + (N-2)^3\alpha_N^2 \ga_NA_g(\xi_j) \int_{\{|y| \le \frac{r_0}{2}\}} \frac{\delta_j^{\frac{N+2}{2}}|y|^{N-4}}{(\delta_j^2+|y|^2)^{\frac{N}{2}}} \frac{\delta_j^{\frac{N-2}{2}}(|y|^2-\delta_j^2)}{(\delta_j^2+|y|^2)^{\frac{N}{2}}} dy + \mco\(\delta_j^{N-1}\) \label{f62} \\
&= \left[2\alpha_N^{2^*-2} \int_0^{\infty} \frac{r^{2N-3}(r^2-1)}{(1+r^2)^{N+1}} dr + (N-2)^3 \int_0^{\infty} \frac{r^{2N-5}(r^2-1)}{(1+r^2)^N} dr \right] \alpha_N^2\ga_N\left|\S^{N-1}\right| A_g(\xi_j)\delta_j^{N-2} \nonumber \\
&\ + o\(\delta_j^{N-2}\) \nonumber \\
&= \mfb_NA_g(\xi_j)\delta_j^{N-2} + o\(\delta_j^{N-2}\). \nonumber
\end{align}
Hence the claim holds.

\medskip \noindent \textbf{Step 2.} We assert that
\begin{equation}\label{eq:le362}
\int_M\(-\mcl_g\mcv_i+\mcv_i^{2^*-1}\)\mcz^0_j dv_g = o\Big(\mcq + \max_{\ell}\delta_{\ell}^{N-2}\Big) \quad \text{for } 1 \le i \ne j \le \nu.
\end{equation}

Indeed, \eqref{ca13}, \eqref{ca23}, \eqref{ca3}, and \eqref{22} yield
\[\left|\int_M\(-\mcl_g\mcv_i+\mcv_i^{2^*-1}\)\mcz^0_j dv_g\right| \lesssim \delta_i^{\frac{N-2}{2}} \int_{B^g_{r_0/2}(\xi_i)} \mcu_i^{2^*-2}\mcu_j dv_g + \mco\Big(\max_{\ell}\delta_{\ell}^N\Big) = o\Big(\mcq + \max_{\ell}\delta_{\ell}^{N-2}\Big)\]
for $N = 4,5$, while
\begin{multline*}
\left|\int_M \(-\mcl_g\mcv_i+\mcv_i^{2^*-1}\)\mcz^0_j dv_g\right| \\ \lesssim \delta_i^{\frac{1}{2}} \int_{B^g_{r_0/2}(\xi_i)} \mcu_i^4\mcu_j dv_g + \underbrace{\delta_i \int_{B^g_{r_0/2}(\xi_i)}
\frac{(\mcu_i^3\mcu_j)(x)}{d_g(x,\xi_i)} (dv_g)_x}_{=: J_6} + \mco\Big(\max_{\ell}\delta_{\ell}^3\Big) = o\Big(\mcq + \max_{\ell}\delta_{\ell}\Big)
\end{multline*}
for $N = 3$. To deduce the last inequality, we applied an estimate
\begin{equation}\label{eq:le363}
|J_6| \lesssim \delta_i\mcq = o\Big(\max_{\ell}\delta_{\ell}\Big)
\end{equation}
whose derivation is postponed to Appendix \ref{subsec:tech}. The assertion is proved.

\medskip
By virtue of \eqref{eq:le361} and \eqref{eq:le362}, estimate \eqref{eq:le36} is true.
\end{proof}

\begin{proof}[Completion of the proof of Proposition \ref{3.3}]
Choose any $j \in \{1,\ldots,\nu\}$. By following the estimation procedures of $\|\II_3\|_{L^{2N/(N+2)}(M)}$ depicted in the proof of Lemma \ref{lem3.3} and using Lemma \ref{a22}, we derive
\begin{align}
&\ \left\|\mcl_g\mcz^0_j - (2^*-1)\bigg(\sum_{i=1}^{\nu}\mcv_i\bigg)^{2^*-2} \mcz^0_j\right\|_{L^{\frac{2N}{N+2}}(M)} \nonumber \\
&\lesssim \left\|\mcl_g\mcz^0_j - (2^*-1)\mcv_j^{2^*-2} \mcz^0_j\right\|_{L^{\frac{2N}{N+2}}(M)} + \bigg\|\sum_{i \ne j} \mcv_i^{2^*-2}\mcv_j\bigg\|_{L^{\frac{2N}{N+2}}(M)} + \bigg\|\sum_{i \ne j} \mcv_i\mcv_j^{2^*-2}\bigg\|_{L^{\frac{2N}{N+2}}(M)} \nonumber \\
&= \mco\Big(\mcq + \max_{\ell}\delta_{\ell}^{N-2}\Big) = o(1). \label{lina2}
\end{align}
Moreover, we observe from \eqref{tezk2}, \eqref{eq:le35}, Lemmas \ref{le3p} and \ref{le36}, \eqref{lina2}, and Proposition \ref{pr3.2} that
\begin{equation*}
\begin{aligned}
&\ \left|\mfb_NA_g(\xi_j)\delta_j^{N-2} + (2^*-1)\sum_{i \ne j} \int_{B^{g_{\xi_j}}_{r_0/2}(\xi_j)} \Lambda^{2^*-1}_{\xi_j}\(\mcu_j^{g_{\xi_j}}\)^{2^*-2} \delta_j\frac{\pa \mcu_j^{g_{\xi_j}}}{\pa\delta_j}\mcv_i dv_g\right| \\
&\lesssim \|f\|_{H^{-1}(M)} + \|\rho\|_{H^1(M)}^2 + o(1)\|\rho\|_{H^1(M)} + o\Big(\mcq + \max_{\ell}\delta_{\ell}^{N-2}\Big) \\
&\lesssim \|f\|_{H^{-1}(M)} + o\Big(\mcq + \max_{\ell}\delta_{\ell}^{N-2}\Big).
\end{aligned}
\end{equation*}

Keeping \eqref{neqij2} and $A_g(\xi_j)>0$ in mind, we can now repeat the proof of Proposition \ref{lou0} to complete the proof. We omit the details.
\end{proof}

\section{The case $N \ge 6$ and $\nu = 1$}\label{se4}
This section is devoted to the proof of Theorem \ref{th1.4}. Throughout this section, we always assume that $N \ge 6$ and the number $\nu$ of the bubbles in Assumption \ref{assum} is $1$.
We also assume that $u_0 > 0$ on $M$ in Subsection \ref{se4.1} and $u_0 = 0$ on $M$ in Subsection \ref{se4.2}.

\medskip
Let $\mcv_{\delta,\xi}$ be the bubble-like function in \eqref{w1}. As in the previous sections, there exist $(\delta_1,\xi_1) \in (0,\infty) \times M$, $\vep_1 > 0$ small, and $\mcv_1 = \mcv_{\delta_1,\xi_1}$ such that
\[\|u-(u_0+\mcv_1)\|_{H^1(M)} = \inf\left\{\left\|u-\(u_0+\mcv_{\tde_1,\txi_1}\)\right\|_{H^1(M)}: \(\tde_1,\txi_1\) \in (0,\infty) \times M\right\} \le \vep_1.\]
In the statement of Theorem \ref{th1.4}, we imposed the condition that $\textup{Weyl}_g(\txi_1) \ne 0$ when $(M,g)$ is non-l.c.f. and either \textup{[}$N \ge 11$ and $u_0 > 0$\textup{]} or \textup{[}$N \ge 6$ and $u_0 = 0$\textup{]}.
By reducing the size of $\vep_1 > 0$ if necessary, we can assume that $\textup{Weyl}_g(\xi_1) \ne 0$.

Setting $\rho = u-(u_0+\mcv_1)$ and $f = \mcl_gu-u^{2^*-1}$, we have
\begin{equation}\label{eqrho3}
\begin{cases}
\displaystyle \mcl_g\rho - (2^*-1)(u_0+\mcv_1)^{2^*-2}\rho = f+\III_1[\rho]+\III_2+\III_3 \quad \text{on } M,\\
\displaystyle \big\langle \rho,\mcz^k_1 \big\rangle_{H^1(M)} = 0 \quad \text{for } k=0,\ldots,N
\end{cases}
\end{equation}
where $\mcz^0_1 = \delta_1 \frac{\pa \mcv_1}{\pa\delta_1}$, $\mcz^k_1 = \delta_1 \frac{\pa\mcv_1}{\pa\xi_1^k}$ for $k=1,\ldots,N$,
\[\III_1[\rho] := (u_0+\mcv_1+\rho)^{2^*-1} - (u_0+\mcv_1)^{2^*-1} - (2^*-1)(u_0+\mcv_1)^{2^*-2}\rho,\]
\begin{equation}\label{III23}
\III_2 := (u_0+\mcv_1)^{2^*-1}-u_0^{2^*-1}-\mcv_1^{2^*-1}, \quad \text{and} \quad \III_3 := -\mcl_g\mcv_1+\mcv_1^{2^*-1}.
\end{equation}
Reminding the conformal factor $\Lambda_{\xi_1}$ giving \eqref{dvg}, we write $g_{\xi_1} = \Lambda_{\xi_1}^{4/(N-2)}g$.

\subsection{The case $u_0 > 0$}\label{se4.1}
This subsection is devoted to the derivation of estimate \eqref{eq:sqe3}. We recall that $\textup{Weyl}_g(\xi_1) \ne 0$ when $N \ge 11$ and $(M,g)$ is non-l.c.f.

\begin{prop}\label{pr41}
It holds that
\begin{equation}\label{l4u1}
\|\rho\|_{H^1(M)} \lesssim \|f\|_{H^{-1}(M)} + \begin{cases}
\delta_1^2|\log \delta_1|^{\frac23} &\text{if } N=6,\\
\delta_1^{\frac{N+2}{4}} &\text{if } 7 \le N \le 13 \text{ or } [N \ge 14 \text{ and } (M,g) \text{ is l.c.f.}],\\
\delta_1^4 &\text{if } N \ge 14 \text{ and } (M,g) \text{ is non-l.c.f.}
\end{cases}
\end{equation}
\end{prop}
\begin{proof}
The proof is presented in three steps.

\medskip \noindent \textbf{Step 1.} Since $u_0 \lesssim \mcv_1$ if $d_{g_{\xi_1}}(x,\xi_1) \le \sqrt{\delta_1}$ and $u_0 \gtrsim \mcv_1$ if $d_{g_{\xi_1}}(x,\xi_1) \ge \sqrt{\delta_1}$, we obtain from \eqref{iqu} that
\begin{equation}\label{eq:III2est}
\begin{aligned}
|\III_2| &\lesssim \(u_0\mcv_1^{2^*-2}+u_0^{2^*-1}\)\bs{1}_{d_{g_{\xi_1}}(x,\xi_1) \le \sqrt{\delta_1}} + \(u_0^{2^*-2}\mcv_1+\mcv_1^{2^*-1}\)\bs{1}_{d_{g_{\xi_1}}(x,\xi_1) \ge \sqrt{\delta_1}} \\
&\lesssim u_0\mcv_1^{2^*-2}\bs{1}_{d_{g_{\xi_1}}(x,\xi_1) \le \sqrt{\delta_1}} + u_0^{2^*-2}\mcv_1\bs{1}_{d_{g_{\xi_1}}(x,\xi_1) \ge \sqrt{\delta_1}}.
\end{aligned}
\end{equation}
Direct computations show
\begin{equation}\label{4i32}
\|\III_2\|_{L^{\frac{2N}{N+2}}(M)} \lesssim \begin{cases}
\delta_1^2|\log \delta_1|^{\frac23} &\text{if } N=6,\\
\delta_1^{\frac{N+2}{4}} &\text{if } N \ge 7.
\end{cases}
\end{equation}

\medskip \noindent \textbf{Step 2.} (1) We first assume that $(M,g)$ is non-l.c.f. so that $\mcv_1 = \Lambda_{\xi_1}\chi(d_{g_{\xi_1}}(\cdot,\xi_1))\mcu_{\delta_1,{\xi_1}}^{g_{\xi_1}}$ on $M$. By \eqref{gxvu},
\begin{equation}\label{4i3}
\begin{aligned}
\III_3(x) &= \Lambda_{\xi_1}^{2^*-1}(x) \(\Delta\chi U_{\delta_1,0} + 2\nabla\chi \cdot \nabla U_{\delta_1,0} + \chi\Delta U_{\delta_1,0} + \chi^{2^*-1}U_{\delta_1,0}^{2^*-1}\)(y) \\
&\ - \Lambda_{\xi_1}^{2^*-1}(x) \ka_NR_{g_{\xi_1}}(x) (\chi U_{\delta_1,0})(y) + \mco\(|y|^{\theta-1}|\nabla U_{\delta_1,0}(y)|\)
\end{aligned}
\end{equation}
for $x = \exp_{\xi_1}^{g_{\xi_1}}(y) \in B_{r_0}^{g_{\xi_1}}(\xi_1)$. On the one hand,
\begin{equation}\label{4i}
\left|\Lambda_{\xi_1}^{2^*-1}(x)\(\Delta\chi U_{\delta_1,0} + 2\nabla\chi \cdot \nabla U_{\delta_1,0} + \chi\Delta U_{\delta_1,0} + \chi^{2^*-1}U_{\delta_1,0}^{2^*-1}\)(y)\right| \lesssim \delta_1^{\frac{N-2}{2}}.
\end{equation}
On the other hand, since \eqref{dvg} implies
\begin{equation}\label{tay}
R_{g_{\xi}}(\xi)=0, \quad \nabla_{g_{\xi}} R_{g_{\xi}}(\xi)=0, \quad \text{and} \quad \Delta_{g_{\xi}}R_{g_{\xi}}(\xi)=-\frac{1}{6}|\textup{Weyl}_g(\xi)|^2_g \quad \text{for } \xi \in M,
\end{equation}
we have that $R_{g_{\xi_1}}(x)=\mco(|y|^2)$, and so
\begin{multline}\label{4i2}
\(\int_{\{|y| \le r_0\}} \left|-\Lambda_{\xi_1}^{2^*-1}(x) \ka_NR_{g_{\xi_1}}(x) (\chi U_{\delta_1,0})(y) + \mco\(|y|^{\theta-1}|\nabla U_{\delta_1,0}(y)|\)\right|^{\frac{2N}{N+2}} dy\)^{\frac{N+2}{2N}} \\
\lesssim \(\int_{\{|y| \le r_0\}} \(|y|^2 U_{\delta_1,0}(y)\)^{\frac{2N}{N+2}}dy\)^{\frac{N+2}{2N}} \lesssim \begin{cases}
\delta_1^{\frac{N-2}{2}} &\text{if } 6\le N \le 9,\\
\delta_1^4|\log\delta_1|^{\frac35} &\text{if } N=10,\\
\delta_1^4 &\text{if } N \ge 11.
\end{cases}
\end{multline}
Thus \eqref{4i3}, \eqref{4i}, and \eqref{4i2} produce
\begin{equation}\label{123}
\|\III_3\|_{L^{\frac{2N}{N+2}}(M)} \lesssim \begin{cases}
\delta_1^{\frac{N-2}{2}} &\text{if } 6 \le N \le 9 \text{ and } (M,g) \text{ is non-l.c.f.}, \\
\delta_1^4|\log\delta_1|^{\frac35} &\text{if } N=10 \text{ and } (M,g) \text{ is non-l.c.f.}, \\
\delta_1^4 &\text{if } N \ge 11 \text{ and } (M,g) \text{ is non-l.c.f}.
\end{cases}
\end{equation}
(2) We next assume that $(M,g)$ is l.c.f. so that $\mcv_1 = \ga_NG_g(\cdot,\xi_1) [\chi(d_{g_{\xi_1}}(\cdot,\xi_1))d_{g_{\xi_1}}(\cdot,\xi_1)^{N-2} \mcu_{\delta_1,{\xi_1}}^{g_{\xi_1}}
+ (1-\chi(d_{g_{\xi_1}}(\cdot,\xi_1)))\alpha_N\delta_1^{(N-2)/2}]$ on $M$ and \eqref{gx} still holds. Then the proof of Lemma \ref{lem3.3} gives
\begin{equation}\label{i3ex}
\III_3(x) = \begin{cases}
\begin{aligned}
&\Lambda_{\xi_1}^{2^*-1}(x) \alpha_N\ga_NA_g(\xi_1) \delta_1^{\frac{N+2}{2}} \left[4N \frac{|y|^{N-2}}{(\delta_1^2+|y|^2)^{\frac{N+2}{2}}} \right. \\
&\hspace{40pt} \left. + 2(N-2)^2 \frac{|y|^{N-4}}{(\delta_1^2+|y|^2)^{\frac{N}{2}}} + \mco\(\frac{|y|^{N-3}}{(\delta_1^2+|y|^2)^{\frac{N}{2}}}\)\right]
\end{aligned}
&\text{if } d_{g_{\xi_1}}(x,\xi_1) \le \tfrac{r_0}{2}, \\
\mco\Big(\delta_1^{\frac{N+2}{2}}\Big) &\text{if } d_{g_{\xi_1}}(x,\xi_1) \ge \tfrac{r_0}{2}.
\end{cases}
\end{equation}
By employing \eqref{i3ex}, we compute
\begin{equation}\label{iii3}
\begin{aligned}
\|\III_3\|_{L^{\frac{2N}{N+2}}(M)} & \lesssim \left\|\frac{\delta_1^{\frac{N+2}{2}}}{\(\delta_1^2+|\cdot|^2\)^2}\right\|_{L^{\frac{2N}{N+2}}(B_{r_0/2}(0))} +\delta_1^{\frac{N+2}{2}} \\
&\lesssim \begin{cases}
\delta_1^4|\log\delta_1|^{\frac23} &\text{if } N=6 \text{ and } (M,g) \text{ is l.c.f.,}\\
\delta_1^{\frac{N+2}{2}} &\text{if } N \ge 7 \text{ and } (M,g) \text{ is l.c.f.}
\end{cases}
\end{aligned}
\end{equation}

\medskip \noindent \textbf{Step 3.} An analogous argument to the proof of Proposition \ref{pr2.2}
(namely, we use a coercivity estimate for $u_0+\mcv_1$ as in Proposition \ref{41}, a decomposition of $\rho$ similar to \eqref{dec}, and analogs of Lemmas \ref{le2.3} and \ref{le24}) with \eqref{123} and \eqref{iii3} yields
\begin{equation}\label{eq:rhoest}
\begin{aligned}
\|\rho\|_{H^1(M)} &\lesssim \|f\|_{H^{-1}(M)} + \|\III_2\|_{L^{\frac{2N}{N+2}}(M)} + \|\III_3\|_{L^{\frac{2N}{N+2}}(M)} \\
&\lesssim \|f\|_{H^{-1}(M)} + \begin{cases}
\delta_1^2|\log \delta_1|^{\frac23} &\text{if } N=6,\\
\delta_1^{\frac{N+2}{4}} &\text{if } 7 \le N \le 13 \text{ or } [N \ge 14 \text{ and } (M,g) \text{ is l.c.f.}],\\
\delta_1^4 &\text{if } N \ge 14 \text{ and } (M,g) \text{ is non-l.c.f.}
\end{cases}
\end{aligned}
\end{equation}
The proof is done.
\end{proof}

We derive a pointwise estimate of $\rho$ that will be useful later.
\begin{lemma}\label{le4.2}
Assume either $6 \le N \le 13$ or [$N \ge 14$ and $(M,g)$ is l.c.f.]. Then there exist a function $\trh_0 \in H^1(M)$ and numbers $\tc_0,\tc_1,\ldots,\tc_N \in \R$ satisfying the nonlinear equation
\begin{equation}\label{trho0}
\begin{cases}
\displaystyle \mcl_g\trh_0 - \left[(u_0+\mcv_1+\trh_0)^{2^*-1} - u_0^{2^*-1} - \mcv_1^{2^*-1}\right] = \sum_{k=0}^N \tc_k\mcl_g\mcz^k_1 \quad \text{on } M, \\
\displaystyle \big\langle \trh_0,\mcz^k_1 \big\rangle_{H^1(M)} = 0 \quad \text{for } k=0,1,\ldots,N
\end{cases}
\end{equation}
with estimates
\begin{equation}\label{tr0}
|\trh_0(x)| \lesssim \delta_1 \left[ \frac{\delta_1}{\delta_1^2+d_{g_{\xi_1}}(x,\xi_1)^2} \bs{1}_{d_{g_{\xi_1}}(x,\xi_1) \le \sqrt{\delta_1}}
+ \(\frac{\delta_1}{\delta_1^2+d_{g_{\xi_1}}(x,\xi_1)^2}\)^{\frac{N-4}{2}} \bs{1}_{d_{g_{\xi_1}}(x,\xi_1) \ge \sqrt{\delta_1}}\right]
\end{equation}
and
\begin{equation}\label{cj}
\sum_{k=0}^N|\tc_k| \lesssim \delta_1^{\frac{N-2}{2}}.
\end{equation}
Furthermore, if we let $\trh_1 = \rho-\trh_0$ so that
\begin{equation}\label{trho1}
\begin{cases}
\displaystyle \mcl_g\trh_1 - \left[(u_0+\mcv_1+\trh_0+\trh_1)^{2^*-1} -(u_0+\mcv_1+\trh_0)^{2^*-1}\right] = f+\III_3-\sum_{k=0}^N \tc_k \mcl_g\mcz^k_1 \quad \text{on } M,\\
\displaystyle \big\langle \trh_1,\mcz^k_1 \big\rangle_{H^1(M)} = 0 \quad \text{for } k=0,1,\ldots,N,
\end{cases}
\end{equation}
then it holds that
\begin{equation}\label{tr1}
\|\trh_1\|_{H^1(M)} \lesssim \|f\|_{H^{-1}(M)} + \|\III_3\|_{L^{\frac{2N}{N+2}}(M)} + \delta_1^{\frac{N+2}{2}}.
\end{equation}
\end{lemma}
\begin{proof}
Rewriting \eqref{eq:III2est}, we have
\begin{equation}\label{i32}
|\III_2| \lesssim \(\mcu_1^{g_{\xi_1}}\)^{2^*-2}\bs{1}_{d_{g_{\xi_1}}(x,\xi_1) \le \sqrt{\delta_1}} + \mcu_1^{g_{\xi_1}} \bs{1}_{d_{g_{\xi_1}}(x,\xi_1) \ge \sqrt{\delta_1}}.
\end{equation}
Hence, by applying Proposition \ref{a5} with $\tih=\III_1[\trh_0]+\III_2$ and the Banach fixed-point theorem, we obtain a solution $\trh_0$ to \eqref{trho0} and numbers $\tc_0, \tc_1, \ldots, \tc_N \in \R$ satisfying \eqref{tr0} and \eqref{cj}; refer to the proof of \cite[Proposition 5.4]{DSW} where similar computations were performed.
Testing \eqref{trho0} with $\trh_0$ and then using \eqref{tr0} and \eqref{4i32} (or \eqref{i32}), we easily observe that $\|\trh_0\|_{H^1(M)}$ and $\|\trh_1\|_{H^1(M)}$ are small.

By conducting calculations analogous to those in Proposition \ref{pr2.2}, we find
\begin{align*}
\|\trh_1\|_{H^1(M)} &\lesssim \|f\|_{H^{-1}(M)} + \|\III_3\|_{L^{\frac{2N}{N+2}}(M)} \\
&\ + \sum_{k=0}^N|\tc_k| \left|\int_M \mcl_g\mcz^k_1\mcv_1 dv_g\right| + \sum_{k=0}^N |\tc_k| \sum_{m=1}^L \left|\int_M \mcl_g\mcz^k_1\psi_m dv_g\right| \\
&\lesssim \|f\|_{H^{-1}(M)} + \|\III_3\|_{L^{\frac{2N}{N+2}}(M)} + \delta_1^{\frac{N+2}{2}},
\end{align*}
so \eqref{tr1} is true.
\end{proof}
\noindent By utilizing the previous lemma, one can improve \eqref{l4u1} for $N=6$.
\begin{cor}\label{cr43}
Suppose that $N=6$. It holds that
\begin{equation}\label{eq:cr43}
\|\rho\|_{H^1(M)} \lesssim \|f\|_{H^{-1}(M)} + \delta_1^2|\log \delta_1|^{\frac12}.
\end{equation}
\end{cor}
\begin{proof}
By \eqref{trho0}, \eqref{tr0}, and \eqref{i32},
\[\|\trh_0\|_{H^1(M)}^2 \lesssim \int_M \left[|\III_2||\trh_0| + (2^*-1)(u_0+\mcv_1)^{2^*-2}\trh_0^2 + |\trh_0|^{2^*}\right] dv_g \lesssim \delta_1^4|\log\delta_1|.\]
It follows from \eqref{123}, \eqref{iii3}, and \eqref{tr1} that
\[\|\rho\|_{H^1(M)} \lesssim \|\trh_0\|_{H^1(M)} + \|\trh_1\|_{H^1(M)}\lesssim \|f\|_{H^{-1}(M)} + \delta_1^2|\log \delta_1|^{\frac12}. \qedhere\]
\end{proof}

As in the previous sections, \eqref{eq:sqe3} is a consequence of Proposition \ref{pr41}, Corollary \ref{cr43}, and the following proposition.
\begin{prop}\label{pr44}
When $N \ge 6$ and $(M,g)$ is l.c.f., we have
\begin{equation}\label{eq:pr441}
\delta_1^{\frac{N-2}{2}} \lesssim \|f\|_{H^{-1}(M)}.
\end{equation}
When $(M,g)$ is non-l.c.f., we have
\begin{equation}\label{eq:pr442}
\left\{\!\begin{aligned}
&\delta_1^{\frac{N-2}{2}} &\text{if } 6 \le N \le 10\\[1ex]
&\delta_1^4 &\text{if } N \ge 11
\end{aligned}\right\} \lesssim \|f\|_{H^{-1}(M)}.
\end{equation}
\end{prop}
\noindent By testing \eqref{eqrho3} with $\mcz^0_1$, we obtain
\begin{equation}\label{tezk3}
\begin{aligned}
\int_M \III_2\mcz^0_1 dv_g + \int_M \III_3\mcz^0_1 dv_g &= -\int_M f\mcz^0_1 dv_g - \int_M \III_1[\rho]\mcz^0_1 dv_g \\
&\ + \int_M \left[\mcl_g\mcz^0_1 - (2^*-1)(u_0+\mcv_1)^{2^*-2}\mcz^0_1\right]\rho dv_g.
\end{aligned}
\end{equation}
In Lemmas \ref{s43} and \ref{s44}, we evaluate two integrals on the left-hand side of \eqref{tezk3}, respectively.
\begin{lemma}\label{s43}
If $N \ge 6$, we have
\begin{equation}\label{eq:s43}
\int_M \III_2\mcz^0_1 dv_g = \mfa_Nu_0(\xi_1)\delta_1^{\frac{N-2}{2}} + o\Big(\delta_1^{\frac{N-2}{2}}\Big)
\end{equation}
where $\mfa_N > 0$ is the constant in \eqref{eq:I2Z}.
\end{lemma}
\begin{proof}
When $(M,g)$ is l.c.f., one can easily check that
\[\int_M \III_2\mcz^0_1 dv_g = \int_M u_0\(\chi\Lambda_{\xi_1}\mcu_1^{g_{\xi_1}}\)^{2^*-2} \chi\Lambda_{\xi_1} \delta_1\frac{\pa \mcu_1^{g_{\xi_1}}}{\pa \delta_1}dv_g + \mco\Big(\delta_1^{\frac{N+2}{2}}\Big) = \mfa_Nu_0(\xi_1)\delta_1^{\frac{N-2}{2}}(1+o(1)).\]

If $(M,g)$ is non-l.c.f., then \eqref{eq:s43} follows from \eqref{i2ex}, \eqref{21}, \eqref{25}, \eqref{26}, $dv_g = \Lambda_{\xi_1}^{-2^*} dv_{g_{\xi_1}}$, and $\Lambda_{\xi_1}(x)=1+\mco(d_{g_{\xi_1}}(x,\xi_1)^2)$,
and the estimate
\[\int_{B^{g_{\xi_1}}_{\eta'\sqrt{\delta_1}}(\xi_1)} d_{g_{\xi_1}}(x,\xi_1)^2\(\mcu_1^{g_{\xi_1}}\)^{2^*-2} \left|\mcz^0_1\right| (dv_{g_{\xi_1}})_x \lesssim \delta_1^{\frac{N+2}{2}}|\log\delta_1| \quad \text{for a constant } \eta'>0. \qedhere\]
\end{proof}

\begin{lemma}\label{s44}
When $N \ge 6$ and $(M,g)$ is l.c.f., we have
\begin{equation}\label{13m}
\int_M \III_3\mcz^0_1 dv_g = \mfb_NA_g(\xi_1)\delta_1^{N-2}(1+o(1))
\end{equation}
where $\mfb_N > 0$ is the constant in \eqref{f62}. Also, $A_g(\xi_1) > 0$ thanks to the positive mass theorem in \cite{SY2}. When $(M,g)$ is non-l.c.f., we have
\begin{equation}\label{n13m}
\int_M\III_3\mcz^0_1 dv_g = \mfc_N|\textup{Weyl}_g(\xi_1)|^2_g \times \begin{cases}
\delta_1^4|\log\delta_1|(1+o(1)) &\text{if } N=6,\\
\delta_1^4(1+o(1)) &\text{if } N \ge 7
\end{cases}
\end{equation}
where $\mfc_6 := \frac{16}{5}$ for $N=6$ and $\mfc_N := \frac{(N-2)\alpha_N^2\ka_N}{24N} |\S^{N-1}| \int_0^{\infty} \frac{r^{N+1}(r^2-1)}{(1+r^2)^{N-1}} dr > 0$ if $N \ge 7$.
\end{lemma}
\begin{proof}
When $(M,g)$ is l.c.f., \eqref{13m} results from \eqref{f61} and \eqref{f62}.

\medskip
Assume that $(M,g)$ is non-l.c.f. By \eqref{4i3} and \eqref{4i},
\begin{align*}
\int_M \III_3\mcz^0_1 dv_g &= -\ka_N\int_{\{|y| \le \frac{r_0}{2}\}} R_{g_{\xi_1}}\big(\exp_{\xi_1}^{g_{\xi_1}}(y)\big) \(U_{\delta_1,0}Z_{\delta_1,0}^0\)(y) dy \\
&\ + \mco\(\int_{\{|y| \le \frac{r_0}{2}\}} |y|^{\theta-1}U_{\delta_1,0}^2(y) dy\) + \mco\(\delta_1^{N-2}\).
\end{align*}
Also, owing to \eqref{tay}, we know
\begin{align*}
&\ -\ka_N\int_{\{|y| \le \frac{r_0}{2}\}} R_{g_{\xi_1}}\big(\exp_{\xi_1}^{g_{\xi_1}}(y)\big) \(U_{\delta_1,0}Z_{\delta_1,0}^0\)(y) dy + \mco\(\int_{\{|y| \le \frac{r_0}{2}\}} |y|^{\theta-1}U_{\delta_1,0}^2(y) dy\) \\
&= \frac{\ka_N}{12N}|\textup{Weyl}_g(\xi_1)|^2_g \int_{\{|y| \le \frac{r_0}{2}\}} |y|^2\(U_{\delta_1,0}Z_{\delta_1,0}^0\)(y) dy + \mco\(\int_{\{|y| \le \frac{r_0}{2}\}} |y|^3U_{\delta_1,0}^2(y) dy\) \\
&= \frac{\ka_N}{12N}|\textup{Weyl}_g(\xi_1)|^2_g \delta_1^4 \int_{\{|y| \le \frac{r_0}{2\delta_1}\}} |y|^2\(UZ^0\)(y) dy + \mco
\(\left\{\!\begin{aligned}
&\delta_1^4 &\text{if } N=6\\
&\delta_1^5|\log\delta_1| &\text{if } N=7\\
&\delta_1^5 &\text{if } N \ge 8
\end{aligned}\right\}\) \\
&= \mfc_N|\textup{Weyl}_g(\xi_1)|^2_g \times \begin{cases}
\delta_1^4|\log\delta_1|(1+o(1)) &\text{if } N=6,\\
\delta_1^4(1+o(1)) &\text{if } N \ge 7.
\end{cases} \qedhere
\end{align*}
\end{proof}

\begin{proof}[Completion of the proof of Proposition \ref{pr44}]
We note from \eqref{s4a} that
\begin{multline*}
\left|\int_M \left[\mcl_g\rho-(2^*-1)(u_0+\mcv_1)^{2^*-2}\rho\right] \mcz^0_1 dv_g\right| \\
\lesssim \left\|\mcl_g\mcz^0_1-(2^*-1)\mcv_1^{2^*-2}\mcz^0_1\right\|_{L^{\frac{2N}{N+2}}(M)} \|\rho\|_{H^1(M)} + \int_M \left|u_0^{2^*-2}\mcz^0_1\rho\right|dv_g.
\end{multline*}
By direct computations, we obtain
\[\left\|\mcl_g\mcz^0_1 - (2^*-1)\mcv_1^{2^*-2}\mcz^0_1\right\|_{L^{\frac{2N}{N+2}}(M)}
\lesssim \begin{cases}
\delta_1^{\frac{N-2}{2}} &\text{if } 6 \le N \le 9 \text{ and } (M,g) \text{ is non-l.c.f.}, \\
\delta_1^4|\log\delta_1|^{\frac35} &\text{if } N=10 \text{ and } (M,g) \text{ is non-l.c.f.}, \\
\delta_1^4 &\text{if } N \ge 11 \text{ and } (M,g) \text{ is non-l.c.f.}, \\
\delta_1^4|\log\delta_1|^{\frac23} &\text{if } N=6 \text{ and } (M,g)\text{ is l.c.f.}, \\
\delta_1^{\frac{N+2}{2}} &\text{if } N \ge 7 \text{ and } (M,g)\text{ is l.c.f.}
\end{cases}\]
If $N \ge 14$ and $(M,g)$ is non-l.c.f., then Proposition \ref{pr41} shows
\[\int_M \left|u_0^{2^*-2}\mcz^0_1\rho\right| dv_g \lesssim \left\|\mcz^0_1\right\|_{L^{\frac{2N}{N+2}}(M)}\|\rho\|_{H^1(M)} \lesssim \delta_1^{\frac{N-2}{2}}\(\|f\|_{H^{-1}(M)} + \delta_1^4\).\]
Suppose that $6 \le N \le 13$ or [$N \ge 14$ and $(M,g)$ is l.c.f.]. We have
\[\int_M \left|u_0^{2^*-2}\mcz^0_1\rho\right| dv_g \lesssim \int_M \left|u_0^{2^*-2}\mcz^0_1\trh_0\right| dv_g + \left\|\mcz^0_1\right\|_{L^{\frac{2N}{N+2}}(M)}\|\trh_1\|_{H^1(M)}.\]
By \eqref{tr0} and \eqref{tr1},
\begin{equation}\label{poi}
\begin{aligned}
&\begin{medsize}
\displaystyle \ \int_M \left|u_0^{2^*-2}\mcz^0_1\trh_0\right| dv_g \end{medsize} \\
&\begin{medsize}
\displaystyle \lesssim \int_M \left|\mcz^0_1(x)\right| \delta_1 \left[\frac{\delta_1}{\delta_1^2+d_{g_{\xi_1}}(x,\xi_1)^2} \bs{1}_{d_{g_{\xi_1}}(x,\xi_1)
\le \sqrt{\delta_1}} + \(\frac{\delta_1}{\delta_1^2+d_{g_{\xi_1}}(x,\xi_1)^2}\)^{\frac{N-4}{2}} \bs{1}_{d_{g_{\xi_1}}(x,\xi_1) \ge \sqrt{\delta_1}}\right] (dv_g)_x
\end{medsize} \\
&\begin{medsize}
\displaystyle \lesssim \delta_1^{\frac{N+2}{2}}|\log\delta_1|
\end{medsize}
\end{aligned}
\end{equation}
and
\begin{multline*}
\left\|\mcz^0_1\right\|_{L^{\frac{2N}{N+2}}(M)}\|\trh_1\|_{H^1(M)} = o\(\|f\|_{H^{-1}(M)}\) \\
+ o(1) \times \begin{cases}
\delta_1^{\frac{N-2}{2}} &\text{if } 6 \le N \le 9 \text{ or } [N \ge 10 \text{ and } (M,g) \text{ is l.c.f.}],\\
\delta_1^4 &\text{if } 10 \le N \le 13 \text{ and } (M,g) \text{ is non-l.c.f.}
\end{cases}
\end{multline*}
Collecting the above calculations, we discover
\begin{multline}\label{z1p}
\left|\int_M \left[\mcl_g\rho-(2^*-1)(u_0+\mcv_1)^{2^*-2}\rho\right] \mcz^0_1 dv_g\right| \\
\lesssim o\(\|f\|_{H^{-1}(M)}\) + o(1) \times \begin{cases}
\delta_1^{\frac{N-2}{2}} &\text{if } 6 \le N \le 9 \text{ or } [N \ge 10 \text{ and } (M,g) \text{ is l.c.f.}],\\
\delta_1^4 &\text{if } N \ge 10 \text{ and } (M,g) \text{ is non-l.c.f.}
\end{cases}
\end{multline}

On the other hand, by \eqref{ab6}, Proposition \ref{pr41}, and Corollary \ref{cr43},
\begin{multline}\label{z3p}
\left|\int_M \III_1[\rho]\mcz^0_1 dv_g\right| \lesssim \int_M |\rho|^{2^*-1} \left|\mcz^0_1\right| dv_g \lesssim \|\rho\|_{H^1(M)}^{2^*-1} \\
= o\(\|f\|_{H^{-1}(M)}\) + o(1) \times \begin{cases}
\delta_1^{\frac{N-2}{2}} &\text{if } 6 \le N \le 9 \text{ or } [N \ge 10 \text{ and } (M,g) \text{ is l.c.f.}],\\
\delta_1^4 &\text{if } N \ge 10 \text{ and } (M,g) \text{ is non-l.c.f.}
\end{cases}
\end{multline}

Now, by putting Lemmas \ref{s43} and \ref{s44}, \eqref{z1p}, \eqref{z3p}, and $|\int_M f\mcz^0_1 dv_g| \lesssim \|f\|_{H^{-1}(M)}$ into \eqref{tezk3}, we obtain the desired estimates \eqref{eq:pr441} and \eqref{eq:pr442}. This concludes the proof of Proposition \ref{pr44}.
\end{proof}

\subsection{The case $u_0 = 0$}\label{se4.2}
This subsection is devoted to the derivation of estimate \eqref{eq:sqe4}. We recall that $\textup{Weyl}_g(\xi_1) \ne 0$ when $(M,g)$ is non-l.c.f.

\begin{prop}\label{pr4.7}
It holds that
\begin{equation}\label{l4q}
\|\rho\|_{H^1(M)} \lesssim \|f\|_{H^{-1}(M)} + \begin{cases}
\delta_1^4|\log\delta_1|^{\frac23} &\text{if } N=6 \text{ and } (M,g) \text{ is l.c.f.},\\
\delta_1^4|\log\delta_1|^{\frac53} &\text{if } N=6 \text{ and } (M,g) \text{ is non-l.c.f.},\\
\delta_1^{\frac{N+2}{2}} &\text{if } N \ge 7 \text{ and } (M,g) \text{ is l.c.f.},\\
\delta_1^4 &\text{if } N \ge 7 \text{ and } (M,g) \text{ is non-l.c.f.}
\end{cases}
\end{equation}
\end{prop}
\begin{proof}
An analogous argument to the proof of Proposition \ref{pr2.2} yields
\begin{equation}\label{eq:rhoest2}
\|\rho\|_{H^1(M)} \lesssim \|f\|_{H^{-1}(M)} + \|\III_3\|_{L^{\frac{2N}{N+2}}(M)};
\end{equation}
cf. \eqref{eq:rhoest}. If $(M,g)$ is l.c.f., then \eqref{l4q} immediately follows from \eqref{eq:rhoest2} and \eqref{iii3}. In the rest of the proof, we assume that $(M,g)$ is non-l.c.f.

Let $y = (y^1,\ldots,y^N) \in B_{r_0}(0)$ and $x = \exp_{\xi_1}^{g_{\xi_1}}y \in M$. In \cite[Lemma 6.4]{LP}, it was shown that
\begin{equation}\label{gx2}
\ga_N G_{g_{\xi_1}}(x,\xi_1) =\begin{cases}
\displaystyle \frac{1}{|y|^4} - \frac{1}{1440}|\textup{Weyl}_g(\xi_1)|^2_g\log|y| + \mco(1) &\text{if } N=6,\\
\begin{aligned}
&\frac{1}{|y|^{N-2}} + \frac{\ka_N}{144(N-4)(N-6)} \frac{|\textup{Weyl}_g(\xi_1)|^2_g}{|y|^{N-6}} \\
&\ -\frac{\ka_N}{12(N-4)} \pa_{y^ky^l}R_{g_{\xi_1}}(\xi_1) \frac{y^ky^l}{|y|^{N-4}} + \mco\(\frac{1}{|y|^{N-7}}\)
\end{aligned}
&\text{if } N \ge 7,
\end{cases}
\end{equation}
where the indices $k$ and $l$ range from $1$ to $N$. If $d_{g_{\xi_1}}(x,\xi_1) \le \frac{r_0}{2}$, then by \eqref{eq:Fi}, \eqref{gx2}, and the expansion $g_{\xi_1}^{kl}(\exp_{\xi_1}^{g_{\xi_1}}y) = \delta^{kl} + \mco(|y|^2)$, it follows that
\begin{align}
&\ \ga_N \la\nabla_{g_{\xi_1}} G_{g_{\xi_1}}(x,\xi_1),\nabla_{g_{\xi_1}} F_1(x)\ra_{g_{\xi_1}} \label{com2} \\
&= \begin{cases}
\displaystyle - 384 \frac{1}{|y|^2} \frac{\delta_1^4}{(\delta_1^2+|y|^2)^3} + \mco\(\frac{\delta_1^4|y|^2}{(\delta_1^2+|y|^2)^3}\) &\text{if } N=6, \\
\begin{aligned}
&\begin{medsize}
\displaystyle -(N-2)\alpha_N \left[\frac{N-2}{|y|^2} + \frac{\ka_N}{144(N-4)}|\textup{Weyl}_g(\xi_1)|^2_g|y|^2 - \frac{(N-6)\ka_N}{12(N-4)} \pa_{y^ky^l}R_{g_{\xi_1}}(\xi_1) y^ky^l\right]
\end{medsize} \\
&\begin{medsize}
\displaystyle \hspace{195pt} \times \frac{\delta_1^{\frac{N+2}{2}}}{(\delta_1^2+|y|^2)^{\frac{N}{2}}} + \mco\(\frac{\delta_1^{\frac{N+2}{2}}|y|^3}{(\delta_1^2+|y|^2)^{\frac{N}{2}}}\)
\end{medsize}
\end{aligned}
&\text{if } N \ge 7.
\end{cases} \nonumber
\end{align}
Thus, employing \eqref{com2} instead of \eqref{com1}, one can argue as in the proof of Lemma \ref{lem3.3} to deduce
\begin{multline}\label{i36}
\III_3(x) = -\Lambda_{\xi_1}^2(x) \left[\frac{2}{5} |\textup{Weyl}_g(\xi_1)|^2_g|y|^4\log|y| \frac{\delta_1^4}{(\delta_1^2+|y|^2)^4} \right. \\
\left. + \frac{8}{15}|\textup{Weyl}_g(\xi_1)|^2_g |y|^2\log|y| \frac{\delta_1^4}{(\delta_1^2+|y|^2)^3} + \mco\(\frac{\delta_1^4|y|^2}{(\delta_1^2+|y|^2)^3}\)\right]
\end{multline}
for $N=6$ and
\begin{align}
&\begin{medsize}
\displaystyle \ \III_3(x) = \Lambda_{\xi_1}^{2^*-1}(x) \left[\frac{\alpha_N^{2^*-1}}{12(N-1)(N-4)} \left\{\frac{1}{12(N-6)}|\textup{Weyl}_g(\xi_1)|^2_g|y|^4 - \pa_{y^ky^l}R_{g_{\xi_1}}(\xi_1)y^ky^l|y|^2\right\} \frac{\delta_1^{\frac{N+2}{2}}}{(\delta_1^2+|y|^2)^{\frac{N+2}{2}}} \right.
\end{medsize} \nonumber \\
&\begin{medsize}
\displaystyle \hspace{85pt} + \frac{(N-2)^3\alpha_N}{6(N-1)(N-4)} \left\{\frac{1}{12(N-6)}|\textup{Weyl}_g(\xi_1)|^2_g|y|^2 - \pa_{y^ky^l}R_{g_{\xi_1}}(\xi_1)y^ky^l\right\} \frac{\delta_1^{\frac{N+2}{2}}}{(\delta_1^2+|y|^2)^{\frac{N}{2}}}
\end{medsize} \nonumber \\
&\begin{medsize}
\displaystyle \hspace{85pt} \left. + \mco\(\frac{\delta_1^{\frac{N+2}{2}}|y|^3}{(\delta_1^2+|y|^2)^{\frac{N}{2}}}\)\right]
\end{medsize}\label{i37}
\end{align}
for $7 \le N \le 10$. If $d_{g_{\xi_1}}(x,\xi_1) \ge \frac{r_0}{2}$, calculations similar to \eqref{ca13}, \eqref{331}, and \eqref{332} reveal that $|\III_3| \lesssim \delta_1^{\frac{N+2}{2}}$. As a result,
\begin{equation}\label{124}
\|\III_3\|_{L^{\frac{2N}{N+2}}(M)} \lesssim \begin{cases}
\delta_1^4|\log\delta_1|^{\frac53} &\text{if } N=6,\\
\delta_1^4 &\text{if } 7 \le N \le 10,
\end{cases}
\end{equation}
which together with \eqref{123} (for $N \ge 11$) and \eqref{eq:rhoest2} implies \eqref{l4q}.
\end{proof}

By \eqref{ca13}, \eqref{331}, \eqref{332}, \eqref{i3ex}, \eqref{i36}, and \eqref{i37}, it holds that
\begin{align}
&\ |\III_3(x)| \label{5i3} \\
&\lesssim \begin{cases}
\begin{medsize}
\displaystyle \delta_1^4 \left[\frac{\log (2+d_{g_{\xi_1}}(x,\xi_1))}{(\delta_1^2+d_{g_{\xi_1}}(x,\xi_1)^2)^2} \bs{1}_{d_{g_{\xi_1}}(x,\xi_1) \le \frac{r_0}{2}} + \bs{1}_{d_{g_{\xi_1}}(x,\xi_1) \ge \frac{r_0}{2}}\right]
\end{medsize}
&\begin{medsize}
\displaystyle \text{if } N=6 \text{ and } (M,g) \text{ is non-l.c.f.},
\end{medsize} \\
\begin{medsize}
\displaystyle \delta_1^{\frac{N-2}{2}} \(\frac{\delta_1}{\delta_1^2+d_{g_{\xi_1}}(x,\xi_1)^2}\)^2 \bs{1}_{d_{g_{\xi_1}}(x,\xi_1) \le \frac{r_0}{2}}
+ \delta_1^{\frac{N+2}{2}} \bs{1}_{d_{g_{\xi_1}}(x,\xi_1) \ge \frac{r_0}{2}}
\end{medsize}
&\begin{medsize}
\displaystyle \text{if } N \ge 6 \text{ and } (M,g) \text{ is l.c.f.}
\end{medsize}
\end{cases} \nonumber
\end{align}
By reasoning as in the proof of Lemma \ref{le4.2} with \eqref{5i3} in hand, we establish the following lemma. The proof is omitted.
\begin{lemma}\label{le4.8}
Assume either $N=6$ or [$N \ge 7$ and $(M,g)$ is l.c.f.]. Then there exist a function $\trh_0 \in H^1(M)$ and numbers $\tc_0, \tc_1, \ldots, \tc_N \in \R$ satisfying
\begin{equation}\label{trho02}
\begin{cases}
\displaystyle \mcl_g\trh_0 - \left[(\mcv_1+\trh_0)^{2^*-1}-\mcv_1^{2^*-1}\right] = \III_3 + \sum_{k=0}^N \tc_k \mcl_g\mcz^k_1 \quad \text{on } M, \\
\displaystyle \big\langle \trh_0,\mcz^k_1 \big\rangle_{H^1(M)} = 0 \quad \text{for } k=0,1,\ldots,N
\end{cases}
\end{equation}
with estimates
\begin{equation}\label{tr02}
|\trh_0(x)| \lesssim \begin{cases}
\begin{medsize}
\displaystyle \delta_1^4\frac{\log (2+d_{g_{\xi_1}}(x,\xi_1))}{\delta_1^2 +d_{g_{\xi_1}}(x,\xi_1)^2} \bs{1}_{d_{g_{\xi_1}}(x,\xi_1) \le \frac{r_0}{2}}+\delta_1^4\bs{1}_{d_{g_{\xi_1}}(x,\xi_1) \ge \frac{r_0}{2}}
\end{medsize}
&\begin{medsize}
\displaystyle \text{if } N=6 \text{ and } (M,g) \text{ is non-l.c.f.},
\end{medsize} \\
\begin{medsize}
\displaystyle \delta_1^{\frac{N}{2}} \frac{\delta_1}{\delta_1^2+d_{g_{\xi_1}}(x,\xi_1)^2} \bs{1}_{d_{g_{\xi_1}}(x,\xi_1) \le \frac{r_0}{2}} +\delta_1^{\frac{N+2}{2}} \bs{1}_{d_{g_{\xi_1}}(x,\xi_1) \ge \frac{r_0}{2}}
\end{medsize}
&\begin{medsize}
\displaystyle \text{if } N \ge 6 \text{ and } (M,g) \text{ is l.c.f.}
\end{medsize}
\end{cases}
\end{equation}
and
\[\sum_{k=0}^N|\tc_k| \lesssim \begin{cases}
\delta_1^4|\log\delta_1| &\text{if } N=6 \text{ and } (M,g) \text{ is non-l.c.f.},\\
\delta_1^{N-2} &\text{if } N \ge 6 \text{ and } (M,g) \text{ is l.c.f.}
\end{cases}\]
Moreover, if we let $\trh_1 = \rho-\trh_0$ so that
\[\begin{cases}
\displaystyle \mcl_g\trh_1 - \left[(\mcv_1+\trh_0+\trh_1)^{2^*-1}-(\mcv_1+\trh_0)^{2^*-1}\right] = f-\sum_{k=0}^N \tc_k \mcl_g\mcz^k_1 \quad \text{on } M,\\
\displaystyle \big\langle \trh_1,\mcz^k_1 \big\rangle_{H^1(M)} = 0 \quad \text{for } k=0,1,\ldots,N,
\end{cases}\]
then it holds that
\begin{equation}\label{tr12}
\|\trh_1\|_{H^1(M)} \lesssim \|f\|_{H^{-1}(M)} + \begin{cases}
\delta_1^6|\log\delta_1| &\text{if } N=6 \text{ and } (M,g) \text{ is non-l.c.f.},\\
\delta_1^N &\text{if } N \ge 6 \text{ and } (M,g) \text{ is l.c.f.}
\end{cases}
\end{equation}
\end{lemma}
\noindent By exploiting \eqref{trho02}--\eqref{tr12}, one can improve \eqref{l4q} for $N=6$. We skip its proof.
\begin{cor}\label{cr46}
Suppose that $N=6$. It holds that
\[\|\rho\|_{H^1(M)} \lesssim \|f\|_{H^{-1}(M)}+\begin{cases}
\delta_1^4|\log\delta_1|^{\frac12} &\text{if } (M,g) \text{ is l.c.f.},\\
\delta_1^4|\log\delta_1|^{\frac32} &\text{if } (M,g) \text{ is non-l.c.f.}
\end{cases}\]
\end{cor}

\begin{prop}\label{pr47}
When $N \ge 6$ and $(M,g)$ is l.c.f., we have
\[\delta_1^{N-2} \lesssim \|f\|_{H^{-1}(M)}.\]
When $(M,g)$ is non-l.c.f., we have
\[\left\{\!\begin{aligned}
&\delta_1^4|\log\delta_1| &\text{if } N=6\\[1ex]
&\delta_1^4 &\text{if } N \ge 7
\end{aligned}\right\} \lesssim \|f\|_{H^{-1}(M)}.\]
\end{prop}
\begin{proof}
Let us estimate the integral $\int_M \III_3\mcz^0_1 dv_g$.
If either $(M,g)$ is l.c.f. or [$N \ge 11$ and $(M,g)$ is non-l.c.f.], then we use the same bubble-like function $\mcv_1$ in both cases $u_0 > 0$ and $u_0 = 0$; see \eqref{w1}. Hence, we can borrow estimates \eqref{13m} and \eqref{n13m}.
In contrast, if $6 \le N \le 10$ and $(M,g)$ is non-l.c.f., the definition of $\mcv_1$ differs, so we need to compute the integral anew.

If $N=6$ and $(M,g)$ is non-l.c.f., it follows from \eqref{i36} that
\begin{align*}
\int_M \III_3\mcz^0_1 dv_g &= \frac{32}{5}|\textup{Weyl}_g(\xi_1)|^2_g \delta_1^4|\log\delta_1| \int_{\{|y| \le \frac{r_0}{2\delta_1}\}} \frac{|y|^2(7|y|^2+4)}{(1+|y|^2)^4} \frac{|y|^2-1}{(1+|y|^2)^3} dy(1+o(1)) \\
&= \mfd_6|\textup{Weyl}_g(\xi_1)|^2_g \delta_1^4|\log\delta_1| (1+o(1))
\end{align*}
where $\mfd_6 := \frac{16}{5}|\S^5| > 0$. If $7 \le N \le 10$ and $(M,g)$ is non-l.c.f., one can infer from \eqref{i37} and \eqref{tay} that
\begin{align*}
&\begin{medsize}
\displaystyle \ \int_M \III_3\mcz^0_1 dv_g
\end{medsize} \\
&\begin{medsize}
\displaystyle = \frac{(N-2)\alpha_N^2\ka_N}{72(N-4)(N-6)} |\textup{Weyl}_g(\xi_1)|^2_g\delta_1^4 \int_{\{|y| \le \frac{r_0}{2\delta_1}\}}\frac{|y|^2[(2N^2-7N+8)|y|^2+2(N-2)^2]}{(1+|y|^2)^{\frac{N+2}{2}}} \frac{|y|^2-1}{(1+|y|^2)^{\frac{N}{2}}} dy
\end{medsize} \\
&\begin{medsize}
\displaystyle \ -\frac{(N-2)\alpha_N^2\ka_N}{6(N-4)} (\pa_{y^ky^l}R_{g_{\xi_1}})(\xi_1) \delta_1^4 \int_{\{|y| \le \frac{r_0}{2\delta_1}\}} \frac{y^ky^l[(2N^2-7N+8)|y|^2+2(N-2)^2]}{(1+|y|^2)^{\frac{N+2}{2}}} \frac{|y|^2-1}{(1+|y|^2)^{\frac{N}{2}}} dy + o(\delta_1^4)
\end{medsize} \\
&\begin{medsize}
\displaystyle = \frac{(N-2)\alpha_N^2\ka_N}{24N(N-6)} \int_0^{\infty} \frac{r^{N+1}[(2N^2-7N+8)r^2+2(N-2)^2](r^2-1)}{(1+r^2)^{N+1}} dr |\S^{N-1}||\textup{Weyl}_g(\xi_1)|^2_g \delta_1^4 (1+o(1))
\end{medsize} \\
&\begin{medsize}
\displaystyle = \mfd_N|\textup{Weyl}_g(\xi_1)|^2_g\delta_1^4(1+o(1))
\end{medsize}
\end{align*}
for some $\mfd_N>0$.

Using the above estimates, we can argue as in Proposition \ref{pr44} to deduce Proposition \ref{pr47}. The details are omitted.
\end{proof}

Consequently, Propositions \ref{pr4.7} and \ref{pr47} lead to the desired estimate \eqref{eq:sqe4} for $N \ge 7$.
If $N=6$ and $(M,g)$ is l.c.f., then \eqref{eq:sqe4} follows directly from Corollary \ref{cr46} and Proposition \ref{pr47}.
If $N=6$ and $(M,g)$ is non-l.c.f., then $\delta_1^4 \le \delta_1^4|\log\delta_1| \le C_1\|f\|_{H^{-1}(M)}$ for some $C_1 > 0$ so that $\delta_1^4 \le 4C_1\|f\|_{H^{-1}(M)}|\log C_1\|f\|_{H^{-1}(M)}|^{-1}$. Since the function $t \mapsto t^4|\log t|^{\frac32}$ is increasing for $t > 0$ small and
\[\lim_{t \to 0+} \frac{1}{|\log t|} \left|\log \frac{ct}{|\log t|}\right| = 1 \quad \text{for any } c > 0,\]
it follows that
\[\delta_1^4|\log\delta_1|^{\frac32} \lesssim 4C_1\|f\|_{H^{-1}(M)} |\log C_1\|f\|_{H^{-1}(M)}|^{\frac12},\]
which implies \eqref{eq:sqe4} again.

\section{Optimality of the results}\label{se5}
This section is devoted to the proof of Theorem \ref{th1.6}.

\subsection{Optimality of \eqref{eq:sqe1} and \eqref{eq:sqe2}}\label{se5.1}
We shall consider \eqref{eq:sqe1} only, since \eqref{eq:sqe2} can be treated analogously. The proof consists of two steps.

\medskip \noindent \textbf{Step 1.} Choose any $\delta \in (0,1)$ and $(\xi_1,\ldots,\xi_\nu) \in M^{\nu}$ such that
$d_g(\xi_i,\xi_j)\ge c$ for all $1 \le i \ne j \le \nu$ and some $c > 0$.
We also take $\delta = \delta_1 = \cdots = \delta_{\nu}$ and $\vep = \delta^{(N-2)/2}$.
Let $\mcv_i = \mcv_{\delta_i,\xi_i}$ and $\mcz^k_i$ be the functions defined in \eqref{ui} and \eqref{eq:wtmcz}, respectively, and
\[u_* = u_0 + \sum_{i=1}^{\nu}\mcv_i + \vep\phi_{\delta}\]
where $\{\phi_{\delta}\}_{\delta \in (0,1)}$ is a family of nonzero smooth functions on $M$ such that $\|\phi_{\delta}\|_{H^1(M)} \simeq 1$ and $\langle \phi_{\delta},\mcz^k_i \rangle_{H^1(M)} = 0$ for $i=1,\ldots,\nu$ and $k=0,\ldots,N$. It holds that $\mcq \simeq \delta^{N-2}$.

Denoting $\rho = \vep\phi_{\delta}$, we observe
\[\|\rho\|_{H^1(M)} = \|\vep\phi_{\delta}\|_{H^1(M)} \simeq \vep.\]
We set $f$ by
\[f = \mcl_gu_*-u_*^{2^*-1} = \mcl_g\rho + u_0^{2^*-1} + \sum_{i=1}^{\nu} \mcl_g\mcv_i - \bigg(u_0+\sum_{i=1}^{\nu}\mcv_i +\rho\bigg)^{2^*-1}.\]
Then \eqref{eqrho} and \eqref{u0rho} imply
\begin{align*}
\Gamma(u_*) = \|f\|_{H^{-1}(M)} &\lesssim \|\rho\|_{H^1(M)} + \|\rho\|_{H^1(M)}^{2^*-1} + \|\I_2\|_{L^{\frac{2N}{N+2}}(M)} + \|\I_3\|_{L^{\frac{2N}{N+2}}(M)} + \|\I_4\|_{L^{\frac{2N}{N+2}}(M)} \\
&\lesssim \|\rho\|_{H^1(M)} + \mcq + \delta^{\frac{N-2}{2}} \simeq \|\rho\|_{H^1(M)} + \vep \simeq \|\rho\|_{H^1(M)}
\end{align*}
where $\I_2$, $\I_3$, and $\I_4$ are the functions defined in \eqref{eq:I2} and \eqref{eq:I3I4}.

\medskip \noindent \textbf{Step 2.} We claim that for any $\delta > 0$ small enough,
\[\inf\left\{\bigg\|u_*-\bigg(u_0+\sum_{i=1}^{\nu} \mcv_{\tde_i,\txi_i}\bigg)\bigg\|_{H^1(M)}: \(\tde_i,\txi_i\) \in (0,\infty) \times M,\, i = 1,\ldots,\nu\right\} \gtrsim \|\rho\|_{H^1(M)}\]
where $\mcv_{\tde_i,\txi_i}$ is defined by \eqref{ui}.
There exist parameters $(\tde_1,\ldots,\tde_{\nu},\txi_1,\ldots,\txi_{\nu}) \in (0,\infty)^{\nu} \times M^{\nu}$, where we still use the same notation, such that the above infimum is achieved by $\trh := u_*-(u_0+\sum_{i=1}^{\nu}\mcv_{\tde_i,\txi_i})$.

It holds that $\|\trh\|_{H^1(M)} \le \|\rho\|_{H^1(M)} \simeq\vep$, so
\[\bigg\|\sum_{i=1}^{\nu}\mcv_{\tde_i,\txi_i}-\sum_{i=1}^{\nu}\mcv_i\bigg\|_{H^1(M)} = \|\rho-\trh\|_{H^1(M)} \le \|\rho\|_{H^1(M)} + \|\trh\|_{H^1(M)} \lesssim \vep,\]
which implies that $\tde_i=(1+o(1))\delta$ and $d_g(\txi_i,\xi_i)=o(\delta)$ where $o(1) \to 0$ as $\delta \to 0$; cf. \cite[Lemma A.1]{BC}. 

We set $\vep' = \max\{|\tde_i-\delta|+d_g(\txi_i,\xi_i): i = 1,\ldots,\nu\}=o(\delta)$ and write $w_i= (\exp^g_{\xi_i})^{-1}(\txi_i)\in B_{r_0}(0)$. By Taylor's theorem with respect to the variables $(\delta,\xi)$, there exist $\mcr_i\in H^1(M)$ such that
\begin{align*}
\mcv_{\tde_i,\txi_i}-\mcv_i &= \frac{\pa \mcv_i}{\pa\delta}(\tde_i-\delta) + \la \nabla_{\tw}\mcv_{\delta,\exp^g_{\xi_i}(\tw)}\Big|_{\tw=0}, w_i\ra + \mcr_i \\
&= \frac{1}{\delta} \mcz^0_i(\tde_i-\delta) + \frac{1}{\delta} \sum_{k=1}^N \mcz^k_i \left[\big(\exp^g_{\xi_i}\big)^{-1}(\txi_i)\right]^k + \mcr_i \quad \text{on } M
\end{align*}
and $\|\mcr_i\|_{H^1(M)} = \mco(\frac{\vep'^2}{\delta^2})$ for $i=1,\ldots,\nu$. Using Lemma \ref{le2p}, one realizes
\[\max_{i\in\{1,\ldots,\nu\}}\|\mcv_{\tde_i,\txi_i}-\mcv_i\|_{H^1(M)} \simeq \frac{\vep'}{\delta} \quad \text{and} \quad
\sum_{i \ne j} \la \mcv_{\tde_i,\txi_i}-\mcv_i,\mcv_{\tde_j,\txi_j}-\mcv_j \ra_{H^1(M)} = o\(\frac{\vep'^2}{\delta^2}\).\]
Also, from $\langle \rho,\mcz^k_i \rangle_{H^1(M)} = \vep \langle \phi_{\delta},\mcz^k_i \rangle_{H^1(M)}= 0$ for $k = 0,\ldots,N$, it can be shown that
\begin{align*}
\bigg\langle \sum_{i=1}^{\nu}\Big(\mcv_{\tde_i,\txi_i}-\mcv_i\Big),\rho \bigg\rangle_{H^1(M)} &= \bigg\langle \sum_{i=1}^{\nu}\mcr_i,\rho \bigg\rangle_{H^1(M)} \lesssim \sum_{i=1}^{\nu}\|\mcr_i\|_{H^1(M)} \|\rho\|_{H^1(M)}\\
&= o(1) \bigg\|\sum_{i=1}^{\nu}\Big(\mcv_{\tde_i,\txi_i}-\mcv_i\Big)\bigg\|_{H^1(M)}\|\rho\|_{H^1(M)}.
\end{align*}

Recalling that $\trh=\rho+\sum_{i=1}^{\nu}(\mcv_{\tde_i,\txi_i}-\mcv_i)$, we get
\begin{align*}
\|\trh\|_{H^1(M)}^2 &= \|\rho\|_{H^1(M)}^2 + \bigg\|\sum_{i=1}^{\nu} \Big(\mcv_{\tde_i,\txi_i}-\mcv_i\Big)\bigg\|_{H^1(M)}^2 + 2\bigg\langle \sum_{i=1}^{\nu} \Big(\mcv_{\tde_i,\txi_i}-\mcv_i\Big),\rho \bigg\rangle_{H^1(M)} \\
&\gtrsim \|\rho\|_{H^1(M)}^2,
\end{align*}
which proves the assertion. The optimality of \eqref{eq:sqe1} was established.

\subsection{Optimality of \eqref{eq:sqe3}}\label{se5.2}
If $N \ge 14$ and $(M,g)$ is non-l.c.f., that is, when $\zeta(t) = t$, one can slightly modify the argument in the previous subsection to prove the optimality of \eqref{eq:sqe3}.
Therefore, we only need to deal with the cases that $6 \le N \le 13$ or [$N \ge 14$ and $(M,g)$ is l.c.f.]. The proof is long, so we separate it into three steps.

\medskip \noindent \textbf{Step 1.} Choose any $\delta_1 > 0$ small and $\xi_1 \in M$ such that $\textup{Weyl}_g(\xi_1) \ne 0$ provided $N \ge 11$ and $(M,g)$ is non-l.c.f. Let $\mcv_1 = \mcv_{\delta_1,\xi_1}$ be the function in \eqref{w1}.
The standard invertibility argument combined with the non-degeneracy of $u_0$ and the Banach fixed-point theorem shows the existence of $\rho \in H^1(M)$ and $(\mtc_0,\ldots,\mtc_N) \in \R^{N+1}$ such that
\begin{equation}\label{rhoex}
\begin{cases}
\displaystyle \mcl_g\rho-(2^*-1)(u_0+\mcv_1)^{2^*-2}\rho = \III_1[\rho] + \III_2 + \III_3 + \sum_{k=0}^N\mtc_k\mcl_g\mcz^k_1,\\
\displaystyle \big\langle \rho,\mcz^k_1 \big\rangle_{H^1(M)} = 0 \quad \text{for } k=0,1,\ldots,N
\end{cases}
\end{equation}
where $\III_1[\rho]$, $\III_2$, $\III_3$, and $\mcz^k_1$ are the functions appearing in \eqref{eqrho3}.
In light of Lemmas \ref{s43} and \ref{s44}, \eqref{z1p}, and \eqref{z3p}, we infer that
\begin{align*}
\sum_{k=0}^N |\mtc_k| &\lesssim \vsi_1(\delta_1) \\
&:= \begin{cases}
\delta_1^{\frac{N-2}{2}} &\text{if } [N \ge 6 \text{ and } (M,g) \text{ is l.c.f.}] \text{ or } [6 \le N \le 10 \text{ and } (M,g) \text{ is non-l.c.f.}],\\
\delta_1^4 &\text{if } 11 \le N \le 13 \text{ and } (M,g) \text{ is non-l.c.f.}
\end{cases}
\end{align*}

We write $u_* = u_0+\mcv_1+\rho$ so that
\[f := \mcl_g u_*-u_*^{2^*-1} = \sum_{k=0}^N \mtc_k \mcl_g\mcz^k_1,\]
which implies
\[\Gamma(u_*) = \|f\|_{H^{-1}(M)} \lesssim \sum_{k=0}^N |\mtc_k| \lesssim \vsi_1(\delta_1).\]
Similar to Lemma \ref{le4.2}, we have the followings:
\begin{itemize}
\item[-] The function $\rho$ can be decomposed into $\rho = \trh_0+\trh_1$ where $\trh_0$ and $\trh_1$ solve \eqref{trho0} and \eqref{trho1}, respectively.
\item[-] If $6 \le N \le 13$ or [$N \ge 14$ and $(M,g)$ is l.c.f.], then \eqref{tr0} and \eqref{cj} are valid.
\item[-] By \eqref{123} and \eqref{iii3}, it holds that
\begin{equation}\label{hp1}
\begin{aligned}
&\ \|\trh_1\|_{H^1(M)} \lesssim \|f\|_{H^{-1}(M)} + \|\III_3\|_{L^{\frac{2N}{N+2}}(M)} + \delta_1^{\frac{N+2}{2}} \\
&\lesssim \begin{cases}
\delta_1^{\frac{N-2}{2}} &\text{if } [N \ge 6 \text{ and } (M,g) \text{ is l.c.f.}] \text{ or } [6 \le N \le 9 \text{ and } (M,g) \text{ is non-l.c.f.}],\\
\delta_1^4|\log\delta_1|^{\frac35} &\text{if } N=10 \text{ and } (M,g) \text{ is non-l.c.f.,}\\
\delta_1^4 &\text{if } 11 \le N \le 13 \text{ and } (M,g) \text{ is non-l.c.f.}
\end{cases}
\end{aligned}
\end{equation}
\end{itemize}
By \eqref{l4u1} and \eqref{eq:cr43}, we observe
\begin{equation}\label{epx}
\|\rho\|_{H^1(M)} \lesssim \vsi_2(\delta_1) := \begin{cases}
\delta_1^2|\log \delta_1|^{\frac12} &\text{if } N=6,\\
\delta_1^{\frac{N+2}{4}} &\text{if } 7 \le N \le 13 \text{ or } [N \ge 14 \text{ and } (M,g) \text{ is l.c.f.}].
\end{cases}
\end{equation}

\medskip \noindent \textbf{Step 2.} We assert that
\[\|\rho\|_{H^1(M)} \gtrsim \begin{cases}
\Gamma(u_*)|\log\Gamma(u_*)|^{\frac12} &\text{if } N=6,\\
\Gamma(u_*)^{\frac{N+2}{2(N-2)}} &\text{if } 7 \le N \le 10 \text{ or } [N \ge 11 \text{ and } (M,g)\text{ is l.c.f.}],\\
\Gamma(u_*)^{\frac{N+2}{16}} &\text{if } 11 \le N \le 13 \text{ and } (M,g) \text{ is non-l.c.f.}
\end{cases}\]
By \eqref{hp1}, it is enough to show that
\[\|\trh_0\|_{H^1(M)} \gtrsim \vsi_2(\delta_1).\]

\medskip
Taking into account \eqref{trho0}, we obtain a representation formula
\[\trh_0(x) = \int_M G_g(x,z)\left[(2^*-1)(u_0+\mcv_1)^{2^*-2}\trh_0 + \III_1[\trh_0] + \III_2 + \sum_{k=0}^N \tc_k \mcl_g\mcz^k_1\right](z) (dv_g)_z\]
for $x \in M$. Employing the conformal change $g_{\xi_1} = \Lambda_{\xi_1}^{4/(N-2)}g$ and \eqref{cova}, we deduce
\[G_g(x,z) = \Lambda_{\xi_1}(x) G_{g_{\xi_1}}(x,z) \Lambda_{\xi_1}(z) \lesssim \frac{1}{d_{g_{\xi_1}}(x,z)^{N-2}}.\]
Let
\[\vsi_3(x) = \begin{cases}
\begin{medsize}
\displaystyle \int_M \frac{1}{d_{g_{\xi_1}}(x,z)^4} \left[\delta_1 \(\frac{\delta_1}{\delta_1^2+d_{g_{\xi_1}}(z,\xi_1)^2}\)^3 \bs{1}_{d_{g_{\xi_1}}(z,\xi_1) \le \sqrt{\delta_1}}
+ \frac{\delta_1^2}{\delta_1^2+d_{g_{\xi_1}}(z,\xi_1)^2} \bs{1}_{d_{g_{\xi_1}}(z,\xi_1) \ge \sqrt{\delta_1}}\right] (dv_g)_z
\end{medsize} \\
\begin{medsize}
\hspace{360pt} \text{if } N=6,
\end{medsize} \\
\begin{medsize}
\displaystyle \delta_1^2 \left[\(\frac{\delta_1}{\delta_1^2+d_{g_{\xi_1}}(x,\xi_1)^2}\)^2 \bs{1}_{d_{g_{\xi_1}}(x,\xi_1) \le \sqrt{\delta_1}}
+ \(\frac{\delta_1}{\delta_1^2+d_{g_{\xi_1}}(x,\xi_1)^2}\)^{\frac{N-6}{2}} \bs{1}_{d_{g_{\xi_1}}(x,\xi_1) \ge \sqrt{\delta_1}}\right]
\end{medsize} \\
\begin{medsize}
\hspace{215pt} \text{if } 7 \le N \le 13 \text{ or } [N \ge 14 \text{ and } (M,g) \text{ is l.c.f.}]
\end{medsize}
\end{cases}\]
for $x \in M$. Using \eqref{tr0}, \eqref{cj}, and Lemma \ref{a1}, it is not hard to check that
\begin{align*}
&\begin{medsize}
\displaystyle \left|\int_M G_g(x,z)(u_0+\mcv_1)^{2^*-2}(z)\trh_0(z) (dv_g)_z\right|
\end{medsize} \\
&\begin{medsize}
\displaystyle \lesssim \int_M \frac{\delta_1}{d_{g_{\xi_1}}(x,z)^{N-2}} \left[\(\frac{\delta_1}{\delta_1^2+d_{g_{\xi_1}}(z,\xi_1)^2}\)^3 \bs{1}_{d_{g_{\xi_1}}(z,\xi_1) \le \sqrt{\delta_1}}
+ \(\frac{\delta_1}{\delta_1^2+d_{g_{\xi_1}}(z,\xi_1)^2}\)^{\frac{N-4}{2}} \bs{1}_{d_{g_{\xi_1}}(z,\xi_1) \ge \sqrt{\delta_1}}\right] (dv_g)_z
\end{medsize} \\
&\begin{medsize}
\displaystyle \lesssim \vsi_3(x),
\end{medsize}
\end{align*}
\begin{align*}
&\ \left|\int_M G_g(x,z)\III_1[\trh_0](z) (dv_g)_z\right| \lesssim \int_M G_{g_{\xi_1}}(x,z)|\trh_0|^{2^*-1}(z) (dv_g)_z \\
&\lesssim \delta_1^{2^*} \(\frac{\delta_1}{\delta_1^2+d_{g_{\xi_1}}(x,\xi_1)^2}\)^{2^*-2} \bs{1}_{d_{g_{\xi_1}}(x,\xi_1) \le \sqrt{\delta_1}}
+ \delta_1^{2^*} \(\frac{\delta_1}{\delta_1^2+d_{g_{\xi_1}}(x,\xi_1)^2}\)^{\frac{N^2-4N-4}{2(N-2)}} \bs{1}_{d_{g_{\xi_1}}(x,\xi_1) \ge \sqrt{\delta_1}},
\end{align*}
and
\begin{align*}
\left|\sum_{k=0}^N \tc_k \int_M G_g(x,z) \(\mcl_g\mcz^k_1\)(z) (dv_g)_z\right| &\lesssim \sum_{k=0}^N |\tc_k| \(\frac{\delta_1}{\delta_1^2+d_{g_{\xi_1}}(x,\xi_1)^2}\)^{\frac{N-2}{2}} \\
&\lesssim \delta_1^{\frac{N-2}{2}} \(\frac{\delta_1}{\delta_1^2+d_{g_{\xi_1}}(x,\xi_1)^2}\)^{\frac{N-2}{2}}.
\end{align*}
Putting all the information above together, we arrive at
\[\trh_0(x) = \int_M G_g(x,z)\III_2(z) (dv_g)_z + \mfp(x) \quad \text{for } x \in M \quad \text{where } |\mfp(x)| \lesssim \vsi_3(x).\]

Meanwhile, testing \eqref{rhoex} with $\rho$ reveals
\begin{align*}
\|\rho\|_{H^1(M)}^2 &= \int_M \left[(2^*-1)(u_0 +\mcv_1)^{2^*-2}\rho +\III_1[\rho] + \III_2 + \III_3\right]\rho dv_g\\
&\ge \int_M (\III_1[\rho] + \III_2 + \III_3)\rho dv_g.
\end{align*}
Using \eqref{123}, \eqref{iii3}, and \eqref{epx}, we verify
\[\left|\int_M (\III_1[\rho]+\III_3)\rho dv_g\right| \lesssim \|\rho\|_{H^1(M)}^{2^*} + \|\III_3\|_{L^{\frac{2N}{N+2}}(M)}\|\rho\|_{H^1(M)} = o\(\vsi_2(\delta_1)^2\).\]
By \eqref{4i32} and \eqref{hp1}, we also find
\[\left|\int_M \III_2\trh_1 dv_g\right| \lesssim \|\III_2\|_{L^{\frac{2N}{N+2}}(M)}\|\trh_1\|_{H^1(M)} = o\(\vsi_2(\delta_1)^2\).\]
We turn to estimate $\int_M \III_2\trh_0 dv_g$. Note that
\begin{equation}\label{gxy}
G_g(x,z) \gtrsim G_{g_{\xi_1}}(x,z)\gtrsim \frac{1}{d_{g_{\xi_1}}(x,z)^{N-2}}>0 \quad \text{for } x,z \in B^{g_{\xi_1}}_{r_0}(\xi_1),\, x \ne z
\end{equation}
and
\begin{equation}\label{eq:III2}
\begin{aligned}
\III_2(x) &= \left[(u_0+\mcv_1)^{2^*-1}-u_0^{2^*-1}-\mcv_1^{2^*-1}\right](x) \\
&\ge (2^*-1) \left[u_0\mcv_1^{2^*-2} + u_0^{2^*-2}\mcv_1\right](x) &&\text{(by the binomial theorem)},\\
&\gtrsim \(\mcu_1^{g_{\xi_1}}\)^{2^*-2}(x) \bs{1}_{d_{g_{\xi_1}}(x,\xi_1) \le \sqrt{\delta_1}} &&\text{(by } u_0 \gtrsim 1) \quad \text{for } x \in M.
\end{aligned}
\end{equation}
Straightforward computations relying on \eqref{gxy}, \eqref{eq:III2}, and \eqref{eq:dist1} show that
\begin{align*}
&\ \int_M\int_M\III_2(z)G_g(x,z)\III_2(x) (dv_g)_x(dv_g)_z\\
&\gtrsim \int_{d_{g_{\xi_1}}(z,\xi_1) \le \sqrt{\delta_1}} \int_{d_{g_{\xi_1}}(x,\xi_1) \le \sqrt{\delta_1}}\frac{1}{d_{g_{\xi_1}}(x,z)^{N-2}} \(\mcu_1^{g_{\xi_1}}\)^{2^*-2}(x)\(\mcu_1^{g_{\xi_1}}\)^{2^*-2}(z) (dv_g)_x(dv_g)_z\\
&\gtrsim \delta_1^{N-2} \int_{\{|y| \lesssim \frac{1}{\sqrt{\delta_1}}\}} \frac{dy}{(1+|y|^2)^3} \gtrsim \begin{cases}
\delta_1^4|\log\delta_1| &\text{if } N=6,\\
\delta_1^{\frac{N+2}{2}} &\text{if } N \ge 7.
\end{cases}
\end{align*}
Furthermore, by applying \eqref{i32} and the bound
\begin{align*}
&\begin{medsize}
\displaystyle \ \int_M \(\frac{\delta_1}{\delta_1^2+d_{g_{\xi_1}}(x,\xi_1)^2}\)^2 \vsi_3(x) (dv_g)_x
\end{medsize} \\
&\begin{medsize}
\displaystyle = \int_M\left[\delta_1 \(\frac{\delta_1}{\delta_1^2+d_{g_{\xi_1}}(z,\xi_1)^2}\)^3 \bs{1}_{d_{g_{\xi_1}}(z,\xi_1) \le \sqrt{\delta_1}}
+ \frac{\delta_1^2}{\delta_1^2+d_{g_{\xi_1}}(z,\xi_1)^2} \bs{1}_{d_{g_{\xi_1}}(z,\xi_1) \ge \sqrt{\delta_1}}\right]
\end{medsize} \\
&\begin{medsize}
\displaystyle \hspace{100pt} \times \int_M \frac{1}{d_{g_{\xi_1}}(x,z)^4} \(\frac{\delta_1}{\delta_1^2+d_{g_{\xi_1}}(x,\xi_1)^2}\)^2 (dv_g)_x(dv_g)_z \quad \text{(by the Fubini-Tonelli theorem)}
\end{medsize} \\
&\begin{medsize}
\lesssim \delta_1^4 \quad \text{if } N=6,
\end{medsize}
\end{align*}
one computes
\[\left|\int_M \III_2\mfp dv_g\right| \lesssim \begin{cases}
\delta_1^{N-2} &\text{if } N=6,\, 7,\\
\delta_1^6|\log\delta_1| &\text{if } N=8,\\
\delta_1^{\frac{N+4}{2}} &\text{if } 9 \le N \le 13 \text{ or } [N \ge 14 \text{ and } (M,g) \text{ is l.c.f.}].
\end{cases}\]
This proves the claim.

\medskip \noindent \textbf{Step 3.} Finally, by adapting Step 2 of the previous subsection, one can obtain that
\[\inf\left\{\left\|u_*-\(u_0+\mcv_{\tde_1,\txi_1}\)\right\|_{H^1(M)}: \(\tde_1, \txi_1\) \in (0,\infty) \times M\right\} \gtrsim \|\rho\|_{H^1(M)}.\]
This establishes the optimality of \eqref{eq:sqe3}.

\subsection{Optimality of \eqref{eq:sqe4}}\label{se5.3}
If $N \ge 7$ and $(M,g)$ is non-l.c.f., that is, when $\zeta(t) = t$, we can slightly modify the argument in Subsection \ref{se5.1} to prove the optimality of \eqref{eq:sqe4}.
Therefore, we only need to consider the cases when $N=6$ or [$N \ge 7$ and $(M,g)$ is l.c.f.].
Because we can follow the steps of Subsection \ref{se5.2} with Lemma \ref{le4.8} in hand, we will only highlight the differences between the previous and current settings.

\medskip \noindent \textbf{Step 1.} On account of Proposition \ref{pr47}, we have
\begin{equation}\label{eq:ckest}
\begin{aligned}
\Gamma(u_*) = \|f\|_{H^{-1}(M)} &= \left\|\sum_{k=0}^N \mtc_k \mcl_g\mcz^k_1\right\|_{H^{-1}(M)} \\
&\lesssim \sum_{k=0}^N|\mtc_k| \lesssim \begin{cases}
\delta_1^4|\log\delta_1| &\text{if } N = 6 \text{ and } (M,g) \text{ is non-l.c.f.},\\
\delta_1^{N-2} &\text{if } N \ge 6 \text{ and } (M,g) \text{ is l.c.f.}
\end{cases}
\end{aligned}
\end{equation}

\medskip \noindent \textbf{Step 2.} It holds that
\begin{align*}
\trh_0(x) &= \int_M G_g(x,z) \left[(2^*-1)\mcv_1^{2^*-2}\trh_0 + \III_1[\trh_0] + \III_3 + \sum_{k=0}^N \tc_k \mcl_g\mcz^k_1\right](z) (dv_g)_z\\
&= \int_M G_g(x,z)\III_3(z)(dv_g)_z + \mfp(x) \quad \text{for } x \in M
\end{align*}
with
\[|\mfp(x)| \lesssim \begin{cases}
\begin{medsize}
\displaystyle \delta_1^6 \int_M \frac{1}{d_{g_{\xi_1}}(x,z)^4} \frac{\log (2+d_{g_{\xi_1}}(z,\xi_1))}{(\delta_1^2+d_{g_{\xi_1}}(z,\xi_1)^2)^3} (dv_g)_z
\end{medsize}
&\begin{medsize}
\text{if } N=6 \text{ and } (M,g) \text{ is non-l.c.f.},
\end{medsize} \\
\begin{medsize}
\displaystyle \delta_1^4 \(\frac{\delta_1}{\delta_1^2+d_{g_{\xi_1}}(x,\xi_1)^2}\)^2 \log\(2+\frac{d_{g_{\xi_1}}(x,\xi_1)}{\delta_1}\)
\end{medsize}
&\begin{medsize}
\text{if } N=6 \text{ and } (M,g) \text{ is l.c.f.},
\end{medsize} \\
\begin{medsize}
\displaystyle \delta_1^{\frac{N+2}{2}} \(\frac{\delta_1}{\delta_1^2+d_{g_{\xi_1}}(x,\xi_1)^2}\)^2 \bs{1}_{d_{g_{\xi_1}}(x,\xi_1) \le \frac{r_0}{2}} + \delta_1^{\frac{N+6}{2}}\bs{1}_{d_{g_{\xi_1}}(x,\xi_1) \ge \frac{r_0}{2}}
\end{medsize}
&\begin{medsize}
\text{if } N \ge 7 \text{ and } (M,g) \text{ is l.c.f.}
\end{medsize}
\end{cases}\]
From \eqref{5i3}, we discover
\[\left|\int_M \III_3\mfp dv_g\right| \lesssim
\begin{cases}
\delta_1^8|\log\delta_1|^2 &\text{if } N=6 \text{ and } (M,g) \text{ is non-l.c.f.,}\\
\delta_1^9|\log\delta_1|^2 &\text{if } N=6 \text{ and } (M,g) \text{ is l.c.f.,}\\
\delta_1^{10} &\text{if } N=7 \text{ and } (M,g) \text{ is l.c.f.,}\\
\delta_1^{12}|\log\delta_1| &\text{if } N=8 \text{ and } (M,g) \text{ is l.c.f.,}\\
\delta_1^{N+4} &\text{if } N \ge 9 \text{ and } (M,g) \text{ is l.c.f.}\\
\end{cases}\]

Recalling \eqref{gxy} and \eqref{eq:dist1}, we consider as follows:

\medskip \noindent(1) If $N=6$ and $(M,g)$ is non-l.c.f., then \eqref{i36} implies
\begin{equation}\label{eq:6nlcf}
\begin{aligned}
&\ \int_M\int_M \III_3(z)G_g(x,z)\III_3(x) (dv_g)_x(dv_g)_z\\
&\gtrsim \int_{B_{r_0/2}(0)} \int_{B_{r_0/2}(0)} \frac{1}{|y_1-y_2|^4} \frac{\delta_1^4|y_2|^2\log|y_2|}{(\delta_1^2+|y_2|^2)^3} \frac{\delta_1^4|y_1|^2\log|y_1|}{(\delta_1^2+|y_1|^2)^3} dy_1 dy_2 + \delta_1^8|\log\delta_1|^2\\
&\gtrsim \delta_1^8|\log\delta_1|^3.
\end{aligned}
\end{equation}

\medskip \noindent(2) If $N \ge 6$ and $(M,g)$ is l.c.f., then \eqref{ca13} and \eqref{i3ex} imply that $\III_3(x) \ge 0$ for $d_{g_{\xi_1}}(x,\xi_1) \ge r_0$,
\[\III_3\big(\exp_{\xi_1}^{g_{\xi_1}} y\big) \gtrsim \frac{\delta_1^{\frac{N+2}{2}}|y|^{N-2}}{(\delta_1^2+|y|^2)^{\frac{N+2}{2}}} \quad \text{for } |y| \le \tfrac{r_0}{2},\]
and
\begin{align*}
\III_3(x) &= \mfq(x) + \Lambda_{\xi_1}^{2^*-1}(x) \left[(\Delta_{g_{\xi_1}} \chi)(x) \left\{G_{g_{\xi_1}}(x,\xi_1) \(d_{g_{\xi_1}}(x,\xi_1)^{N-2}\mcu_1^{g_{\xi_1}}(x) - \alpha_N \delta_1^{\frac{N-2}{2}}\)\right\} \right. \\
&\ \left. + 2\la (\nabla_{g_{\xi_1}}\chi)(x), \nabla_{g_{\xi_1}}\left\{G_{g_{\xi_1}}(x,\xi_1) \(d_{g_{\xi_1}}(x,\xi_1)^{N-2}\mcu_1^{g_{\xi_1}}(x) -\alpha_N\delta_1^{\frac{N-2}{2}}\)\right\} \ra_{g_{\xi_1}} \right] \\
&=\mfq(x)+ \mco\Big(\delta_1^{\frac{3(N-2)}{2}}\Big) \quad \text{ for } \tfrac{r_0}{2} \le d_{g_{\xi_1}}(x,\xi_1)\le r_0
\end{align*}
where $\mfq$ is a nonnegative function on $M$. These result in
\begin{align*}
&\ \int_M\int_M \III_3(z)G_g(x,z)\III_3(x) (dv_g)_x(dv_g)_z\\
&\gtrsim \int_{B_{r_0/2}(0)}\int_{B_{r_0/2}(0)} \frac{1}{|y_1-y_2|^{N-2}} \frac{\delta_1^{\frac{N+2}{2}}|y_2|^{N-2}}{\(\delta_1^2+|y_2|^2\)^{\frac{N+2}{2}}}
\frac{\delta_1^{\frac{N+2}{2}}|y_1|^{N-2}}{\(\delta_1^2+|y_1|^2\)^{\frac{N+2}{2}}} dy_1 dy_2 + \mco\Big(\delta_1^{\frac{N+2}{2}+\frac{3(N-2)}{2}}\Big) \\
&\gtrsim \begin{cases}
\delta_1^8|\log\delta_1| &\text{if } N=6,\\
\delta_1^{N+2} &\text{if } N \ge 7.
\end{cases}
\end{align*}

\medskip \noindent \textbf{Step 3.} Combining the above computations and adapting Step 2 of Subsection \ref{se5.1}, one can derive the sharpness of \eqref{eq:sqe4}.
In particular, if $N=6$ and $(M,g)$ is non-l.c.f., then we deduce $\|\rho\|_{H^1(M)} \gtrsim \delta_1^4|\log\delta_1|^{3/2}$ from \eqref{eq:6nlcf}.
Because \eqref{eq:ckest} tells us that $\Gamma(u_*) \le C_2\delta_1^4|\log\delta_1|$ for some $C_2 > 0$, we conclude
\[\Gamma(u_*)|\log \Gamma(u_*)|^{\frac12} \le C_2\delta_1^4|\log\delta_1| \left|\log C_2 \delta_1^4|\log\delta_1|\right|^{\frac12} \lesssim \delta_1^4|\log\delta_1|^{\frac32} \lesssim \|\rho\|_{H^1(M)}.\]

\section{Proof of Corollary \ref{cor:main}}\label{se6}
We will only show the proof of (1) in Corollary \ref{cor:main} here, because the proofs of (2) and (3) are similar.

\begin{proof}[Proof of \eqref{eq:cor}]
Since $\|u\|^2_{H^1(M)} \le (\nu_0+\frac12)S^N$, we have
\begin{align*}
\Gamma(u) = \left\|\mcl_gu-u^{2^*-1}\right\|_{H^{-1}(M)} &\le\sup_{\|\vph\|_{H^1(M)}=1} \left|\la u,\vph \ra_{H^1(M)} - \int_M u^{2^*-1}\vph dv_g\right| \\
&\lesssim \|u\|_{H^1(M)}+\|u\|_{H^1(M)}^{2^*-1} \lesssim 1.
\end{align*}
Let $\Theta(u)$ be the left-hand side of \eqref{eq:cor}. Because \eqref{u00} has a trivial solution, $\Theta(u)$ is bounded by a positive constant depending only on $N$ and $\nu_0$.

Suppose that \eqref{eq:cor} is false. Then there are sequences $\{u_n\}_{n \in \N} \subset H^1(M)$ and $\{C_n\}_{n \in \N} \subset (0,\infty)$ such that
\[\|u_n\|^2_{H^1(M)} \le \(\nu_0+\tfrac12\)S^N,\ \Theta(u_n) \ge C_n\Gamma(u_n),\ \Gamma(u_n) > 0, \text{ and } C_n \to \infty \text{ as } n \to \infty,\]
since $\Gamma(u)=0$ implies that $\Theta(u)=0$. In particular, $\Gamma(u_n) \to 0$ as $n \to \infty$. By Theorem \ref{thm:Struwe},
there exist $\nu \in \{0,1,\ldots,\nu_0\}$, a smooth solution $u_0$ to \eqref{u00} with $c=1$, and a sequence of bubble-like functions $\{(\mcv_{1n},\ldots,\mcv_{\nu n})\}_{n \in \N}$ such that \eqref{eq:inter} and \eqref{eq:inter2} hold.
We may choose each $\mcv_{in}$ by \eqref{ui} if $u_0>0$ and by \eqref{vi} if $u_0=0$. Let us consider the following two cases: $\nu > 0$ and $\nu = 0$.

\medskip \noindent (1) If $\nu > 0$, then there exists $n_0 \in \N$ large such that $\{u_n\}_{n \in \N,\, n \ge n_0}$ fulfills Assumption \ref{assum}. Hence, by Theorems \ref{th1.2} and \ref{th1.3}, there exists a constant $C > 0$ independent of $n$ such that
\[\Theta(u_n) \le \left\|u_n-\bigg(u_0+\sum_{i=1}^{\nu} \mcv_{in}\bigg)\right\|_{H^1(M)} \le C\Gamma(u_n).\]

\noindent (2) If $\nu=0$, then $u_n \to u_0$ strongly in $H^1(M)$ as $n \to \infty$. By following the strategy in Section \ref{se2}, we will derive $\Theta(u_n) \le C\Gamma(u_n)$ for some $C > 0$ independent of $n$.

For the moment, we assume that $u_0 > 0$ on $M$. Setting $\rho_n = u_n-u_0$ and $f_n = \mcl_gu_n-u_n^{2^*-1}$, we decompose $\rho_n$ as
\[\rho_n = \rho_{1n} + \sum_{m=1}^L \vth_{mn} \psi_m \quad \text{where } \vth_{mn} \in \R,\, \rho_{1n} \perp \{\psi_m: m=1,\ldots,L\}\]
where $\psi_m$ is the function defined in Subsection \ref{subsec:set}. Since $u_0$ is non-degenerate by the hypothesis, there exists a constant $c_0 \in (0,1)$ such that
\begin{equation}\label{cos6}
(2^*-1)\int_M u_0^{2^*-2}\rho_{1n}^2 dv_g \le c_0\|\rho_{1n}\|^2_{H^1(M)} \quad \text{for all } n \in \N.
\end{equation}
Notice also that
\begin{equation}\label{nern}
\mcl_g\rho_n-(2^*-1)u_0^{2^*-2}\rho_n = f_n+(u_0+\rho_n)^{2^*-1}-u_0^{2^*-1}-(2^*-1)u_0^{2^*-2}\rho_n.
\end{equation}
By testing \eqref{nern} with $\rho_{1n}$ and using \eqref{cos6}, one can show
\[\|\rho_n\|_{H^1(M)} \lesssim \|\rho_{1n}\|_{H^{-1}(M)}+\sum_{m=1}^L |\vth_{mn}| \lesssim \|f_n\|_{H^{-1}(M)}+\sum_{m=1}^L|\vth_{mn}|.\]
Furthermore, for any $s \in \{1,\ldots,L\}$, testing \eqref{nern} with $\psi_s$ yields
\[(2^*-1-\bmu_s)|\vth_{sn}| \int_M u_0^{2^*-2}\psi_s^2 dv_g \lesssim \|f_n\|_{H^{-1}(M)} \quad \text{with } \bmu_s \in (0, 2^*-1),\]
so $|\vth_{sn}| \lesssim \|f_n\|_{H^{-1}(M)}$. This gives
\[\Theta(u_n) \le \|\rho_n\|_{H^1(M)} \le C\|f_n\|_{H^{-1}(M)}=C\Gamma(u_n).\]
The case $u_0 = 0$ is easier to handle.

\medskip
In both cases, we obtain
\[\infty \leftarrow C_n\le \frac{\Theta(u_n)}{\Gamma(u_n)}\le C \quad \text{as } n \to \infty,\]
which is absurd. Consequently, \eqref{eq:cor} holds.
\end{proof}

\bigskip \noindent {\small \textbf{Acknowledgement.}
H. Chen and S. Kim were supported by Basic Science Research Program through the National Research Foundation of Korea (NRF) funded by the Ministry of Science and ICT (2020R1C1C1A0\\
101013314).
We sincerely thank Bruno Premoselli and Saikat Mazumdar for providing valuable comments and references.}

\appendix
\section{Some useful estimates}\label{a}
We recall that $(M,g)$ is a smooth closed Riemannian manifold of dimension $N \ge 3$.

\medskip The next two elementary lemmas result from straightforward calculations.
\begin{lemma}
Assume that $a, b>0$. It holds that
\begin{equation}\label{s4a}
(a+b)^p-a^p=\mco(b^p) \quad \text{for } 0 < p \le 1,
\end{equation}
\begin{equation}\label{iqu}
|(a+b)^p-a^p| \lesssim a^{p-1}b+b^p, \quad |(a+b)^p-a^p-b^p| \lesssim a^{p-1}b+ab^{p-1} \quad \text{for } p \ge 1,
\end{equation}
and
\begin{equation}\label{ab6}
(a+b)^p = a^p + pa^{p-1}b + \frac{p(p-1)}{2}a^{p-2}b^2\bs{1}_{p>2} + \mco(b^p) \quad \text{for } p \ge 1.
\end{equation}
\end{lemma}
\begin{lemma}\label{a4}
Let $\mcu_{\delta,\xi}$ be the function in \eqref{udx}, $0 < \delta < r_0$ small numbers, and $\xi \in M$. For $p > 0$, it holds that
\[\int_{B^g_{r_0}(\xi)} \mcu_{\delta,\xi}^p dv_g \lesssim \begin{cases}
\delta^{\frac{N-2}{2}p} &\text{if } 0<p<\frac{N}{N-2},\\
\delta^{\frac{N}{2}}|\log\delta| &\text{if } p=\frac{N}{N-2},\\
\delta^{N-\frac{N-2}{2}p} &\text{if } p>\frac{N}{N-2}.
\end{cases}\]
\end{lemma}

\begin{lemma}\label{a22}
Suppose that $3 \le N \le 5$. Let $\mcv_{\delta,\xi}$ be a bubble-like function defined by \eqref{ui} or \eqref{vi}. For any indices $1 \le i \ne j \le \nu$, a fixed number $\tau>0$, and nonnegative exponents $p$ and $q$ such that $p+q=2^*$, it holds that
\[\int_M \mcv_{\delta_i,\xi_i}^p\mcv_{\delta_j,\xi_j}^q dv_g \lesssim \begin{cases}
q_{ij}^{\min\{p,q\}} &\text{if } |p-q|\ge \tau,\\
q_{ij}^{\frac{N}{N-2}}|\log q_{ij}| &\text{if } p=q
\end{cases}\]
provided $q_{ij}$ in \eqref{rq} small.
\end{lemma}
\begin{proof}
By applying a change of variables, \eqref{exp}, and \eqref{gx}, and referring to the proof of \cite[(E1)--(E3)]{B} or \cite[Proposition B.2]{FG}, one can derive the above inequality.
\end{proof}
\begin{lemma}\label{a1}
Suppose that $p > 2$. Then we have
\begin{align*}
&\ \int_M \frac{1}{d_g(x,z)^{N-2}} \(\frac{\delta_i}{\delta_i^2+d_g(z,\xi_i)^2}\)^{\frac{p}{2}} (dv_g)_z \\
&\lesssim \begin{cases}
\displaystyle \delta_i^{\frac{p}{2}} \(\delta_i^2+d_g(x,\xi_i)^2\)^{-\frac{p-2}{2}} &\text{if } 2<p<N, \\
\displaystyle \delta_i^{\frac{N}{2}} \(\delta_i^2+d_g(x,\xi_i)^2\)^{-\frac{N-2}{2}} \log\(2+d_g(x,\xi_i)\delta_i^{-1}\) &\text{if } p=N, \\
\displaystyle \delta_i^{N-\frac{p}{2}} \(\delta_i^2+d_g(x,\xi_i)^2\)^{-\frac{N-2}{2}} &\text{if } p>N.
\end{cases}
\end{align*}
\end{lemma}
\begin{proof}
Refer to \cite[Lemma A.7]{DSW}.
\end{proof}

\begin{lemma}\label{aux}
For any $a, b \in (1,\frac{N}{2})$, $\delta > 0$ small, and $\xi \in M$, we have
\begin{multline*}
\int_M \frac{1}{d_g(x,z)^{N-2}} \left[\(\frac{\delta}{\delta^2+d_g(z,\xi)^2}\)^a \bs{1}_{d_g(z,\xi) \le \sqrt{\delta}} + \(\frac{\delta}{\delta^2+d_g(z,\xi)^2}\)^b \bs{1}_{d_g(z,\xi) \ge \sqrt{\delta}}\right] (dv_g)_z \\
\lesssim \delta \(\frac{\delta}{\delta^2+d_g(x,\xi)^2}\)^{a-1} \bs{1}_{d_g(x,\xi) \le \sqrt{\delta}} + \delta\(\frac{\delta}{\delta^2+d_g(x,\xi)^2}\)^{b-1} \bs{1}_{d_g(x,\xi) \ge \sqrt{\delta}}
\end{multline*}
and
\begin{multline*}
\int_M \frac{1}{d_g(x,z)^{N-2}} \left[\(\frac{\delta}{\delta^2+d_g(z,\xi)^2}\)^a \bs{1}_{d_g(z,\xi) \le \frac{r_0}{2}} + \delta^a \bs{1}_{d_g(z,\xi) \ge \frac{r_0}{2}}\right] (dv_g)_z \\
\lesssim \delta \(\frac{\delta}{\delta^2+d_g(x,\xi)^2}\)^{a-1} \bs{1}_{d_g(x,\xi) \le \frac{r_0}{2}} + \delta^a \bs{1}_{d_g(x,\xi) \ge \frac{r_0}{2}}.
\end{multline*}
\end{lemma}
\begin{proof}
The above inequalities can be proved as in the proof of \cite[Lemma 3.6]{DSW}. The details are omitted.
\end{proof}

\begin{lemma}
For $\xi \in M$ and $y_1, y_2 \in B_{r_0}(0)$ where $r_0 > 0$ is small enough, it holds that
\begin{equation}\label{eq:dist1}
d_g\big(\exp_{\xi}^g(y_1),\exp_{\xi}^g(y_2)\big)^2 = |y_1-y_2|^2 + \mco\(\(|y_1|^2+|y_2|^2\)|y_1-y_2|^2\)
\end{equation}
and
\begin{equation}\label{eq:dist2}
\nabla_{y_2}d_g\big(\exp_{\xi}^g(y_1),\exp_{\xi}^g(y_2)\big)^2 = 2(y_2-y_1) + \mco\(\(|y_1|^2+|y_2|^2\)|y_1-y_2|\).
\end{equation}
\end{lemma}
\begin{proof}
Refer to \cite[Lemma A.8]{JK}.
\end{proof}

\section{Technical computations}\label{a2}
\subsection{Proof of Proposition \ref{41}}\label{subsec:coer}
We argue by contradiction. Suppose that there exist sequences of parameters $\{(\delta_{in},\xi_{in})\}_{n \in \N}$, functions $\{\vrh_n\}_{n \in \N}$, and numbers $\{c_{0n}\}_{n \in \N} \subset (0,1]$ such that $\delta_{in} \to 0$ and $c_{0n} \to 1$ as $n \to \infty$, $\|\vrh_n\|_{H^1(M)}=1$ for all $n \in \N$,
\begin{equation}\label{2.11}
\begin{medsize}
\displaystyle \int_M \bigg(u_0+\sum_{i=1}^{\nu}\mcv_{in}\bigg)^{2^*-2}\vrh_n^2 dv_g = \sup \bigg\{\int_M \bigg(u_0+\sum_{i=1}^{\nu}\mcv_{i}\bigg)^{2^*-2}\vrh^2 dv_g :\|\vrh\|_{H^1(M)}=1\bigg\} \ge \frac{c_{0n}}{2^*-1},
\end{medsize}
\end{equation}
and
\begin{multline}\label{2.12}
\la \vrh_n,\mcv_{in} \ra_{H^1(M)} = \big\langle \vrh_n,\mcz^k_{in} \big\rangle_{H^1(M)} = \la \vrh_n,\psi_m \ra_{H^1(M)} = 0 \\
\text{for } i=1,\ldots,\nu,\ k=0,1,\ldots,N,\ m=1,\ldots,L.
\end{multline}
Here, $\mcv_{in} = \chi(d_g(\cdot,\xi_{in}))\mcu_{\delta_{in},\xi_{in}} + (1-\chi(d_g(\cdot,\xi_{in}))) U_{\delta_{in},0}(\frac{r_0}{2})$,
$\mcz^0_{in} = \delta_{in} \frac{\pa\mcv_{in}}{\pa\delta_{in}}$, and $\mcz^k_{in} = \delta_{in} \frac{\pa\mcv_{in}}{\pa\xi_{in}^k}$. By \eqref{2.11} and \eqref{2.12},
\begin{multline}\label{2.13}
\mcl_g\vrh_n -\mu_n\bigg(u_0+\sum_{i=1}^{\nu}\mcv_{in}\bigg)^{2^*-2}\vrh_n\\
= \sum_{i=1}^{\nu}\mu_{in}\mcl_g\mcv_{in} + \sum_{i=1}^{\nu}\sum_{k=0}^N \mu_{in}^k\mcl_g\mcz^k_{in} + \sum_{m=1}^L\tmu_{mn}u_0^{2^*-2}\psi_m \quad \text{on } M
\end{multline}
where $\mu_n, \mu_{in},\mu_{in}^k, \tmu_{mn}\in \R$ are Lagrange multipliers. Testing \eqref{2.13} with $\vrh_n$ and applying \eqref{2.12}, we arrive at
\[\mu_n = \left[\int_M \bigg(u_0+\sum_{i=1}^{\nu}\mcv_{in}\bigg)^{2^*-2}\vrh_n^2 dv_g\right]^{-1} \in [c(\nu,N,L), c_{0n}^{-1}(2^*-1)]\]
where the lower bound $c(\nu,N,L)$ is positive and dependent only on $\nu$, $N$, and $L$. Hence we may assume that $\mu_n \to \mu_{\infty} \in [c(\nu,N,L), c_{0n}^{-1}(2^*-1)]$ as $n \to \infty$.
	
Let $q_{ij,n}$, $\mcq_n$, $\msr_{ij,n}$ be the quantities introduced in \eqref{rq} where $(\xi_i,\xi_j,\delta_i,\delta_j)$ is replaced with $(\xi_{in},\xi_{jn},\delta_{in},\delta_{jn})$.
We present the rest of the proof by dividing it into four steps.

\medskip \noindent \textbf{Step 1.} We claim that
\begin{equation}\label{2.15}
\sum_{i=1}^{\nu}|\mu_{in}| + \sum_{i=1}^{\nu}\sum_{k=0}^N |\mu_{in}^k| + \sum_{m=1}^L |\tmu_{mn}| = o(1)
\end{equation}
where $o(1) \to 0$ as $n \to \infty$. To prove it, we argue as in the proof of Lemma \ref{le24}.

\medskip
Firstly, we test \eqref{2.13} with $\mcv_{jn}$ for $j \in \{1,\ldots,\nu\}$ and employ \eqref{2.12} to get
\[-\mu_n \int_M \bigg(u_0+\sum_{i=1}^{\nu}\mcv_{in}\bigg)^{2^*-2}\vrh_n \mcv_{jn} dv_g = \bigg\langle \sum_{i=1}^{\nu} \mu_{in}\mcv_{in} + \sum_{i=1}^{\nu}\sum_{k=0}^N \mu_{in}^k\mcz^k_{in} + \sum_{m=1}^L \tmu_{mn}\psi_m,\mcv_{jn} \bigg\rangle_{H^1(M)}.\]
Thus we infer from Lemma \ref{le2p} and \eqref{teu} that
\begin{align}
|\mu_n| \mco\Big(\mcq_n + \max_{\ell}\delta_{\ell n}^{\frac{N-2}{2}}\Big) &= |\mu_{jn}| \int_{\R^N} U^{2^*} + \(|\mu_{jn}| + \sum_{k=0}^N |\mu_{jn}^k|\) o\Big(\delta_{jn}^{\frac{N-2}{2}}\Big) + \sum_{m=1}^L|\tmu_{mn}| \mco\Big(\delta_{jn}^{\frac{N-2}{2}}\Big) \nonumber \\
&\ + \left[\sum_{i \ne j} |\mu_{in}| + \sum_{i \ne j}\sum_{k=0}^N |\mu_{in}^k|\right] \cdot \left[\mco(\mcq_n) + o\Big(\max_{\ell}\delta_{\ell n}^{\frac{N-2}{2}}\Big)\right] \label{2.16}
\end{align}
as $n \to \infty$.

Secondly, by testing \eqref{2.13} with $\mcz^q_{jn}$ for any $j \in \{1,\ldots,\nu\}$ and $q \in \{0,1,\ldots,N\}$, we deduce
\begin{equation}\label{2.17}
|\mu_{jn}^q| \(\int_{\R^N}|\nabla Z^q|^2 + o(1)\) = o(1)\left[\sum_{i=1}^{\nu}|\mu_{in}| + \sum_{(i,k) \ne (j,q)} |\mu_{in}^k| + \sum_{m=1}^L|\tmu_{mn}|\right] + o(1).
\end{equation}
	
Finally, we test \eqref{2.13} with $\psi_s$ for $s \in \{1,\ldots,L\}$. According to Lemma \ref{le2p}, \eqref{psij} and \eqref{ccp}, it holds that
\begin{equation}\label{2.18}
|\tmu_{sn}| \int_Mu_0^{2^*-2}\psi_s^2 dv_g = \mco\Big(\max_{\ell}\delta_{\ell n}^{\frac{N-2}{2}}\Big) \left[\sum_{i=1}^{\nu} |\mu_{in}| + \sum_{i=1}^{\nu}\sum_{k=0}^N |\mu_{in}^k|\right] + \mco\Big(\max_{\ell}\delta_{\ell n}^{\frac{N-2}{2}}\Big).
\end{equation}

Claim \eqref{2.15} now follows from \eqref{2.16}, \eqref{2.17} and \eqref{2.18}.

\medskip \noindent \textbf{Step 2.} We assert that
\begin{equation}\label{2.19}
\begin{cases}
\vrh_n \rightharpoonup 0 &\text{weakly in } H^1(M),\\
\vrh_n \to 0 &\text{strongly in } L^p(M) \text{ for } p \in (1, 2^*)
\end{cases}
\quad \text{as } n \to \infty.
\end{equation}

\medskip
Since $\|\vrh_n\|_{H^1(M)}=1$, there exists $\vrh_{\infty} \in H^1(M)$ such that
\[\begin{cases}
\vrh_n \rightharpoonup \vrh_{\infty} &\text{weakly in } H^1(M),\\
\vrh_n \to \vrh_{\infty} &\text{strongly in } L^p(M) \text{ for } p \in (1, 2^*)
\end{cases}
\quad \text{as } n \to \infty,\]
up to a subsequence. Given any $\vph\in C^{\infty}(M)$, we test \eqref{2.13} with $\vph$ and take the limit $n \to \infty$. As in \eqref{ccp}, we can derive
\[\int_M \left[\bigg(u_0+\sum_{i=1}^{\nu}\mcv_{in}\bigg)^{2^*-2} - u_0^{2^*-2}\right]\vrh_n\vph dv_g=o(1).\]
This fact, \eqref{2.12}, and \eqref{2.15} imply
\[\mcl_g \vrh_{\infty} = \mu_{\infty} u_0^{2^*-2}\vrh_{\infty} \quad \text{on } M \quad \text{and} \quad \la \vrh_{\infty},\psi_m \ra_{H^1(M)} = 0 \quad \text{for } m=1,\ldots,L,\]
which together with the non-degeneracy of $u_0$ and $\mu_{\infty}\in [c(\nu,N,L), 2^*-1]$ yields $\vrh_{\infty} = 0$ on $M$. This proves the assertion.

\medskip \noindent \textbf{Step 3.} For a fixed index $j \in \{1,\ldots,\nu\}$, let
\[\tvrh_{jn}(y) = \delta_{jn}^{\frac{N-2}{2}}\chi(\delta_{jn}|y|) \vrh_n\big(\exp_{\xi_{jn}}^g(\delta_{jn}y)\big) \quad \text{for any } y \in \R^N\]
provided $n \in \N$ large enough. We claim that
\begin{equation}\label{2.20}
\begin{cases}
\tvrh_{jn} \rightharpoonup 0 &\text{weakly in } \dot{H}^1(\R^N),\\
\tvrh_{jn} \to 0 &\text{strongly in } L^p_{\loc}(\R^N) \text{ for } p \in (1,2^*)
\end{cases} \quad \text{as } n \to \infty.
\end{equation}

Because $\|\vrh_n\|_{H^1(M)}=1$, the set $\{\tvrh_{jn}\}_{n \in \N}$ is bounded in $\dot{H}^1(\R^N)$.
By passing to a subsequence, we may assume that $\tvrh_{jn} \rightharpoonup \tvrh_{j\infty}$ weakly in $\dot{H}^1(\R^N)$ and $\tvrh_{jn} \to \tvrh_{j\infty}$ strongly in $L^p_{\loc}(\R^N)$ for all $p \in (1,2^*)$.
Given a function $\vph \in C^{\infty}_c(\R^N)$, we set
\[\tvph_{jn}(x) = \chi(d_g(x,\xi_{jn})) \delta_{jn}^{\frac{2-N}{2}} \vph\(\delta_{jn}^{-1}\big(\exp^g_{\xi_{jn}}\big)^{-1}(x)\) \quad \text{for } x \in M.\]
Testing \eqref{2.13} with $\tvph_{jn}$, we obtain
\begin{multline}\label{tpv}
\int_M \left[\la\nabla_g\vrh_n, \nabla_g\tvph_{jn}\ra_g + \ka_NR_g\vrh_n\tvph_{jn} - \mu_n\bigg(u_0+\sum_{i=1}^{\nu}\mcv_{in}\bigg)^{2^*-2}\vrh_n\tvph_{jn}\right] dv_g \\
= \bigg\langle \sum_{i=1}^{\nu} \mu_{in}\mcv_{in} + \sum_{i=1}^{\nu}\sum_{k=0}^N \mu_{in}^k\mcz^k_{in} + \sum_{m=1}^L \tmu_{mn} \psi_m,\tvph_{jn} \bigg\rangle_{H^1(M)}.
\end{multline}
It holds that $\|\tvph_{jn}\|_{H^1(M)} \le C$, so
\[\begin{cases}
\displaystyle \int_M \mcv_{jn}^{2^*-2} \vrh_n \tvph_{jn} dv_g = \int_{\{|y| \le \frac{r_0}{2\delta_{jn}}\}} U^{2^*-2}\tvrh_{jn}\vph
+ \mco\Big(\delta_{jn}^2\Big) = \int_{\R^N} U^{2^*-2}\tvrh_{j\infty}\vph + o(1),\\
\displaystyle \int_Mu_0^{2^*-2}\vrh_n\tvph_{jn}dv_g \simeq \delta_{jn}^2 \int_{\supp(\vph)} u_0^{2^*-2}(\exp_{\xi_{jn}}^g\big(\delta_{jn}y)\big) (\tvrh_{jn}{\vph})(y) dy = o(1).
\end{cases}\]
Also, if $d_g(\xi_{in},\xi_{jn}) < \frac{r_0}{4}$, then
\[\left|\int_M \mcv_{in}^{2^*-2}\vrh_n\tvph_{jn} dv_g\right| \lesssim \left\|\left[\delta_{jn}^{\frac{N-2}{2}} \mcv_{in}\(\exp_{\xi_{jn}}^g\big(\delta_{jn}\cdot\big)\)\right]^{2^*-2}\right\|_{L^{\frac{2N}{N+2}}(\supp(\vph))} = o(1) \quad \text{for } i \ne j,\]
because
\begin{align*}
&\ \(\frac{\delta_{jn}}{\delta_{in}}\)^{\frac{4N}{N+2}} \int_{\supp(\vph)} \frac{dy}{\big(1+\delta_{in}^{-2} \big|\delta_{jn}y - \big(\exp_{\xi_{jn}}^g\big)^{-1}(\xi_{in})\big|^2\big)^{\frac{4N}{N+2}}} \\
&\simeq \(\frac{\delta_{jn}}{\delta_{in}}\)^{\frac{4N}{N+2}-N} \int_{\big\{|y+\delta_{in}^{-1}(\exp_{\xi_{jn}}^g)^{-1}(\xi_{in})| \lesssim \delta_{in}^{-1}\delta_{jn}\big\}} \frac{dy}{(1+|y|^2)^{\frac{4N}{N+2}}} \\
&\lesssim \begin{cases}
\displaystyle \(\frac{\delta_{jn}}{\delta_{in}}\)^{\frac{4N}{N+2}} &\displaystyle \text{if } \lim_{n \to \infty} \frac{\delta_{jn}}{\delta_{in}}=0, \\
\displaystyle \(\frac{\delta_{jn}}{\delta_{in}}\)^{\frac{4N}{N+2}-N} + \(\frac{\delta_{jn}}{\delta_{in}}\)^{-\frac{4N}{N+2}} &\displaystyle \text{if } \lim_{n \to \infty} \frac{\delta_{jn}}{\delta_{in}}=\infty, \\
\displaystyle \msr_{ij,n}^{-{\frac{8N}{N+2}}} &\displaystyle \text{if } \lim_{n \to \infty} \frac{\delta_{jn}}{\delta_{in}} \in (0,\infty)
\end{cases}\\
&=o(1).
\end{align*}
If $d_g(\xi_{in},\xi_{jn}) \ge \frac{r_0}{4}$, then
\[\left|\int_M \mcv_{in}^{2^*-2}\vrh_n\tvph_{jn} dv_g\right|\lesssim \delta_{in}^2\int_M |\vrh_n\tvph_{jn}| dv_g=o(1) \quad \text{for } i \ne j.\]
Therefore, a reasoning analogous to \eqref{teu} demonstrates
\[\int_M \bigg(u_0+\sum_{i=1}^{\nu}\mcv_{in}\bigg)^{2^*-2}\vrh_n\tvph_{jn} dv_g = \int_{\R^N} U^{2^*-2}\tvrh_{j\infty}\vph + o(1)\]
as $n \to \infty$.

On the other hand, it is plain to verify that
\[\int_M \(\la \nabla_g\vrh_n,\nabla_g\tvph_{jn} \ra_g + \ka_NR_g\vrh_n\tvph_{jn}\) dv_g = \int_{\R^N} \nabla\tvrh_{j\infty} \cdot \nabla\vph + o(1),\]
while \eqref{2.15} guarantees
\begin{multline*}
\left|\bigg\langle \sum_{i=1}^{\nu} \mu_{in}\mcv_{in} + \sum_{i=1}^{\nu}\sum_{k=0}^N \mu_{in}^k\mcz^k_{in} + \sum_{m=1}^L \tmu_{mn}\psi_m,\tvph_{jn} \bigg\rangle_{H^1(M)}\right| \\
\lesssim \left[\sum_{i=1}^{\nu}|\mu_{in}| +\sum_{i=1}^{\nu}\sum_{k=0}^N |\mu_{in}^k| +\sum_{m=1}^L |\tmu_{mn}|\right]
\|\vph\|_{H^1(M)} = o(1).
\end{multline*}
Sending $n \to \infty$ in \eqref{tpv}, we observe from \eqref{2.12} that
\[\begin{cases}
\displaystyle -\Delta \tvrh_{j\infty} = \mu_{\infty}U^{2^*-2}\tvrh_{j\infty} \quad \text{in } \R^N, \quad \tvrh_{j\infty} \in \dot{H}^1(\R^N), \\
\displaystyle \int_{\R^N} \nabla\tvrh_{j\infty} \cdot \nabla U = \int_{\R^N} \nabla\tvrh_{j\infty} \cdot \nabla Z^k = 0 \quad \text{for all } k=0,\ldots,N.
\end{cases}\]
Because $U$ is an extremizer of the Sobolev embedding (see \eqref{eq:S}), $\tvrh_{j\infty}=0$ on $M$ as claimed.

\medskip \noindent \textbf{Step 4.} We prove that
\begin{equation}\label{2.14}
\lim_{n \to \infty} \int_M \bigg(u_0+\sum_{i=1}^{\nu}\mcv_{in}\bigg)^{2^*-2}\vrh_n^2 dv_g=0.
\end{equation}
This contradicts \eqref{2.11}, so \eqref{eq:coer} must be valid.

\medskip
We have
\[\int_M \bigg(u_0+\sum_{i=1}^{\nu}\mcv_{in}\bigg)^{2^*-2}\vrh_n^2 dv_g \lesssim \int_M u_0^{2^*-2}\vrh_n^2 dv_g + \sum_{i=1}^{\nu} \int_M \mcv_{in}^{2^*-2}\vrh_n^2 dv_g.\]
On the other hand, \eqref{2.19} gives
\[\int_M u_0^{2^*-2}\vrh_n^2 dv_g = o(1).\]
Also, we know from \eqref{2.20} that $\tvrh_{in}^2 \rightharpoonup 0$ weakly in $L^{\frac{N}{2}}(\R^N)$, so
\[\int_M \mcv_{in}^{2^*-2}\vrh_n^2 dv_g \lesssim \int_{\R^N} U^{2^*-2}\tvrh_{in}^2 + \mco(\delta_{in}^2)\|\vrh_n\|_{H^1(M)}^2 = o(1).\]
Consequently, \eqref{2.14} follows.

\subsection{Derivation of \eqref{22} and \eqref{eq:le363}}\label{subsec:tech}
We derive two estimates \eqref{22} and \eqref{eq:le363} appearing in the proofs of Lemmas \ref{le2.4} and \ref{le36}, respectively.
In this subsection, we write $d_{ij} = d_g(\xi_i,\xi_j)$, and $y_{ij} = (\exp_{\xi_i}^g)^{-1}(\xi_j)/\delta_i$ whenever it is well-defined.

\begin{proof}[Proof of \eqref{22}]
It suffices to check that $\int_M \mcv_i^{2^*-2}\mcv_j dv_g = o(\mcq)$ for $1 \le i \ne j \le \nu$. There are three possibilities:

\medskip \noindent \textbf{Case 1. ($\msr_{ij} = \frac{d_{ij}}{\sqrt{\delta_i\delta_j}}$):} It holds that $d_{ij} \ge \delta_i$ and $(\sqrt{\delta_i\delta_j}/d_{ij})^{N-2} \simeq q_{ij} \le \mcq$. Taking $x = \xi_j$ and $p = 4$ in Lemma \ref{a1}, one confirms that
\[\int_M \mcv_i^{2^*-2}\mcv_j dv_g \lesssim \left\{\!\begin{aligned}
&\delta_i \delta_j^{\frac{1}{2}} d_{ij}^{-1} &\text{if } N=3 \\[1ex]
&\delta_i^2 \delta_j d_{ij}^{-2} \log\(2+d_{ij}\delta_i^{-1}\) &\text{if } N=4 \\[1ex]
&\delta_i^2 \delta_j^{\frac{3}{2}} d_{ij}^{-2} &\text{if } N=5
\end{aligned}\right\}
= o(\mcq).\]

\medskip \noindent \textbf{Case 2. ($\msr_{ij}=\sqrt{\frac{\delta_i}{\delta_j}}$):} It holds that $d_{ij}\le \delta_i$, i.e., $|y_{ij}| \le 1$ and $({\frac{\delta_j}{\delta_i}})^{\frac{N-2}{2}} \simeq q_{ij} \le \mcq$. By \eqref{eq:dist1},
\begin{align*}
\int_M \mcv_i^{2^*-2}\mcv_j dv_g &\lesssim \int_{B^g_{r_0/2}(\xi_i)} \(\frac{\delta_i}{\delta_i^2+d_g(x,\xi_i)^2}\)^2 \(\frac{\delta_j}{\delta_j^2+d_g(x,\xi_j)^2}\)^{\frac{N-2}{2}} (dv_g)_x + \mco\Big(\delta_i^2\delta_j^{\frac{N-2}{2}}\Big) \\
&\lesssim \delta_j^{\frac{N-2}{2}} \int_{\{|y| \le \frac{r_0}{2\delta_i}\}} \frac{1}{(1+|y|^2)^2} \frac{dy}{[(\frac{\delta_j}{\delta_i})^2+|y-y_{ij}|^2]^{\frac{N-2}{2}}} + \mco\Big(\delta_i^2\delta_j^{\frac{N-2}{2}}\Big) \\
&\lesssim \delta_j^{\frac{N-2}{2}} \(1 +\int_2^{\frac{r_0}{\delta_i}} t^{-3}dt\) + \mco\Big(\delta_i^2\delta_j^{\frac{N-2}{2}}\Big) \simeq \delta_j^{\frac{N-2}{2}} = o(\mcq).
\end{align*}

\medskip \noindent \textbf{Case 3. ($\msr_{ij}=\sqrt{\frac{\delta_j}{\delta_i}}$):} It holds that $d_{ij}\le \delta_j$ and $({\frac{\delta_i}{\delta_j}})^{\frac{N-2}{2}} \simeq q_{ij} \le \mcq$. Hence
\begin{align*}
\int_M \mcv_i^{2^*-2}\mcv_j dv_g &\lesssim \frac{\delta_i^{N-2}}{\delta_j^{\frac{N-2}{2}}} \int_{\{|y| \le \frac{r_0}{2\delta_i}\}}\frac{1}{(1+|y|^2)^2}
\frac{dy}{[1+(\frac{\delta_i}{\delta_j}|y-y_{ij}|)^2]^{\frac{N-2}{2}}} + \mco\Big(\delta_i^2\delta_j^{\frac{N-2}{2}}\Big) \\
&\lesssim \frac{\delta_i^{N-2}}{\delta_j^{\frac{N-2}{2}}} \(1+\int_1^{\frac{r_0}{\delta_i}}t^{N-5}dt\) + \mco\Big(\delta_i^2\delta_j^{\frac{N-2}{2}}\Big) = o(\mcq) .
\end{align*}
Consequently, \eqref{22} is proved.
\end{proof}

\begin{proof}[Proof of \eqref{eq:le363}] Recall that $N=3$. For indices $1 \le i \ne j \le \nu$, there are three possibilities:

\medskip \noindent \textbf{Case 1. ($\msr_{ij} = \frac{d_{ij}}{\sqrt{\delta_i\delta_j}}$):} We consider two subcases separately.

\noindent \textsc{Subcase 1-1. ($d_{ij} \ge \frac{3r_0}{4}$):} We have
\[J_6 \lesssim \delta_i^{\frac32}\delta_j^{\frac12} \int_{\{|y| \le \frac{d_{ij}}{2\delta_i}\}} \frac{dy}{(1+|y|^2)^{\frac32}|y|} \lesssim \delta_i\mcq.\]

\noindent \textsc{Subcase 1-2. ($d_{ij} \le \frac{3r_0}{4}$):} If $d_g(x,\xi_i)\le d_{ij}/2$, then $d_g(x,\xi_j)\ge d_{ij}/2$. Also, if $d_g(x,\xi_i)\ge 2d_{ij}$, then $d_g(x,\xi_j)\ge d_g(x,\xi_i)/2$. Thus
\begin{align*}
J_6 &\lesssim \left[\int_{B^g_{d_{ij}/2}(\xi_i)} + \int_{B^g_{d_{ij}/2}(\xi_j)}\right]
\frac{\delta_i^{\frac52}}{(\delta_i^2+d_g(x,\xi_i)^2)^{\frac32}d_g(x,\xi_i)} \frac{\delta_j^{\frac12}}{(\delta_j^2+d_g(x,\xi_j)^2)^{\frac12}} (dv_g)_x \\
&\ +\left[\int_{\substack{\big(B_{2d_{ij}}^g(\xi_i) \setminus B_{d_{ij}/2}^g(\xi_i)\big) \\
\cap \big(B_{d_{ij}/2}^g(\xi_j)\big)^c}}
+ \int_{\substack{\big(B_{r_0/2}^g(\xi_i) \setminus B_{2d_{ij}}^g(\xi_i)\big) \\ \cap \big(B_{d_{ij}/2}^g(\xi_j)\big)^c}}\right] \frac{\delta_i^{\frac52}\delta_j^{\frac12}}{d_g(x,\xi_i)^4d_g(x,\xi_j)} (dv_g)_x \\
&\lesssim \frac{\delta_i^{\frac32}\delta_j^{\frac12}}{d_{ij}} \int_{\{|y| \le \frac{d_{ij}}{2\delta_i}\}} \frac{dy}{(1+|y|^2)^{\frac32}|y|} + \frac{\delta_i^{\frac52} \delta_j^{\frac52}}{d_{ij}^4} \int_{\{|y| \le \frac{d_{ij}}{2\delta_j}\}} \frac{dy}{(1+|y|^2)^{\frac12}} \\
&\ + \frac{\delta_i^{\frac52}\delta_j^{\frac12}}{d_{ij}} \int_{\{\frac{d_{ij}}{2}\le |y| \le 2d_{ij}\}} \frac{dy}{|y|^4}
+\delta_i^{\frac52} \delta_j^{\frac12}\int_{\{2{d_{ij}}\le |y|\le \frac{r_0}{2}\}} \frac{dy}{|y|^5} \\
&\lesssim\frac{\delta_i^{\frac32}\delta_j^{\frac12}}{d_{ij}}\lesssim \delta_i\mcq
\end{align*}
where we employed $d_{ij} \ge \max\{\delta_i,\delta_j\}$ for the third inequality.

\medskip \noindent \textbf{Case 2. ($\msr_{ij} =\sqrt{\frac{\delta_i}{\delta_j}}$):} We have
\[J_6 \lesssim \delta_i^{\frac12}\delta_j^{\frac12} \int_{\{|y| \le \frac{r_0}{2\delta_i}\}} \frac{1}{(1+|y|^2)^{\frac32}|y|} \frac{dy}{[(\frac{\delta_j}{\delta_i})^2+|y-y_{ij}|^2]^{\frac12}}
\lesssim \delta_i^{\frac12} \delta_j^{\frac12} \(1 + \int_2^{\frac{r_0}{\delta_i}} t^{-3}dt\) \lesssim \delta_i\mcq.\]

\medskip \noindent \textbf{Case 3. ($\msr_{ij} =\sqrt{\frac{\delta_j}{\delta_i}}$):} We have
\[J_6 \lesssim \frac{\delta_i^{\frac32}}{\delta_j^{\frac12}} \int_{\{|y| \le \frac{r_0}{2\delta_i}\}} \frac{1}{(1+|y|^2)^{\frac32}|y|} \frac{dy}{[1+(\frac{\delta_i}{\delta_j}|y-y_{ij}|)^2]^{\frac12}}
\lesssim \frac{\delta_i^{\frac32}}{\delta_j^{\frac12}} \(1 + \int_2^{\frac{r_0}{\delta_i}} t^{-2}dt\) \lesssim \delta_i\mcq.\]
As a result, \eqref{eq:le363} holds.
\end{proof}

\subsection{Potential analysis}
Lemma \ref{le4.2} is based on the following linear theory. One can derive an analogous result for the case $u_0=0$, which is needed in the proof of Lemma \ref{le4.8}.
\begin{defn}\label{deb1}
Given two functions
\[V(x) = \(\frac{\delta_1^2}{\delta_1^2+d_{g_{\xi_1}}(x,\xi_1)^2}\) \bs{1}_{d_{g_{\xi_1}}(x,\xi_1) \le \sqrt{\delta_1}}
+ \delta_1 \(\frac{\delta_1}{\delta_1^2+d_{g_{\xi_1}}(x,\xi_1)^2}\)^{\frac{N-4}{2}} \bs{1}_{d_{g_{\xi_1}}(x,\xi_1) \ge \sqrt{\delta_1}}\]
and
\[W(x) = \(\frac{\delta_1}{\delta_1^2+d_{g_{\xi_1}}(x,\xi_1)^2}\)^2 \bs{1}_{d_{g_{\xi_1}}(x,\xi_1) \le \sqrt{\delta_1}}
+ \(\frac{\delta_1}{\delta_1^2+d_{g_{\xi_1}}(x,\xi_1)^2}\)^{\frac{N-2}{2}} \bs{1}_{d_{g_{\xi_1}}(x,\xi_1) \ge \sqrt{\delta_1}}\]
for $x \in M$ where $g_{\xi_1} = \Lambda_{\xi_1}^{4/(N-2)}g$, we define two weighted $L^{\infty}(M)$-norms $\|\cdot\|_*$ and $\|\cdot\|_{**}$ by
\[\|\trh_0\|_* = \sup_{x \in M} |\trh_0(x)|V(x)^{-1} \quad \text{and} \quad \|\tih\|_{**} = \sup_{x \in M} |\tih(x)|W(x)^{-1}.\]
\end{defn}
\begin{prop}\label{a5}
Assume that $N \ge 6$ and $u_0>0$ is non-degenerate. Let $\mcv_1$ and $\mcz^k_1$ for $k = 0,\ldots,N$ be the functions appearing in Subsection \ref{se4.1}.
Given any data $\tih$ with $\|\tih\|_{**} < \infty$, there exist unique $\trh_0 \in H^1(M)$ and $\tc_0,\tc_1,\ldots,\tc_N \in \R$ satisfying
\begin{equation}\label{aeq}
\begin{cases}
\displaystyle \mcl_g\trh_0 - (2^*-1)(u_0+\mcv_1)^{2^*-2}\trh_0 = \tih + \sum_{k=0}^N \tc_k \mcl_g\mcz^k_1 \quad \text{on } M,\\
\displaystyle \big\langle \trh_0,\mcz^k_1 \big\rangle_{H^1(M)} = 0 \quad \text{for } k=0,1,\ldots,N
\end{cases}
\end{equation}
as well as
\begin{equation}\label{rh}
\|\trh_0\|_* \lesssim \|\tih\|_{**} \quad \text{and} \quad \sum_{k=0}^N|\tc_k| \lesssim \delta_1^{\frac{N-2}{2}}\|\tih\|_{**}.
\end{equation}
\end{prop}
\begin{proof}
The existence of $\trh_0$ will follow from the standard method once \eqref{rh} is established.

\medskip
To show the first inequality of \eqref{rh}, we argue by contradiction. If it is false, there exist sequences $\{\trh_{0n}\}_{n \in \N}$, $\{\tih_n\}_{n \in \N}$,
$\{(\delta_{1n},\xi_{1n})\}_{n \in \N} \subset (0,\infty) \times M$, $\{\mcv_{\delta_{1n},\xi_{1n}}\}_{n \in \N}$, and $\{\tc_{kn}\}_{n \in \N} \subset \R$ for $k = 0,\ldots,N$ satisfying \eqref{aeq},
\[\|\trh_{0n}\|_*=1 \quad \text{for all } n \in \N, \quad \text{and} \quad \delta_{1n}+\|\tih_n\|_{**} \to 0 \quad \text{as } n \to \infty.\]
The above $*$- and $**$-norms are determined by the functions $V_n$ and $W_n$, which equal $V$ and $W$ in Definition \ref{deb1} with $(\delta_1,\xi_1)=(\delta_{1n},\xi_{1n})$, respectively.
Let $\mcv_{1n} = \mcv_{\delta_{1n},\xi_{1n}}$, $\mcz^0_{1n} = \delta_{1n}\frac{\pa\mcv_{1n}}{\pa\delta_{1n}}$, and $\mcz^k_{1n} = \delta_{1n}\frac{\pa\mcv_{1n}}{\pa\xi_{1n}^k}$ for $k = 1,\ldots,N$. For simplicity, we drop the subscript $n$ in Steps 1 and 2.

\medskip \noindent \textbf{Step 1.} We assert that
\begin{equation}\label{acj}
\sum_{k=0}^N |\tc_k| \lesssim \delta_1^{\frac{N-2}{2}} \|\tih\|_{**} + \delta_1^{\frac{N+2}{2}}|\log\delta_1| \|\trh_0\|_*.
\end{equation}

Indeed, by testing the first equation in \eqref{aeq} by $\mcz^l_1$ for $l = 0,\ldots,N$, we find
\[\bigg\langle \sum_{k=0}^N \tc_k\mcz^k_1,\mcz^l_1 \bigg\rangle_{H^1(M)} = \int_M \left[\mcl_g\trh_0 -(2^*-1)(u_0+\mcv_1)^{2^*-2}\trh_0\right]\mcz^l_1 dv_g - \int_M \tih\mcz^l_1 dv_g.\]
Direct computations show
\[\int_M \left|\mcl_g\mcz^l_1 - (2^*-1)\mcv_1^{2^*-2}\mcz^l_1\right|V dv_g \lesssim \delta_1^{\frac{N+2}{2}}.\]
Hence, using $|\mcz^l_1| \lesssim \mcu_1^{g_{\xi_1}}$, we obtain
\begin{multline*}
\left|\int_M \left[\mcl_g\trh_0 -(2^*-1)(u_0+\mcv_1)^{2^*-2}\trh_0\right]\mcz^l_1 dv_g\right| \\
\lesssim \int_M \left|u_0^{2^*-2}\mcz^l_1\trh_0\right| dv_g + \|\trh_0\|_* \int_M \left|\mcl_g\mcz^l_1 - (2^*-1)\mcv_1^{2^*-2}\mcz^l_1\right|V dv_g \lesssim \delta_1^{\frac{N+2}{2}}|\log\delta_1|\|\trh_0\|_*.
\end{multline*}
Furthermore,
\[\left|\int_M \tih\mcz^l_1 dv_g\right| \lesssim \delta_1^{\frac{N-2}{2}}\|\tih\|_{**}.\]
Since $\langle \mcz^k_1,\mcz^l_1 \rangle_{H^1(M)} = c\delta^{kl}+o(1)$ for some constant $c>0$, \eqref{acj} follows.

\medskip \noindent \textbf{Step 2.} We claim that
\begin{equation}\label{trh}
|\trh_0(x)|V(x)^{-1} \lesssim \frac{\delta_1^2}{\delta_1^2+d_{g_{\xi_1}}(x,\xi_1)^2} \log\(2+\frac{d_{g_{\xi_1}}(x,\xi_1)}{\delta_1}\) + \|\tih\|_{**} + \delta_1^2|\log\delta_1|\|\trh_0\|_*.
\end{equation}

Owing to the non-degeneracy of $u_0$, there exists the unique Green's function $G_0$ of the operator $\mcl_g-(2^*-1)u_0^{2^*-2}$. From \cite{R} with the boundedness of $u_0$, we know
\[|G_0(x,z)|\lesssim \frac{1}{d_g(x,z)^{N-2}} \lesssim \frac{1}{d_{g_{\xi_1}}(x,z)^{N-2}}\]
and so
\begin{equation}\label{eq:trh0}
|\trh_0(x)| \lesssim \int_M \frac{1}{d_{g_{\xi_1}}(x,z)^{N-2}} \left|\left[(u_0+\mcv_1)^{2^*-2}-u_0^{2^*-2}\right]\trh_0 + \tih + \sum_{k=0}^N \tc_k\mcl_g\mcz^k_1\right|(z) (dv_g)_z
\end{equation}
for $x \in M$. Let us analyze the right-hand side of \eqref{eq:trh0}. Making use of $\|\trh_0\|_*=1$ and Lemma \ref{aux}, we get
\begin{multline*}
\int_M \frac{1}{d_{g_{\xi_1}}(x,z)^{N-2}} \left|\left[(u_0+\mcv_1)^{2^*-2}-u_0^{2^*-2}\right]\trh_0\right|(z)(dv_g)_z \\
\lesssim \int_M\frac{1}{d_{g_{\xi_1}}(x,z)^{N-2}} \(\mcv_1^{2^*-2}V\)(z) (dv_g)_z \lesssim \frac{\delta_1^2}{\delta_1^2+d_{g_{\xi_1}}(x,\xi_1)^2}\log\(2+\frac{d_{g_{\xi_1}}(x,\xi_1)}{\delta_1}\)V(x)
\end{multline*}
and
\[\int_M \frac{1}{d_{g_{\xi_1}}(x,z)^{N-2}}|\tih(z)| (dv_g)_z \lesssim \|\tih\|_{**} V(x).\]
Also, since
\begin{equation}\label{poil}
\left|\(\mcl_g\mcz^k_1\)(z)\right| \lesssim \mcu_1^{g_{\xi_1}}(z)+\(\mcu_1^{g_{\xi_1}}\)^{2^*-1}(z) \quad \text{for } z \in M,
\end{equation}
we see from Lemma \ref{a1} and \eqref{acj} that
\[\int_M \frac{1}{d_{g_{\xi_1}}(x,z)^{N-2}}\left|\sum_{k=0}^N \tc_k\(\mcl_g\mcz^k_1\)(z)\right| (dv_g)_z \lesssim \(\|\tih\|_{**} + \delta_1^2|\log\delta_1|\|\trh_0\|_*\)V(x).\]
Thus \eqref{trh} holds.

\medskip \noindent \textbf{Step 3.} Since $\|\trh_{0n}\|_{*}=1$, there exists $x_n^* \in M$ such that
\[\left|\trh_{0n}(x_n^*)\right| V_n(x_n^*)^{-1} \ge \tfrac{1}{2} \quad \text{for all } n \in \N.\]
This together with \eqref{trh} guarantee that $d_{g_{\xi_{1n}}}(x_n^*,\xi_{1n}) \lesssim \delta_{1n}$.

\medskip \noindent \textbf{Step 4.} Given a cut-off function $\chi \in C^{\infty}_c([0,\infty))$ satisfying \eqref{chi0}, we define
\[\hrh_{0n}(y) = \chi(\delta_{1n}|y|) \trh_{0n}\big(\exp^{g_{\xi_{1n}}}_{\xi_{1n}}(\delta_{1n}y)\big) \quad \text{for } y \in \R^N.\]
If we write $y_n^* = \delta_{1n}^{-1}(\exp^{g_{\xi_{1n}}}_{\xi_{1n}})^{-1}(x_n^*)$, then $|y_n^*| \lesssim 1$,
\[|\hrh_{0n}(y)| \le \|\trh_{0n}\|_* \chi(\delta_{1n}|y|) V_n\big(\exp^{g_{\xi_{1n}}}_{\xi_{1n}}(\delta_{1n}y)\big) \lesssim 1,\]
and
\[|{\hrh}_{0n}(y_n^*)| \gtrsim \chi(\delta_{1n}|y_n^*|) V_n\big(\exp^{g_{\xi_{1n}}}_{\xi_{1n}}(\delta_{1n}y_n^*)\big) \gtrsim 1.\]
By standard elliptic regularity theory, there exist $\hrh_{0\infty} \in \dot{H}^1_{\text{loc}}(\R^N) \cap L^{\infty}(\R^N)$ and $y_{\infty}^* \in \R^N$ such that
\begin{equation}\label{urp}
\hrh_{0n} \to \hrh_{0\infty} \quad \text{in } C^{1,\eta}_{\loc}(\R^N) \quad \text{as } n \to \infty \quad \text{for some } \eta \in (0,1)
\end{equation}
and
\begin{equation}\label{unry}
y_n^* \to y_{\infty}^* \quad \text{as } n \to \infty \quad \text{where } |\hrh_{0\infty}(y_{\infty}^*)| \gtrsim 1 \text{ and } |y_{\infty}^*| \lesssim 1,
\end{equation}
along a subsequence. It follows from \eqref{dvg} and $\langle \trh_{0n},\mcz^k_{1n} \rangle_{H^1(M)}=0$ that
\begin{equation}\label{urzo}
\int_{\R^N} \nabla\hrh_{0\infty} \cdot \nabla Z^k = 0 \quad \text{for } k=0,1,\ldots,N.
\end{equation}
Besides, one can verify that
\[\delta_{1n}^2\ka_NR_g\big(\exp^{g_{\xi_{1n}}}_{\xi_{1n}}(\delta_{1n}\cdot)\big) \hrh_{0n} \to 0 \quad (\text{by } \eqref{urp}),\]
\[\delta_{1n}^2(u_0+\mcv_{1n})^{2^*-2}\big(\exp^{g_{\xi_{1n}}}_{\xi_{1n}}(\delta_{1n}\cdot)\big) \hrh_{0n}
\to U^{2^*-2}\hrh_{0\infty} \quad (\text{by \eqref{urp} and } u_0 \in L^{\infty}(M)),\]
\[\delta_{1n}^2\sum_{k=0}^N \tc_{kn} \(\mcl_g\mcz^k_{1n}\)\big(\exp^{g_{\xi_{1n}}}_{\xi_{1n}}(\delta_{1n}\cdot)\big) \to 0 \quad (\text{by } \eqref{acj} \text{ and } \eqref{poil}),\]
and
\[\delta_{1n}^2\tih_n\big(\exp^{g_{\xi_{1n}}}_{\xi_{1n}}(\delta_{1n}\cdot)\big) \to 0 \quad (\text{by } \|\tih_n\|_{**} \to 0)\]
uniformly in compact sets of $\R^N$ as $n\to \infty$. Therefore, passing the equation of $\hrh_{0n}$ to the limit yields
\[-\Delta\hrh_{0\infty} = (2^*-1)U^{2^*-2}\hrh_{0\infty} \quad \text{in } \R^N.\]
By \eqref{urzo}, we conclude that $\hrh_{0\infty} = 0$, which is impossible in view of \eqref{unry}. As a consequence, the first inequality of \eqref{rh} must hold. The second inequality of \eqref{rh} immediately follows from it and \eqref{acj}.
\end{proof}

\end{document}